\documentclass[12pt]{amsart} 
\usepackage[left=3cm,top=2.5cm,bottom=2.5cm,right=3cm]{geometry}
\usepackage{times}
\usepackage{amssymb, amsmath, amsthm}
\usepackage{mathtools}
\usepackage{graphicx,xspace}
\usepackage{epsfig}
\usepackage{enumitem}
\usepackage[usenames,dvipsnames]{xcolor}
\usepackage{tikz}
\usepackage{gensymb}
\usepackage[T1]{fontenc}
\usepackage[utf8]{inputenc}
\usepackage{bm}
\usepackage{subcaption}

\usepackage{mathrsfs}
\definecolor{green}{RGB}{0,127,0}
\definecolor{red}{RGB}{191,0,0}
\usepackage[colorlinks,cite color=red,link color=green,pagebackref=true]{hyperref}
\usepackage{todonotes}

\usetikzlibrary{matrix,arrows,calc}
\usetikzlibrary{decorations.markings}
\usetikzlibrary{patterns}
\usetikzlibrary{snakes}
\usetikzlibrary{shapes}

\usepackage[capitalize]{cleveref}
\usepackage{seqsplit}
\usepackage{booktabs}

\theoremstyle{plain}

\newtheorem{lemma}{Lemma}[section]
\newtheorem{theorem}[lemma]{Theorem}
\newtheorem{corollary}[lemma]{Corollary}
\newtheorem{proposition}[lemma]{Proposition}

\newtheorem{definition}[lemma]{Definition}
\newtheorem{definition-lemma}[lemma]{Definition-Lemma}

\newtheorem{assumption}[lemma]{Assumption}
\newtheorem{claim}[lemma]{Claim}

\theoremstyle{remark}
\newtheorem{remark}{Remark}
\newtheorem{example}{Example}

\definecolor{green}{RGB}{0,127,0}
\definecolor{red}{RGB}{191,0,0}
\newcommand{\xdownarrow}[1]{{\left\downarrow\vbox to #1{}\right.\kern-\nulldelimiterspace}}

\newcommand{\R}{\mathbb{R}}
\newcommand{\C}{\mathbb{C}}

\newcommand{\Sym}[1]{\mathfrak{S}_{#1}}

\newcommand{\Z}{\mathbb{Z}}

\newcommand{\QQ}{\mathbb{Q}}
\newcommand{\PP}{\mathbb{P}}

\newcommand{\SSS}{\mathbf{S}}
\newcommand{\SP}{\mathcal{S}\mathcal{P}}

\newcommand{\Poly}{\mathscr{P}}

\newcommand{\pp}{\mathbf{p}}
\newcommand{\qq}{\mathbf{q}}

\newcommand{\E}{\mathbb{E}}

\newcommand{\Y}{\mathbb{Y}}

\def\la{\lambda}

\def\ka{\kappa}
\def\a{{(\alpha)}}

\DeclareMathOperator{\SSym}{Sym}

\DeclareMathOperator{\conn}{conn}

\DeclareMathOperator{\supp}{supp}

\DeclareMathOperator{\Span}{Span}
\DeclareMathOperator{\Ch}{Ch}

\DeclareMathOperator{\Cov}{Cov}

\newcommand{\m}{\mathfrak{h}}

\DeclareMathOperator{\Symm}{Sym}

\def\vv{\mathbf{v}}

\newcommand{\ribbon}{\vec{\bm{\Gamma}}}

\newcommand{\bigone}{\prod_{i\ge 1}\bigg( i\cdot\frac{\alpha-1}{u}\bigg)^{\!|\SSS^i_{\tiny\rightarrow}(\ribbon)|}}

\newcommand{\bigtwo}{\bigg(\frac{i\alpha}{u^2}\bigg)^{\!|\mathbf{P}_i(\ribbon)|}\,v_i^{\,|\SSS_i(\ribbon)|}}

\newcommand{\al}{\alpha}
\newcommand{\ga}{\gamma}
\newcommand{\eps}{\epsilon}

\DeclareMathOperator{\AFP}{AFP}
\DeclareMathOperator{\eAFP}{eAFP}

\usepackage[none]{hyphenat}

\title[Jack-deformed random Young diagrams]{Universality of global asymptotics of Jack-deformed random Young diagrams at varying temperatures}

\author[C.~Cuenca]{Cesar Cuenca}
\address{
Department of Mathematics,
The Ohio State University,
231 West 18th Avenue,
Columbus, OH 43210, USA.
}
\email{cesar.a.cuenk@gmail.com}

\author[M.~Dołęga]{Maciej Dołęga}
\address{
Institute of Mathematics, 
Polish Academy of Sciences, 
ul. Śniadeckich 8, 
00-956 Warszawa, Poland.
}
\email{mdolega@impan.pl}

\author[A.~Moll]{Alexander Moll}
\address{
Department of Mathematics and Statistics,
Reed College,
3203 Southeast Woodstock Boulevard,
Portland, OR 97202, USA.
}
\email{amoll@reed.edu}

\thanks{This research was funded in whole or in part by {\it Narodowe Centrum Nauki}, grant 2021/42/E/ST1/00162. For the purpose of Open Access, the author has applied a CC-BY public copyright licence to any Author Accepted Manuscript (AAM) version arising from this submission.}

\begin{document}
\emergencystretch 3em
\begin{abstract}
This paper establishes universal formulas describing the global asymptotics of two distinct discrete versions of $\beta$-ensembles in the high, low and fixed temperature regimes. Our results affirmatively answer a question posed by the second author and Śniady.

We first introduce a special class of Jack measures on Young diagrams of arbitrary size, called the ``Jack--Thoma measures'', and prove the LLN and CLT in the three aforementioned limit regimes. In each case, we provide explicit formulas for polynomial observables of the limit shape and Gaussian fluctuations around the limit shape. These formulas have surprising positivity properties and are expressed as sums of weighted lattice paths.
Second, we show that the previous formulas are universal: they also describe the limit shape and Gaussian fluctuations for the model of random Young diagrams of a fixed size derived from Jack characters with the approximate factorization property. 
Finally, in stark contrast with continuous $\beta$-ensembles, we show that the limit shapes at high and low temperatures of our random Young diagrams are one-sided infinite staircase shapes.
For the Jack--Plancherel measure, we describe this shape explicitly by
relating its local minima with the zeroes of Bessel functions.
\end{abstract}

\maketitle

\section{Introduction and main results}

\subsection{Jack-deformed random Young diagrams as discrete $\beta$-ensembles}\label{sec:IntroOfIntro}

The \emph{Gaussian Unitary Ensemble} (GUE) is one of the most well-studied models of
random matrices. The joint density of eigenvalues $x_1 > \dots > x_N$ of the GUE is given by a formula of the form
\[ \frac{1}{Z_{\beta, N, V}}\prod_{1\leq i<j \leq N} |x_i-x_j|^\beta
  \prod_{i=1}^N\exp(-V(x_i)),\]
where $\beta = 2,\, V(x) = \frac{x^2}{2}$, and $Z_{\beta, N, V}$ is the normalization constant. When we choose an arbitrary $\beta>0$ and $V(x) = \frac{x^2}{2}$, we obtain the joint density of eigenvalues of the model
of random matrices known as the
\emph{Gaussian $\beta$-Ensemble} (G$\beta$E). For generic $V(x)$ and arbitrary $\beta>0$, this model is known in probability as a \textit{$\beta$-ensemble} and in statistical mechanics as a 1d \emph{log-gas system} of $N$ particles in a potential $V$ at inverse temperature $\beta>0$ \cite{Forrester2010}.

\indent While the asymptotic behavior as $N \to \infty$ of $\beta$-ensembles is well
understood for a fixed $\beta>0$ (see for instance~\cite{BenGuionnet1997,Johansson1998,DumitriuEdelman2002,ValkoVirag2009,RamirezRiderVirag2011}), the asymptotic analysis of $\beta$-ensembles in the \emph{high-temperature regime} in which $N \rightarrow \infty$ and $\beta \rightarrow 0$ simultaneously with $N \beta \sim c$ for some constant $c$ has recently gained a lot of attention in the literature \cite{AllezBouchaudGuionnet2012,Benaych-GeorgesPeche2015,DuyShirai2015,NakanoTrink2018,AkemannByun2019,Pakzad2020,HardyLambert2021,Forrester2022,Benaych-GeorgesCuencaGorin2022}.

It is well known by now that the remarkable
asymptotic properties of GUE as $N \to \infty$ are reflected in
the asymptotics of its discrete counterpart, 
given by the \emph{Plancherel measure} on Young
diagrams~\cite{VershikKerov1977,LoganShepp1977,BaikDeiftJohansson1999,BorodinOkounkovOlshanski2000,Okounkov2000,IvanovOlshanski2002},
see e.g.~\cite{BaikDeiftSuidan2016} and references therein. This
measure has a natural one-parameter deformation, introduced by Kerov~\cite{Kerov2000}, and called the \emph{Jack--Plancherel
  measure}. By definition, the Jack--Plancherel measure is the probability measure on the
set $\Y_d$ of Young diagrams of a fixed size $d\in\Z_{\ge 0}$ that depends on the
Jack deformation parameter $\al>0$ and is given by the formula
\begin{equation}\label{eq:Jack-Planch}
\PP^\a_{d}(\lambda) := \frac{d!\,\alpha^d}
{\prod_{(i, j)\in\lambda}\big(\alpha(\lambda_i-j)+(\lambda_j'-i)+1\big)
\big(\alpha(\lambda_i-j)+(\lambda_j'-i)+\alpha\big)}.
\end{equation}
It was expected (see e.g.~\cite{Okounkov2003}) that the
Jack--Plancherel measure would be the most 
natural discrete analogue of the Gaussian $\beta$-Ensemble;
the parameter $\alpha$ would need to be related to the inverse temperature
$\beta>0$ by the relation $\alpha = 2/\beta$. 
This prediction was confirmed in several ways, for instance, the second author and
F\'eray proved in~\cite{DolegaFeray2016} the LLN and CLT for random Jack--Plancherel
Young diagrams and conjectured that the edge scaling limit is given by the $\beta$-Tracy--Widom law; this conjecture was later proved by Guionnet and Huang~\cite{GuionnetHuang2019}.

\vspace{3pt}

In the last two decades, different models for discrete
$\beta$-ensembles which are analogs of the continuous
$\beta$-ensembles with a generic potential $V$, and simultaneously are generalizations of the 
Jack--Plancherel measure, have been proposed. 
It turns out that all these models are significantly different, and the methods developed for their analyses are different as well.

The first models are the \textit{Jack measures}; they are natural
$\alpha$-deformations of Schur measures~\cite{Okounkov2001} and, at
the same time, distinguished limits of the Macdonald
measures~\cite{BorodinCorwin2014}.
The study of Jack measures was first suggested by Borodin and Olshanski in~\cite{BorodinOlshanski2005} 
and were first studied in general by the third author in~\cite{Moll2023}.
Jack measures are probability measures on the infinite set of all Young diagrams $\Y$ defined by two homomorphisms $\rho_1,\rho_2\colon\SSym\to\C$ on the real algebra of symmetric functions $\SSym$. 
Denote the power-sum symmetric functions by $p_1,p_2,\dots$, and the Jack symmetric functions by
$J^\a_\lambda$. If
\begin{equation}\label{eq:NonNeg}
  \sum_{k \geq 1}\frac{\rho_1(p_{k}) \overline{\rho_2(p_{k})}}{k} < \infty,\qquad
  \rho_1\left(J_{\la}^\a\right)\overline{\rho_2\left(J_\la^\a\right)}\ge 0,\quad \forall\, \la\in\Y,
\end{equation}
then the Jack measure $M_{\rho_1, \rho_2}^\a$ is defined by the formula
\begin{equation}\label{eqn:Jack_measure_intro}
M_{\rho_1, \rho_2}^\a(\la) :=
\frac{\rho_1\left(J_\la^\a\right)
  \overline{\rho_2\left(J_\la^\a\right)}}
{j^\a_\la}\exp\left(-\sum_{k=1}^{\infty}{\frac{\rho_1(p_k) \overline{\rho_2(p_k)}}{k\alpha}} \right),
\quad \la\in\Y.
\end{equation}

A second model, introduced by the second author and Śniady~\cite{DolegaSniady2019},
defines probability measures on the finite set of Young diagrams
$\Y_d$ of a fixed size $d$. These measures
$\mathbb{P}^\a_{\chi_d}$ are defined by the decomposition of
a general choice of \emph{reducible Jack character} $\chi_d:\Y_d\to\C$ into a convex combination of \emph{irreducible Jack-characters} $\chi^\a_\la$:
\begin{equation}
\label{eq:proba-linearcomb-AIntro}
\chi_d = \sum_{\lambda\in\Y_d} \mathbb{P}^\a_{\chi_d}(\lambda) \ \chi^\a_\lambda.
\end{equation}
The irreducible Jack characters are defined
as the normalized coefficients of the Jack polynomial
$J^\a_\la$ expressed in the power-sum basis $p_\mu$, and they coincide
with the normalized irreducible characters of the symmetric groups in
the special case $\alpha=1$ (see~\cref{sec:Jack-deformed-1} for more details).
In particular, the case $\alpha=1$ recovers measures considered by
Biane~\cite{Biane2001} and Śniady~\cite{Sniady2006c}, in
the context of asymptotic representation theory of the symmetric groups.

Some other variants have been proposed, of which lastly we will mention 
the \textit{discrete $\beta$-ensembles} of Borodin--Gorin--Guionnet~\cite{BorodinGorinGuionnet2017}, 
defined as certain probability measures on the set $\Y^{(N)}$ of Young diagrams with at most $N$ rows.
These measures have attracted much interest recently, e.g.~\cite{DimitrovKnizel2019, GuionnetHuang2019}, because of their appearance, when $\beta=2$, as marginals of 2d statistical mechanical models, such as random tilings, last passage percolation, etc.

\vspace{3pt}

These discrete variants of $\beta$-ensembles all have different origins and are 
even supported on different sets, namely $\Y$, $\Y_d$ and $\Y^{(N)}$. 
Remarkably, the asymptotic behaviors of large Young diagrams sampled from each of them 
admit nice analyses and descriptions.
The third author proved in~\cite{Moll2023} a weak version of LLN and
CLT for certain specialization of Jack measures in high, low,
and fixed temperature regimes\footnote{He used a different
parametrization for $M_{\rho_1,\rho_2}$; in particular his regimes are not the same as the
ones in this paper and his assumptions are not satisfied in
case of Jack--Thoma measures from this paper; see \cite[Sec.~7.1.3]{Moll2023}.} under
certain regularity assumptions, provided that the specializations $\rho_1,\rho_2$ satisfy
the nonnegativity assumption~\eqref{eq:NonNeg}.
He considers especially the Jack measures $M_{\rho,\rho}$ with the same specializations $\rho_1 = \rho_2 =: \rho$, so that~\eqref{eq:NonNeg} is trivially satisfied.
Next, in the work of the second author and Śniady~\cite{DolegaSniady2019}, 
the authors prove a LLN and CLT for large random Young diagrams sampled by the
measures $\mathbb{P}^\a_{\chi_d}$, when $(\chi_d)_{d\ge 1}$ is a 
sequence of Jack characters satisfying the \textbf{Approximate Factorization
Property} or \textbf{AFP} (see~\cref{sec:Jack-deformed-1} for precise definitions). In the special case $\al=1$, it corresponds to the LLN
for random representations of large symmetric groups studied by
Biane~\cite{Biane2001} and the CLT for those representations proved by
Śniady~\cite{Sniady2006c}.
Finally, for the discrete $\beta$-ensembles~\cite{BorodinGorinGuionnet2017}, the authors prove a LLN and CLT when $\beta>0$ is fixed, by using the discrete loop equations.

Our understanding of the
relations between all these models was limited to the simple fact
that their intersection includes (not directly, though) the Jack--Plancherel
measure~\eqref{eq:Jack-Planch}. The problem of finding an intrinsic relation between these
models and their methods of analysis has been suggested to us on
several occasions by experts in the field, and the second author
proposed it already in his joint work with Śniady~\cite[Section 1.3]{DolegaSniady2019}. 
In this paper, we take a step forward in this direction, by establishing a deep connection between the first two models mentioned previously (Jack measures and Jack character measures with the AFP). Specifically, we find an infinite-parameter family of measures, called the \emph{Jack--Thoma measures}, and their depoissonized versions which fit into both settings.
We prove that the global asymptotic behavior of these measures is governed by explicit formulas expressed in terms of simple and elegant combinatorial objects known as \emph{lattice paths}.
Furthermore, we show that this special family of measures is dense enough so that, actually, all Jack character measures with the AFP display the same universal global-scale behavior.
The path to this surprising discovery was intricate, and we describe each step of our journey in detail, presenting a series of novel results along the way.

\subsection{Main result for Jack measures}

Our first main results describe the asymptotic behavior of certain
Jack measures via explicit combinatorial formulas. Let $\alpha> 0$, $\vv = (v_1, v_2,
  \cdots)\in\R^{\infty}$, and $u>0$. We say that the measure
\begin{equation}\label{CDM_def_intro}
M^\a_{u;\vv}(\la) := \exp\left(-\frac{u^2\cdot v_1}{\alpha}\right)\frac{J_\la^\a(u\cdot\vv)\cdot u^{|\la|}}
{j^\a_\la},\quad\forall\,\la\in\Y,
\end{equation}
is a \textbf{Jack--Thoma measure}; it is a special case of 
Jack measure $M_{\rho_1,\rho_2}$ from \eqref{eqn:Jack_measure_intro} with
the following choice of specializations: $\rho_1(p_k) = u\cdot v_k$ and $\rho_2(p_k) = u\cdot\delta_{1,k}$, for all $k\ge 1$.
The nonnegativity assumption in~\eqref{eq:NonNeg} is no longer automatically satisfied, as 
it was for the special case of $M_{\rho,\rho}$ studied in~\cite{Moll2023}.
For our Jack--Thoma measures, and for a fixed $\al>0$, the nonnegativity assumption 
reduces to the classification of \emph{$\al$-Jack-positive specializations of $\SSym$}.
This classification was completed by Kerov, Okounkov and Olshanski in~\cite{KerovOkounkovOlshanski1998},
where they extended the celebrated Thoma's classification of extremal characters
of the infinite symmetric group (see~\cite{Thoma1964,KerovOkounkovOlshanski1998,Matveev2019}).
Based on their result, $M^\a_{u;\vv}$ defines a probability measure on the set $\Y$ of partitions of
arbitrary size whenever $u,\vv$ are points parametrized by the points from the
Thoma cone 
\begin{multline*}
\Omega = \{ (a, b, c)\in \R^{\infty}\times \R^{\infty}\times \R \mid
a = (a_1, a_2, \cdots),\quad b = (b_1, b_2, \cdots),\\
a_1 \ge a_2 \ge \cdots \ge 0,\quad b_1 \ge b_2 \ge \cdots \ge 0,\quad
\sum_{i=1}^{\infty}{(a_i + b_i)} \le c \}
\end{multline*}
by setting $v_1 = u^{-1}\cdot c$ and $v_k = u^{-1}\sum_{i=1}^\infty a_i^k +
u^{-1} (-\alpha)^{1-k}\sum_{i=1}^\infty b_i^k$, 
for $k \geq 2$, and $u>0$ is arbitrary.
This shows that the class of Jack--Thoma measures is a rich infinite-parameter family.
Moreover, there exists a unique measure belonging to both classes
$M_{\rho,\rho}$ and $M^\a_{u; \vv}$, and it is the Jack--Plancherel measure.
Hence, next to the previously studied $M_{\rho,\rho}$, the Jack--Thoma measures 
form the most natural class of Jack measures that exist.

We are interested in the behavior of $M^\a_{u;\vv}$-distributed Young diagrams 
for a generic choice of the infinite family of parameters $\vv = (v_1, v_2,
  \cdots)\in\R^{\infty}$ as the ratio $\frac{\alpha}{u^2}$ tends
  to zero. This asymptotic regime describes the typical behavior of a large
$M^\a_{u; \vv}$-distributed random Young diagram. 
We study the asymptotic behavior of $M^\a_{u;\vv}$ in three different regimes: the 
  \textbf{fixed temperature regime} when $\frac{\alpha-1}{u} \to
  0$ (for instance when $\alpha>0$ is fixed and $u\to\infty$), the
  \textbf{high temperature regime} when $\frac{\alpha-1}{u}\to g>0$, 
and the \textbf{low temperature regime} when $\frac{\alpha-1}{u}\to g<0$.
  Therefore, in this subsection assume that $\alpha=\alpha(d)$, $u = u(d)$ are
  sequences of positive real numbers such that
\begin{equation}\label{eq:assumptions}
\lim_{d\to\infty} \frac{\alpha}{u^2} = 0,\qquad \lim_{d\to\infty} \frac{\alpha-1}{u}
 = g \in \R,\qquad M^\a_{u;\vv}(\lambda) \geq 0,\quad \forall\, \la\in\Y,\
 \forall\, d\in\Z_{\ge 0}.
\end{equation}

Note that in the case when $g\ne 0$ (the high and low temperature regimes), it is quite nontrivial to find sequences $\al, u$ that satisfy the conditions \eqref{eq:assumptions}, with the nonnegativity of $M^\a_{u;\vv}$ being especially troublesome.
When $g=0$ and $\al>0$ is fixed, we noted that \cite{KerovOkounkovOlshanski1998} furnishes examples where $v_k$ depends on $\al$; however, this becomes an issue if $\al$ varies, because the parameters $v_k$ must remain constant.
Nevertheless, we manage to find a large family of pairs $(g, \vv)\in(\R\setminus\{0\})\times\R^\infty$ for which these conditions are satisfied: see~\cref{def:admissible_pair} and \cref{prop:CDMPar2}.
Our construction is based on a new classification theorem for \emph{totally Jack-positive specializations}, see~\cref{thm_jack_spec}, i.e.~specializations which take nonnegative values on all Jack symmetric functions $J^\a_\la$, for all $\al>0$.

In our first main limit theorem, we prove that the profile $\omega_{\Lambda_{(\alpha;u\sqrt{v_1})}}$ of a rescaled random Young diagram $\Lambda_{(\alpha;u\sqrt{v_1})}$, obtained from a Young diagram $\lambda$ by replacing each box of size $1\times 1$ by a box of size $\frac{\alpha}{u\sqrt{v_1}}\times\frac{1}{u\sqrt{v_1}}$, concentrates around a deterministic
shape, and this holds true in all three limit regimes.
Let us remark that each box in the Young diagram was rescaled to have size $\frac{\alpha}{u\sqrt{v_1}}\times\frac{1}{u\sqrt{v_1}}$, because of the following two principles that we followed:

\vspace{2pt}

$\bullet$ Each box should be $\alpha$-anisotropic, meaning that the ratio between the width and height is $\alpha$.

$\bullet$ The expected area of a $M^\a_{u;\vv}$--distributed random Young diagram, after the rescaling, should be equal to $1$.\footnote{See \cref{subsec:LLNCLTCDM} and specifically \cref{exam:size_jack_thoma}.}

\vspace{2pt}

As a side comment, we mention that we deduce our first main theorem from a general statement about the convergence of random functions (see~\cref{theo:GeneralLLN}) that might be helpful in a wide variety of similar contexts, and is of independent interest.
Here is the abbreviated version of our first main result (and see~\cref{theo:LLNCDM} for the detailed version).

\begin{theorem}\label{theo:LLNCDM_intro}
Fix $g\in\R$ and $\vv = (v_1, v_2,\cdots)\in\R^\infty$. Suppose that $\alpha, u\in\R_{>0}$ satisfy~\eqref{eq:assumptions}.
Then there exists some deterministic function $\omega_{\Lambda_{g;\vv}}\colon\R\to\R_{\ge 0}$ with the property that 
\begin{equation}\label{eq:limitCDM'}
\lim_{d\rightarrow\infty}\omega_{\Lambda_{(\alpha;u\sqrt{v_1})}} = \omega_{\Lambda_{g;\vv}}.
\end{equation}
The limit in \eqref{eq:limitCDM'} means convergence with respect to the supremum norm, in probability.
Moreover, the limit shape $\omega_{\Lambda_{g;\vv}}$ is described by explicit formulas for various
observables that uniquely determine it (see~\eqref{eq:MomentLukasiewicz} below for the moments of the
associated transition measure, and~\eqref{s_ell} for the moments of the measure with the density $\frac{\omega_{\Lambda_{g;\vv}}(x)-|x|}{2}$).
These formulas are expressed in terms of a weighted enumeration of Łukasiewicz paths.
\end{theorem}

Informally, a \emph{Łukasiewicz path} is a directed lattice path
with steps $(1,k),\, k\in\Z_{\geq -1}$, starting at $(0,0)$, finishing at $(\ell,0)$, and staying in the first quadrant.
These objects are in bijection with non-crossing
set partitions, a combinatorial model intrinsically related to free
probability (see~\cref{rem:LukasiewiczNC} and the discussion
in~\cref{sub:IntroAFP}). 
Consider the following weighted generating
series for the set $\mathbf{L}_0(\ell)$ of Łukasiewicz paths of length $\ell$:

\begin{equation}\label{eq:MomentLukasiewicz}
M_\ell := v_1^{-\ell/2}\cdot\sum_{\Gamma\in\mathbf{L}_0(\ell)}
{\prod_{i\ge 0} (ig)^{\text{number of horizontal steps at height $i$}}\cdot
\prod_{j\ge 1} v_j^{\text{number of steps $(1,j)$}}}.
\end{equation}

We prove that the $\ell$-th moment of the transition measure
$\mu_{g;\vv}$ associated with the limit shape
$\omega_{\Lambda_{g;\vv}}$ via the
Markov--Krein correspondence is given by this series:

\[ \int_\R{x^\ell \mu_{g;\vv}(dx)} = M_\ell,\quad\text{for all $\ell\ge 1$}. \]

Note that in the special case when $g=0$ (e.g.~when $\al>0$ is fixed), $\vv=(1,0,0\dots)$ 
and $\vv=(1,c,c^2,\dots)$, our formula~\eqref{eq:MomentLukasiewicz} 
recovers the combinatorial formula for the moments of the \emph{semicircle
distribution} and the \emph{free Poisson distribution}, respectively.
These two special cases correspond to the LLN for the
Jack--Plancherel measure (\cite{LoganShepp1977,VershikKerov1977} for
$\alpha=1$, \cite{DolegaFeray2016} for general $\alpha>0$) and for the
Schur-Weyl measure (\cite{Biane2001} for
$\alpha=1$); the
high and low temperature regimes give one-parameter
deformations of these laws.
In particular, note that when specialized to $\al=1,\, g=0$, we provide a new proof of Biane's result, see~\cref{rem:LukasiewiczNC}.
Moreover, Jack--Thoma measures are distinguished mixtures of the measures on partitions of a fixed size already
studied by Kerov, Okounkov and Olshanski~\cite{KerovOkounkovOlshanski1998} in the
context of harmonic functions on the Young graph with Jack
multiplicities and a fixed $\alpha$. In light of our depoissonization analysis (see~\cref{sub:IntroAFP}), our~\cref{theo:LLNCDM_intro} describes the limit
shape for all these measures. We want to point out here that this limit
shape is described not only for an arbitrary and fixed parameter
$\alpha$ (which turns out to be the same as in the $\alpha=1$ case
studied by Okounkov in~\cite{Okounkov2000}), but it also covers the high and low temperature limits of the Kerov--Okounkov--Olshanski measures.

We conclude our analysis of Jack--Thoma measures by proving
the central limit theorem. We determine the
(non-centered) Gaussian process describing fluctuations around the limit shape
by giving explicit formulas for its mean shift and covariance matrix. 
For technical reasons, we require that the assumptions on the
growth of $\alpha$ and $u$ in~\eqref{eq:assumptions} are slightly
enhanced and replaced by the following conditions (see~\cref{rem:Assumpt}):
\begin{equation}\label{eq:assumptions_refined}
\frac{\al}{u^2} = \frac{1}{d} + o\left(\frac{1}{d}\right),\quad
\frac{\al-1}{u} = g+\frac{g'}{\sqrt{d}}+o\left(\frac{1}{\sqrt{d}}\right),\quad\text{ as }d\to\infty.
\end{equation}
This is the abbreviated version of our CLT (see~\cref{theo:CLTCDM} for the precise statement):

\begin{theorem}\label{theo:CLTCDM_intro}
Fix $g,g'\in\R$ and $\vv = (v_1, v_2,\cdots)\in\R^\infty$. 
Suppose that the sequences of positive reals $\alpha, u$ satisfy~\eqref{eq:assumptions_refined}.
Then there exists a Gaussian process $\Delta_{g,g';\vv}$ such that the random functions 
$\Delta^\a_{u;\vv} := 
\frac{u}{\sqrt{\al}}\bigg(\frac{\omega_{\Lambda_{(\al; u\sqrt{v_1})}} - \omega_{\Lambda_{g;\vv}}}{2}\bigg)$
converge to $\Delta_{g,g';\vv}$ in distribution for multidimensional
polynomial test functions. This process is uniquely determined
by equation~\eqref{MeanGaussian_2} that describes its mean vector,
and equation~\eqref{CovarianceGaussian_2} that describes its covariance matrix.
\end{theorem}

To prove these results, we develop a combinatorial approach to
Nazarov--Sklyanin operators, building on the work initiated
in~\cite{Moll2023}. For random partitions distributed according to the Jack--Thoma measures, 
and their transition measures, we can describe various important observables 
by elegant combinatorial models of Łukasiewicz ribbon paths that we introduce here 
(e.g.~see~\cref{theo:Expectations}). 
These combinatorial formulas allowed us to take limits in our regimes of interest and were the basis of our LLN and CLT. Finally, let us reiterate that we are able to construct, in~\cref{sec:JackPos}, a very rich family 
of specializations with certain positivity properties that lead to valid sequences of measures for which \cref{theo:LLNCDM_intro} and \cref{theo:CLTCDM_intro} are applicable in the high and low temperature regimes. These sequences are parametrized by $g\in\R\setminus\{0\}$ and $c\ge a_1 \ge a_2 \ge \cdots \ge 0$ with $\sum_{i=1}^{\infty}{a_i} \le c$ (see \cref{thm_jack_spec}
and \cref{prop:CDMPar2}). These ideas turn out to be a key link between Jack--Thoma measures and Jack character measures with the AFP, thus allowing us to establish that the asymptotic behavior of both models is governed by universal formulas expressed in terms of Łukasiewicz ribbon paths.

\subsection{Main result for measures with the approximate factorization property}\label{sub:IntroAFP}

\begin{figure}[tbp]	
\begin{center}
\includegraphics[width = \linewidth]{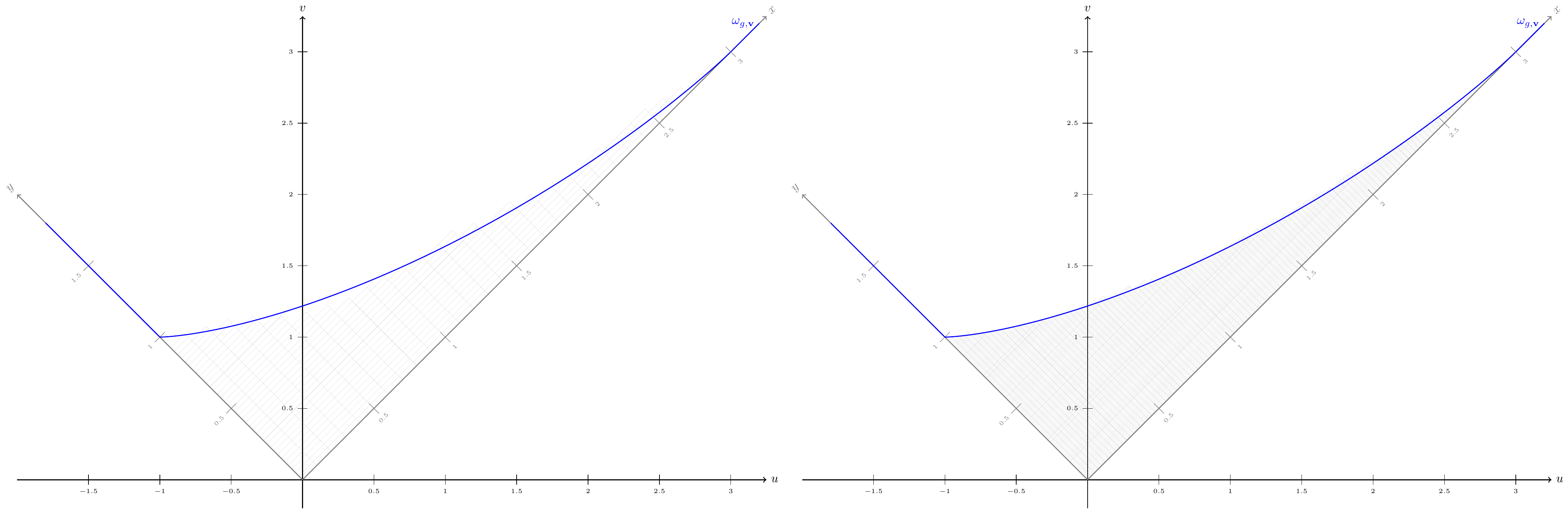}
\end{center}
\caption{The profiles of anisotropically-scaled Young diagrams of size $d=400$ and $d=6400$, sampled
by the Jack--Schur--Weyl measure from \cref{ex:Schur-Weyl} with
parameters $\alpha = 4$, $N = \sqrt{\alpha d}$, where each box is not of
size $1\times 1$, but of size $\frac{2}{\sqrt{d}}\times\frac{1}{2\sqrt{d}}$.
The blue curve is the limit shape in the fixed temperature regime (fixed $\al>0$), 
discovered by Biane~\cite{Biane2001}.}
\label{fig:SchurWeylFixed}
\end{figure}

Our second main result relates the previously studied Jack measures
with the model introduced by the second author and
Śniady in~\cite{DolegaSniady2019}. We recall that given a general reducible Jack character
$\chi_d\colon\Y_d\to\R$, we can associate a probability measure on the set of Young
diagrams $\Y_d$ of size $d$ by the formula
\eqref{eq:proba-linearcomb-AIntro}. The Jack characters
$\chi_d$ which are interesting from the probabilistic point of view are those which have the \emph{approximate factorization
property} (\textbf{AFP}). Roughly speaking, in the simplest case this property says
that the character $\chi_d(\pi_1 \cdot \pi_2)$, evaluated on the disjoint product of
partitions $\pi_1,\pi_2$, behaves asymptotically the same as the product of characters
$\chi_d(\pi_1)\cdot \chi_d(\pi_2)$, and 
similar formulas hold asymptotically for all higher order cumulants; 
see~\cref{sec:Jack-deformed-1} for the precise definitions. The second author and Śniady
proved in~\cite{DolegaSniady2019} the LLN and CLT for the Jack character measures 
with the AFP in all three previously discussed regimes
(see~\cref{theo:LLNDS} and \cref{theo:CLTDS}).

The primary tool used in~\cite{DolegaSniady2019} is the rich algebraic
structure of Jack characters. This approach, also known as the
\emph{dual approach}, stems from the successful viewpoint of Kerov and
Olshanski~\cite{KerovOlshanski1994}, who observed that the set of
characters of the symmetric groups, treated as functions on the set
of Young diagrams representing the associated irreducible
representation, is strictly related to the algebra of shifted
symmetric functions. The generalization of this approach to the
setting of Jack characters was proposed by Lassalle
in~\cite{Lassalle2008b,Lassalle2009}, who formulated challenging
conjectures on various positivity phenomena of Jack
characters. Despite significant
efforts~\cite{DolegaFeraySniady2014,DolegaFeray2016,Sniady2019,BenDali2023},
these conjectures remain wide open prior to our work.\footnote{The
  second author in a joint work with Ben Dali~\cite{BenDaliDolega2023}
  have found a proof of the Lassalle's positivity
  conjecture from~\cite{Lassalle2008b} that relies on the
  results obtained in this paper.} Consequently, finding formulas
for the moments of the limiting transition measure
$\mu_{\Lambda_\infty}$ from~\cref{theo:LLNDS} and for the
covariance of the Gaussian process $\Delta_\infty$
from~\cref{theo:CLTDS} turned out to be far beyond our reach. Our second main result solves this problem and we show that that the formulas for the moments and for the
covariance described in the previous section are in fact universal for both models.
\begin{theorem}[Abbreviated version of~\cref{theo:ShapeDS,theo:CovDS}]
  \label{theo:MainDSIntro}
	Assume that $\chi_d\colon\Y_d\to\R$ are reducible Jack characters
        such that the sequence $(\chi_d)_{d\ge 1}$ fulfills the
        AFP (enhanced AFP, respectively) given by \cref{def:approx-factorization-charactersA}. Let
        $\mu_{\Lambda_\infty}$ and $\Delta_\infty$ be the limiting
        transition measure and the limiting Gaussian process described
        by~\cref{theo:LLNDS} and~\cref{theo:CLTDS}, respectively. Then
        the moments $M_\ell := \int_\R{x^\ell \mu_{\Lambda_\infty}(dx)}$ are given by the
        polynomial formula~\eqref{eq:MomentLukasiewicz}, and the mean shift $\E[\langle \Delta_\infty,\, x^{\ell-2} \rangle]$ and the covariance $\Cov\left(\langle \Delta_{\infty}, x^{k-2}
          \rangle, \langle \Delta_{\infty}, x^{l-2} \rangle\right)$ are given by the explicit
        polynomial formulas~\eqref{eq:ExpDS} and~\eqref{eq:CovDS}, respectively, 
	that are obtained from the universal formulas~\eqref{MeanGaussian_2}
        and~\eqref{CovarianceGaussian_2} by incorporating the terms
        $v_k', v_{(k|l)}\,$ from the second order asymptotics of the character $\chi_d$ (given by~\eqref{eq:SecondCumu}).
      \end{theorem}
      
      Our method for proving~\cref{theo:MainDSIntro} synthesizes tools described in the previous section with algebraic tools developed in~\cite{DolegaFeray2016,DolegaSniady2019}. We show that the Jack--Thoma measures described earlier are, in fact, a Poissonized version of a specific class of examples of measures with the AFP (see~\cref{ex:CDM}). This allows us to deduce~\cref{theo:MainDSIntro} for this specific class. However, to understand the asymptotic behavior of arbitrary measures with the AFP, we need new ideas.
We use the theory developed in~\cite{DolegaFeray2016,DolegaSniady2019} to prove that the class of examples from \cref{ex:CDM} is large enough to ensure that the polynomial formulas described in \cref{theo:MainDSIntro} are universal for the entire class of measures with the AFP.

\subsection{Connections to Free Probability Theory and Algebraic Combinatorics}

Before delving into direct applications of our formulas, it is worth
mentioning some interesting ramifications of our work in related subjects.

Biane's fundamental work~\cite{Biane1998} established a connection between asymptotic representation theory of symmetric groups and free probability; we recover this fact because \eqref{eq:MomentLukasiewicz} implies that the $v_k$'s are the free cumulants of the limiting transition measure $\mu_{g;\vv}$ when $g=0$, see~\cref{rem:LukasiewiczNC}.
For a general $g$, the formula \eqref{eq:MomentLukasiewicz} defines a novel notion of 
\emph{$g$-deformed free cumulants} that we hope to investigate in future works. 
Other such deformations have appeared previously in
the literature, e.g.~see \cite{Benaych-GeorgesCuencaGorin2022,
  MergnyPotters2022, Nica1995}. Ours is more similar to the one in
\cite{Benaych-GeorgesCuencaGorin2022, MergnyPotters2022} because both
have, conjecturally, the property that the natural notion of $g$-deformed free
convolution is positivity preserving:\footnote{The $q$-deformed free
  convolution of Nica~\cite{Nica1995} is not positivity preserving,
  see \cite{Oravecz2005}.} if $\mu, \nu$ are probability measures, it
is conjectured that there exists a probability measure
$\mu\boxplus_g\nu$ whose $g$-cumulants are the sums of the
$g$-cumulants of $\mu$ and $\nu$. This is related to the conjecture of
Stanley~\cite{Stanley1989} on the positivity of
the structure constants for Jack polynomials. Furthermore, a different
positivity of the $g$-deformed free cumulant $v_k$ was proved by
Śniady~\cite{Sniady2019}: there exists a polynomial $K_k(-g,R_2,R_3,\dots)$ in $g$ and the 
free cumulants $R_2,R_3,\dots$ with positive integer coefficients,
such that $v_k = K_k(-g,R_2,R_3,\dots)$. This polynomial is the top-degree part of
the so-called \emph{Kerov polynomials}, whose positivity was conjectured by
Lassalle~\cite{Lassalle2009} and is still wide open (see
also~\cite{Feray2009,DolegaFeraySniady2010,DolegaFeray2016}). Śniady found a
combinatorial interpretation of polynomials $K_k$ in terms of graphs
embedded into non-oriented surfaces such that the exponent of $g$ is
computed by a recursive decomposition of these graphs. In practice,
controlling this exponent is almost impossible, except for some special
cases, see~\cite{DolegaFeraySniady2014,Dolega2017a,ChapuyDolega2022,BenDali2022a}. Comparing to Śniady's formula, our combinatorial
interpretation involves drastically simpler
combinatorial objects and, most importantly, the exponent of $g$ is an
explicit statistic of the underlying
path. Therefore we provide new combinatorial tools to attack
Lassalle's positivity conjectures~\cite{Lassalle2008b,Lassalle2009}
and they already turned out to be very successful --- the second
author in a joint work with Ben Dali~\cite{BenDaliDolega2023} used the results presented here
to prove integrality of the coefficients in Kerov's polynomials and to
fully prove Lassalle's conjecture from~\cite{Lassalle2008b}. We believe
that the tools developed in this paper will lead to ultimate
understanding of the Kerov's polynomials.

\subsection{Further applications}

\begin{figure}[tbp]	
\begin{center}
\includegraphics[width = \linewidth]{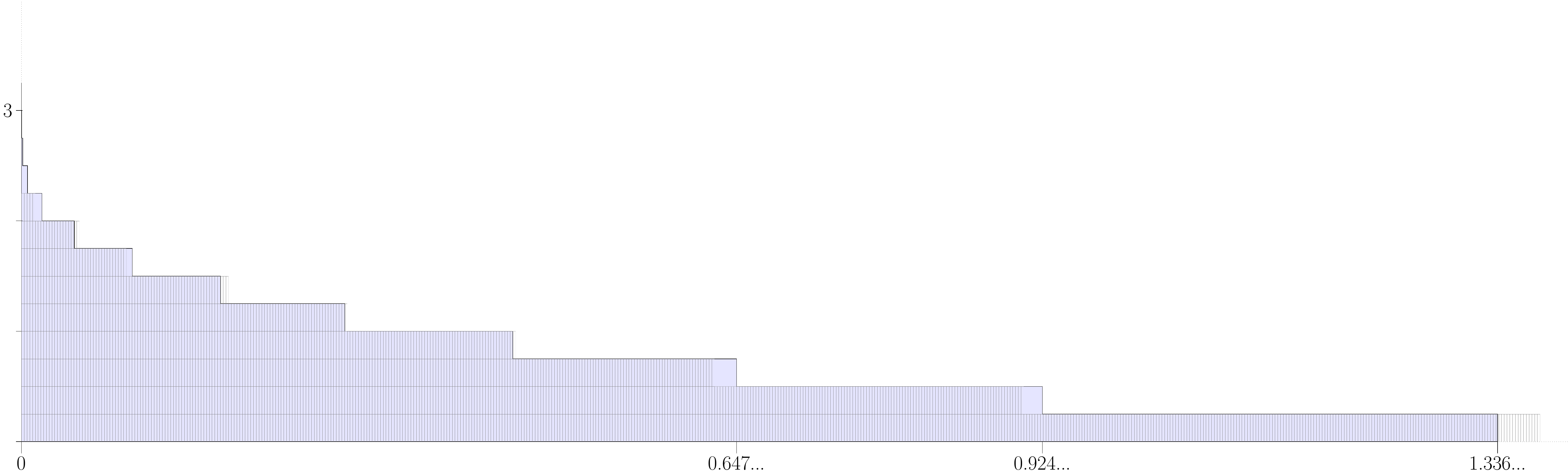}
\end{center}
\caption{The limit shape $\Lambda_{g;(1,0,0,\cdots)}$ with $g\!=\!-\frac{1}{4}$ (low temperature regime),
$\vv\!=\!(1,0,0,\cdots)$, drawn in the French notation and obtained from \cref{thm_asymptotic_positions_abbr}, is shaded in light blue.
Juxtaposed is an anisotropically-scaled Young diagram $\la^{(d)}$ of size $d = 1600$ sampled by the Jack--Plancherel measure $\PP^\a_d$ with $\al= \frac{4^2}{d}$, so each box (with gray-colored sides) is not of size $1\!\times\!1$, but of size $\frac{4}{d}\!\times\!\frac{1}{4}$.
The three smallest zeroes of the Bessel function $J_{-4z}(8)$ are equal to
$l^{(1/4)}_1 = -1.086...$, $l^{(1/4)}_2 = -0.424...$, and $l^{(1/4)}_3 = 0.102...$, therefore
\cref{thm_asymptotic_positions_abbr} implies 
$\lim_{d\to\infty}\frac{\la^{(d)}_1}{-g\cdot d} = 1.336...$, 
$\lim_{d\to\infty}\frac{\la^{(d)}_2}{-g\cdot d} = 0.924...$, 
and $\lim_{d\to\infty}\frac{\la^{(d)}_3}{-g\cdot d} = 0.647...$ in probability.
}
\label{fig:LowTempLim}
\end{figure}

\cref{theo:LLNCDM_intro} implies that the limit shape $\omega_{\Lambda_{g;\vv}}$ in the fixed temperature regime (i.e., when $g=0$) does not depend on the choice of $\alpha>0$, as long as it is fixed, and is the same as in the classical setting when $\alpha=1$. In this case, the limit shape is \emph{balanced}, meaning that $\omega_{\Lambda_{g;\vv}}(x) = |x|$, for $|x|$ large enough. However, in the high and low temperature regimes, i.e., when $g>0$ and $g<0$, respectively, the situation is different.

Drawing an analogy with recent studies on the high and low temperature
regimes of log-gases (see~\cref{sec:IntroOfIntro}), we apply our
combinatorial formulas to study the limit shape
$\omega_{\Lambda_{g;\vv}}$ (or equivalently, the associated transition
measure $\mu_{g;\vv}$) in the high and low temperature regimes. We show
that, unlike in the continuous counterpart of log-gases, there is a
phase transition in our models from the fixed
temperature regime to the high and low temperature
regimes. Specifically, we prove that for $g\neq 0$, the measure
$\mu_{g;\vv}$ is discrete with an infinite support bounded
only on one side and with a unique accumulation point at $+\infty$ or
$-\infty$, depending on the sign of $g$ (see
\cref{measures_discrete,main_thm_2}, and compare \cref{fig:LowTempLim} to \cref{fig:SchurWeylFixed}).

Finally, we relate $\mu_{g;\vv}$ to a specific unbounded Toeplitz-like
matrix, and we use it to describe the edge asymptotics of the random
partitions. By analogy with the celebrated
works of Baik, Borodin, Deift, Johanson, Okounkov, Olshanski~\cite{BaikDeiftJohansson1999,BorodinOkounkovOlshanski2000,Okounkov2000} on the
relation between the asymptotics of the largest rows of
Plancherel-sampled Young diagrams and the asymptotics of the largest
eigenvalues of GUE,
we relate the asymptotics of the largest rows of
Jack--Plancherel-sampled Young diagrams with the asymptotics of the largest eigenvalues
of finite deterministic Jacobi matrices. It allows us to prove the
following theorem that gives an explicit description of the limit
shape of Jack--Plancherel-sampled Young diagrams in the regimes of high and low temperature,
and expresses the asymptotics of the largest rows in terms of zeroes of the
Bessel function of the first kind, as a function of the label parameter (here we only state the abbreviated version for the case of low temperature, i.e.~for $g<0$; see~\cref{thm_asymptotic_positions} for the more general statement and \cref{fig:LowTempLim} to see the limit shape drawn with the help of this theorem).

\begin{theorem}\label{thm_asymptotic_positions_abbr}
Let  $(\la^{(d)}\in\Y_d)_{d\ge 1}$ be a sequence of random Young
diagrams distributed by the Jack--Plancherel measures.
Next, assume that $g<0$ and let
\[
\big\{l_1^{(-g)}<l_2^{(-g)}<\cdots\big\}  = \{ z\in\R\colon J_{z\cdot g^{-1}}(-2 g^{-1}) = 0 \}
\]
be the set of zeroes of the Bessel function of the first kind $J_{z\cdot g^{-1}}(-2 g^{-1})$.
In the regime $d,\alpha^{-1}\to\infty,\, \sqrt{\al d}\to -g^{-1}$, then
\[
\frac{\la^{(d)}_i}{-g\cdot d}\longrightarrow - l_i^{(-g)}-ig,\quad\text{in probability},
\]
for all $i=1,2,\cdots$, and moreover the limit shape is
\[
\omega_{\Lambda_{g;\vv}}(x) = 
\begin{cases}
x, &\text{ if } x > -l_1^{(-g)}-g,\\
-2\big( l_i^{(-g)}+ig \big) - x, &\text{ if } x \in \big[ -l_i^{(-g)},\ -l_i^{(-g)}-g \big], \text{ for some } i \geq 1,\\
x - 2ig, &\text{ if } x\in\big[ -l_{i+1}^{(-g)}-g,\ -l_i^{(-g)} \big], \text{ for some } i\ge 1.
\end{cases}
\]
\end{theorem}

\subsection{Organization of the paper}

In~\cref{sec:Jack-deformed-1}, we review the definition of Jack characters with the AFP, their associated probability measures, and we 
restate the LLN and CLT obtained in~\cite{DolegaSniady2019}.
We also give a large family of examples relating this model to the Jack--Thoma measures.
The latter model is introduced in~\cref{sec:Jack-deformed-2}, where we also prove our first main theorems: 
the LLN and CLT for the Jack--Thoma measures.
\Cref{sec:JackPos} is devoted to study Jack-positive specializations in different
regimes and relating our new family of measures with the Thoma cone.
In~\cref{sec:Duality}, we prove our second main theorem: first we show that
the conditional Jack--Thoma measure fits into the model of Jack-characters
with the AFP, and second we study the algebraic structure of Jack characters
in order to deduce formulas for the moments of the limiting measure and for the mean shift and covariance of the limiting Gaussian process.
The last~\cref{sec:LimitShape} is devoted to the study of the
limit shape, as well as its transition measure, arising from the high and low temperature limits of random partitions.
In particular, for the example of the Jack--Plancherel measure, we find the limits in probability of the largest rows/columns of the random partitions in terms of the zeroes of the Bessel function.

\section{First model of random Young diagrams: Jack characters and the asymptotic factorization property}\label{sec:Jack-deformed-1}

In this section and the next, we review two different constructions of 
probability measures on the set of Young diagrams that depend on the \emph{Jack parameter}, 
and that were introduced in \cite{DolegaSniady2019,Moll2023}.
It turns out that these two a priori different models are intimately related 
(see e.g.~\cref{prop:Depoissonization}), and this relation will be exploited 
later in \cref{sec:Duality}.
In this section, we review the construction from \cite{DolegaSniady2019} 
that yields probability measures on the finite set $\Y_d$ of Young diagrams of a fixed size $d$,
and is inspired by the relation between symmetric functions and the 
representation theory of symmetric groups; see \cite{Biane2001,Sniady2006c}. 
We also review the LLN and CLT as the size $d$ of the partitions tends to infinity.

\subsection{Asymptotics of the Jack parameter}

In this paper, we employ the Jack parameter, classically denoted by $\al>0$.
Later in this section, when we state our limit theorems, we will have instead a sequence $(\alpha(d))_{d\ge 1}$ of Jack parameters, i.e.~of positive real numbers. To simplify our notation, a general term of this sequence will be denoted by $\alpha = \alpha(d)$. We will assume that
\begin{equation}\label{eq:double-scaling-refined}
\frac{\sqrt{\al}-\frac{1}{\sqrt{\al}}}{\sqrt{d}} = g + \frac{g'}{\sqrt{d}} + o\left(d^{-\frac{1}{2}}\right),\ \text{ as }d\to\infty,
\end{equation}
for some constants $g,g'\in\R$.

The regime when $g>0$ and $\al,\, d$ are of the same order, i.e.~$\al = \Theta(d)$ (resp.~$g<0$ and $\al = \Theta(d^{-1})$) is called the \emph{high temperature regime} (resp.~\emph{low temperature regime}). Likewise, the regime when $g=0$, $\,\al=\Theta(1)$ and $g'=\sqrt{\al}-\frac{1}{\sqrt{\al}}$ is called the \emph{fixed temperature regime}.\footnote{The parameter $\al$ should be regarded as $2/\beta$, where $\beta>0$ is the \emph{inverse temperature} parameter from log-gas systems or beta ensembles, e.g.~see \cite{Forrester2022}.}

\subsection{Conditional cumulants}\label{sec:afp}

Let $\Y_d$ denote the set of partitions (equivalently, Young diagrams) 
of $d\in\Z_{\ge 0}$, and let $\Y := \bigsqcup_{d\geq 0}{\Y_d}$ 
be the set of all partitions. For any $\mu,\mu'\in\Y$, we define their \emph{product} 
$\mu\cdot\mu'$ as their union,
for example $(4,2,1) \cdot (5,3,1)=(5,4,3,2,1,1)$. In this way,
the set of all partitions $\Y$
becomes a unital semigroup with the unit being the empty partition $\emptyset$. The
corresponding semigroup algebra of formal linear combinations
of partitions is denoted by $\R[\Y]$.
Any function $\chi\colon\Y\to\R$ can be canonically extended to a linear map $\chi\colon\R[\Y]\to\R$, such that  $\chi(1)=1$, which will be denoted by the same symbol.

For any partitions $\mu^1,\dots,\mu^n\in\Y$, we define their \emph{conditional cumulant} 
(or simply \emph{cumulant}, for simplicity) with respect to the 
function $\chi\colon\Y\to\R$ as a coefficient in the expansion of the following
formal power series in $t_1,\dots,t_n$:
\begin{equation}
\label{eq:log-laplace}
\kappa^{\chi}_n(\mu^1,\dots,\mu^n):= 
\left. \frac{\partial^n}{\partial t_1 \cdots \partial t_n}
\log \chi\left( e^{t_1 \mu^1+\cdots+t_n \mu^n} \right) 
\right|_{t_1=\cdots=t_n=0}.
\end{equation}

An important case for us will be when each partition 
consist of just one part $\mu^i=(k_i)$; then we will use a simplified notation:
\[ \kappa^{\chi}_n(k_1,\dots,k_n) := \kappa^{\chi}_n\big( (k_1),\dots,(k_n) \big).
\]
For example,
\begin{align*}
\kappa^{\chi}_1(k) &= \chi(k), \\   
\kappa^{\chi}_2(k,l) &= \chi(k,l)- \chi(k) \chi(l).
\end{align*}

\subsection{Irreducible Jack characters}

Let $\alpha>0$ be a positive real.
Let $J^\a_\la$, $\la\in\Y$, denote the integral versions of Jack
symmetric functions; see~\cite[Ch.~VI.10]{Macdonald1995}. 
For any $d\in\Z_{\ge 0}$ and $\lambda \in \Y_d$, expand 
$J^\a_\la$ into power-sum symmetric functions:
\begin{equation} 
\label{eq:definition-theta-A}
J^\a_\la = \sum_{\mu\in\Y_d} \theta^\a_\mu(\lambda)\ p_\mu.
\end{equation}
For a partition $\mu$ with $m_i(\mu)$ parts equal to $i$, define the
standard numerical factor
\begin{equation}
\label{eq:zlambda}
z_\mu:= \prod_{i\geq 1} m_i(\mu)!\ i^{m_i(\mu)}.
\end{equation}
We renormalize the coefficients $\theta^\a_\mu(\lambda)$ by
setting
\begin{equation}
\label{eq:character-Jack-unnormalized-zmiana}
\chi^\a_\lambda(\mu) :=  \alpha^{-\frac{\|\mu\|}{2}}\ \frac{z_\mu}{d!}\ \theta^\a_\mu(\lambda),
\end{equation}
where
\[ \|\mu\|:=|\mu|-\ell(\mu),\quad |\mu| := \sum_{i\geq 0}\mu_i, \quad
  \ell(\mu) := \sum_{i \geq 1}m_i(\mu), \]
and call them \emph{irreducible Jack characters}. This terminology
is justified by the fact that in the special case $\alpha=1$, the
irreducible Jack character $\chi^{(\alpha=1)}_\lambda(\mu)$ coincides with
the normalized character of the irreducible representation of the
symmetric group $\Sym{d}$ that is associated to $\la$.

We regard each irreducible Jack character $\chi^\a_\la$, $\la\in\Y_d$, 
as a function $\Y_d\to\R$. It is known that $\theta^\a_{(1^d)}(\la) = 1$ 
(see \cite[(10.29)]{Macdonald1995}), therefore
\begin{equation}\label{norm_jack}
\chi^\a_\lambda(1^d) = 1, \text{ for any }\la\in\Y_d.
\end{equation}

Moreover we have the following duality between irreducible Jack
characters:

\begin{equation}
  \label{eq:duality}
\chi_\lambda^\a(\mu) = (-1)^{\|\mu\|}\chi_{\lambda'}^{(1/\alpha)}(\mu),
\end{equation}
where $\lambda'$ is the conjugate of the partition $\lambda$.
For this, recall the automorphism $\omega_{1/\alpha}$ of the algebra
of symmetric functions $\Symm$ defined by $\omega_{1/\alpha}(p_r) = (-1)^{r-1}\alpha^{-1} p_r$, for all $r\in\Z_{\ge 1}$,
see \cite[Ch.~VI, (10.6)]{Macdonald1995}. It satisfies, see
\cite[Ch.~VI, (10.24)]{Macdonald1995},
\begin{equation}
\label{eq:OmegaDuality}
\omega_{1/\alpha}(J^{(1/\alpha)}_{\lambda}) = \alpha^{-|\lambda|} J^\a_{\lambda'}
\quad \forall\, \lambda\in\Y.
\end{equation}
This directly implies that
\[ \theta^\a_\mu(\lambda) =(-\alpha)^{\|\mu\|}\cdot\theta^{(1/\alpha)}_\mu(\lambda'),\]
  and plugging it into definition of the irreducible Jack character
  \eqref{eq:character-Jack-unnormalized-zmiana} yields the result.

  \begin{remark}
    In the special case $\alpha=1$, this identity expressed in the
    language of representation theory states that the tensor product
    of the irreducible
    representation of $\Sym{d}$ associated with partition $\lambda$ with the sign representation corresponds to the
    irreducible representation associated with the conjugate partition $\lambda'$.
    \end{remark}

\subsection{Characters with the approximate factorization property and
  associated probability measures}
\label{subsec:CharactersAFP}

A general function $\chi_d\colon\Y_d\to\R$ will be called a
\emph{reducible Jack character} (or simply \emph{Jack character}, for short)
if it is a convex combination of the irreducible Jack characters. 
In particular, this definition and \eqref{norm_jack} imply that $\chi_d(1^d)=1$.

\begin{definition}[First model of random Young diagrams]\label{def:first}
Given a Jack character $\chi_d:\Y_d\to\R$, its associated 
probability measure $\mathbb{P}^\a_{\chi_d}$ on the set
of Young diagrams $\Y_d$ is defined by the expansion:
\begin{equation}
\label{eq:proba-linearcomb-A}
\chi_d = \sum_{\lambda\in\Y_d} \mathbb{P}^\a_{\chi_d}(\lambda) \ \chi^\a_\la.
\end{equation}
\end{definition}

\begin{remark}
It was proved in~\cite{DolegaFeray2016} that 
$\{\chi^\a_\lambda(\cdot)\}_{\lambda \in \Y_d}$ is a linear basis
of the space of functions $\Y_d\to\R$, so
the decomposition \eqref{eq:proba-linearcomb-A} uniquely determines the values $\mathbb{P}^\a_{\chi_d}(\lambda)$.
\end{remark}

\begin{remark}
If $\chi_d:\Y_d\to\C$ is any function with $\chi_d(1^d)=1$, then \eqref{eq:proba-linearcomb-A} defines a complex measure $\mathbb{P}^\a_{\chi_d}$ on $\Y_d$ with total mass $1$.
If $\chi_d$ takes real values, then $\mathbb{P}^\a_{\chi_d}$ is a signed measure.
\end{remark}

\begin{example}
The representation theory of the infinite symmetric group $S(\infty)$ furnishes a large class of probability measures as in \cref{def:first}, in the case $\al=1$.
In fact, let $\chi\colon S(\infty)\to\C$ be a character of $S(\infty)$, i.e.~a central, positive definite function that takes value $1$ at the identity.
Denote the restriction of $\chi$ to $S(d)\subset S(\infty)$ by $\chi\big|_{S(d)}\colon S(d)\to\C$, and let $\chi_d\colon\Y_d\to\C$ be defined by
\[
\chi_d(\mu) := \chi\big|_{S(d)}\big(\textrm{any permutation of cycle-type }\mu\big),\quad\textrm{ for all }\mu\in\Y_d.
\]
Then $\chi_d$ takes real values and, moreover, $\chi_d$ is a Jack character for $\al=1$.
In particular, $\chi_d$ leads to a valid probability measure $\mathbb{P}^{(\al=1)}_{\chi_d}$ by the prescription of \cref{def:first};
for a proof, see \cite[Prop.~1.6  and Prop.~3.5]{BorodinOlshanski2017}.
As $d\to\infty$, these measures converge, in certain sense, to the \emph{spectral measure} of $\chi$ on the Thoma simplex; see~\cite[Sec.~7]{BorodinOlshanski2017} for the precise statements and for further ideas in this direction.
\end{example}

We need to extend the definition of cumulants (recall~\cref{eq:log-laplace}) to make sense of $\kappa^{\chi_d}_n (k_1,\dots,k_n)$ for Jack characters $\chi_d\colon\Y_d\to\R$. 
First, extend the domain of $\chi_d$ to the set 
\[ \Y_{\leq d}:=\bigsqcup_{0\leq k\leq d} \Y_{k} \] 
of partitions of smaller numbers by adding an appropriate number of parts, each equal to $1$:
\begin{equation}
\label{eq:extended-character}
\chi_d(\pi):= \chi_d( \pi,1^{d-|\pi|}) \qquad \text{for } |\pi|\leq d.   
\end{equation}
In this way, the cumulant $\kappa^{\chi_d}_n (k_1,\dots,k_n)$ 
is well-defined if $d\geq k_1+\cdots+k_n$ is large enough.
Note that $\chi_d(\emptyset)=1$.

\begin{definition}\label{def:approx-factorization-charactersA}
Assume that for each integer $d\geq 1$, we are given a function $\chi_d\colon\Y_d\to\R$.

We say that the sequence $(\chi_d)_{d\ge 1}$ has the \textbf{approximate factorization property} (the \textbf{AFP} for short) \cite{Sniady2006c,DolegaSniady2019} if the following two conditions are satisfied:

\smallskip

	(i) For each integer $n\geq 1$ and all integers $k_1,\dots,k_n\geq 2$, the limit
	\begin{equation}
	\label{eq:aprox-fact-property}
	\lim_{d\to\infty} \kappa^{\chi_d}_n (k_1,\dots,k_n)\ d^{\frac{k_1+\cdots+k_n+n-2}{2}} 
	\end{equation}
	exists and is finite.
	
\smallskip

(ii) 	The following quantity is finite
\begin{equation}\label{eq:what-is-m}
	\sup_{k\geq 2} \frac{\sqrt[k]{|v_k|}}{k} < \infty,
      \end{equation}
      where $v_k$ are the limits appearing in
      \eqref{eq:aprox-fact-property} for $n=1$:
      \begin{align}
\lim_{d\to\infty}\kappa^{\chi_d}_1(k) \ d^{\frac{k-1}{2}} &=  \lim_{d\to\infty}\chi_d(k) \ d^{\frac{k-1}{2}}  =
v_k. \label{eq:refined-asymptotics-characters}
\end{align}

\smallskip

Further, we say that  $(\chi_d)_{d\ge 1}$ has the \textbf{enhanced AFP} if, in addition to (i) and (ii), 
the condition \eqref{eq:refined-asymptotics-characters} is strengthened as follows: for each $k\ge 2$, 
there exist some real constants $v_k'$ such that 
\begin{equation}\label{eq:refined-asymptotics-characters-2}
\kappa^{\chi_d}_1(k) \ d^{\frac{k-1}{2}} =  \chi_d(k) \ d^{\frac{k-1}{2}} =
v_k + v_k'\,d^{-\frac{1}{2}} + o\left( d^{-\frac{1}{2}}\right),\qquad \text{for }d\to\infty.
\end{equation}
\end{definition}

In the following, for each $k,l\geq 2$, we define the real constants
$v_{(k|l)}=v_{(l|k)}$ as the limits appearing in \eqref{eq:aprox-fact-property} for $n=2$:
\begin{align}
\lim_{d\to\infty}\kappa^{\chi_d}_2(k,l) \ d^{\frac{k+l}{2}} &=  v_{(k|l)}. \label{eq:SecondCumu}
\end{align}

\begin{remark}
Our definitions of the AFP and enhanced AFP differ from those in \cite{DolegaSniady2019}.
We also point out that condition \eqref{eq:what-is-m} might be relaxed in such a way that \cref{theo:LLNDS,theo:CLTDS} remain satisfied; see \cite[(1.13)]{DolegaSniady2019} for an example of such a weaker condition in some cases.
\end{remark}

\begin{proposition}
  \label{prop:DualityAFP}
Let $\vv = (v_1, v_2, \cdots)\in\R^{\infty}$ and $\pm\vv := (v_1, -v_2,
v_3, -v_4, \cdots)$. For $g,g' \in \R$ define
$\mathcal{F}^{\eAFP}_{g,g',\vv} \subset \mathcal{F}^{\AFP}_{g,g',\vv}$
as the sets of the sequences of Jack characters with the enhanced AFP, and the AFP,
respectively, with the parameters $g,g',\vv$ described
in~\eqref{eq:double-scaling-refined} and
\cref{def:approx-factorization-charactersA}. Then, we have the
bijective map
\[\mathcal{F}^{\AFP}_{g,g',\vv} \longrightarrow
\mathcal{F}^{\AFP}_{-g,-g',\pm\vv}, \quad \left(\chi^{(\alpha)}_d\right)_{d\ge 1} \mapsto
\left(\tilde{\chi}^{(\tilde{\alpha})}_d\right)_{d\ge 1}, \]
which restricts to the bijective map $\mathcal{F}^{\eAFP}_{g,g',\vv} \longrightarrow
\mathcal{F}^{\eAFP}_{-g,-g',\pm\vv}$ and is given by
\begin{equation}\label{eq:mapDuality}
\tilde{\alpha} := \alpha^{-1},\quad \tilde{\chi}^{(\tilde{\alpha})}_d(\mu) := (-1)^{\|\mu\|}\chi^{(\alpha)}_d(\mu).
\end{equation}
In particular,
\begin{equation}\label{MeasuresDuality}
\mathbb{P}^{(\tilde{\alpha})}_{\tilde{\chi}_d}(\lambda) =
\mathbb{P}^{(\alpha)}_{\chi_d}(\lambda').
\end{equation}
\end{proposition}

\begin{proof}
Most claims follow directly from our definition for $\tilde{\chi}^{(\tilde{\alpha})}_d$.
In fact, the only non-trivial statement is that
$\tilde{\chi}^{(\tilde{\alpha})}_d$ are Jack characters, but this is assured by \eqref{MeasuresDuality}.
Thus, it remains only to prove~\eqref{MeasuresDuality}:
\[ \tilde{\chi}^{(\tilde{\alpha})}_d(\mu) =
  (-1)^{\|\mu\|}\chi^{(\alpha)}_d(\mu) =
  (-1)^{\|\mu\|}\sum_{\la\in\Y_d} \mathbb{P}^{(\alpha)}_{\chi_d}(\la)
  \ \chi^{(\alpha)}_\la(\mu) = \sum_{\la\in\Y_d}
  \mathbb{P}^{(\alpha)}_{\chi_d}(\la') \ \chi^{(\tilde{\alpha})}_{\la}(\mu),\]
where the last equality follows from \eqref{eq:duality}.
  \end{proof}

\subsection{The LLN and CLT for characters with the approximate factorization property}\label{sec:thms_DS}

The main results of~\cite{DolegaSniady2019} are the Law of Large Numbers
(LLN) and Central Limit Theorem (CLT) for large random Young
diagrams. In order to state them, we need to work with rescaled Young diagrams. For two
positive numbers $w, h \in \R_{>0}$ and a given Young diagram $\lambda$,
we denote by $T_{w,h}\lambda$ the \emph{anisotropic}
diagram obtained from $\lambda$ by replacing each box of size
$1\times 1$ by a box of size $w\times h$, see~\cref{fig:french-fries}.
We will be mostly interested in the following scaling:
\begin{equation*}
\Lambda^\a := T_{\sqrt{\frac{\al}{|\la|}},\sqrt{\frac{1}{\al|\la|}}}\la,
\end{equation*}
so that the rescaled partition $\Lambda^\a$ has $\alpha$-anisotropic boxes (the ratio between the width and height of a box should be $\alpha$) and the area of $\Lambda^\a$ is equal to $1$.
In fact, since we will use this notation only when it is clear that the size of the random partition $\la$ is a fixed number $d$, we shall prefer the following more detailed notation:
\begin{equation}\label{eq:Lambda}
\Lambda^\a_d := T_{\sqrt{\frac{\al}{d}},\sqrt{\frac{1}{\al d}}}\la,\quad\textrm{if }|\la|=d.
\end{equation}

In order to state asymptotic results, it is convenient to draw Young diagrams using the \emph{Russian convention}, which is given by changing the usual Cartesian coordinates $(x,y)$ into a new coordinate system $(u,v)$ given by
\[ u = x-y, \qquad v = x+y. \]
The boundary of an anisotropic Young diagram $T_{w,h}\lambda$ 
drawn in the Russian convention
is the graph of a function $\omega_{T_{w,h}\lambda}\colon\R\to\R_{\ge 0}$, 
which will be called the \emph{profile} of $T_{w,h}\lambda$, see~\cref{fig:french-fries}.

\begin{figure}[tbp]	
  \begin{center}
    \includegraphics{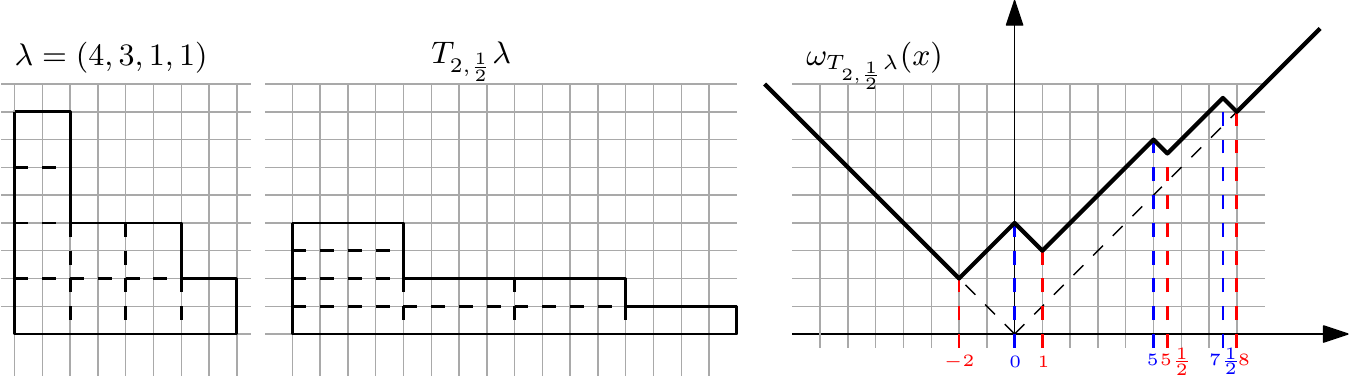}
	\end{center}
\caption{
		Young diagram $\lambda = (4,3,1,1)$ (left) and
                the anisotropic Young diagram
                $T_{2,\frac{1}{2}}\lambda$ (center)
		are in the French convention. 
		The solid line in the rightmost picture is the \emph{profile} $\omega_{T_{2,\frac{1}{2}}\lambda}$ of 
                $T_{2,\frac{1}{2}}\lambda$, obtained by
                switching from the French to the Russian convention. 
		The local minima $x_1,x_2,x_3,x_4$ and maxima $y_1,y_2,y_3$ of the profile
                are indicated in red and blue, respectively.}
\label{fig:french-fries}
\end{figure}

\begin{theorem}[Law of Large Numbers \cite{DolegaSniady2019}]	\label{theo:LLNDS}
	Assume that $\chi_d\colon\Y_d\to\R$ are Jack characters
	such that $(\chi_d)_{d\ge 1}$ fulfills the AFP. 
	Let $\lambda_d$ be a random Young diagram with $d$ boxes
	distributed according to $\mathbb{P}^\a_{\chi_d}$, and let 
	$\Lambda^\a_d$ be the associated rescaled diagram.
	
	Then there exists some deterministic function $\omega_{\Lambda_\infty}\colon\R\to\R_{\ge 0}$ 
	with the property that 
	\begin{equation}\label{eq:limit}
	\lim_{d\to\infty} \omega_{\Lambda^\a_d} = \omega_{\Lambda_\infty},
	\end{equation}
	where the convergence holds true with respect to the supremum norm, in probability.
\end{theorem}

\begin{theorem}[Central Limit Theorem \cite{DolegaSniady2019}]\label{theo:CLTDS}
Assume that $\chi_d\colon\Y_d\to\R$ are Jack characters such that $(\chi_d)_{d\ge 1}$ fulfills the enhanced AFP. Consider the sequence of random functions
\[{\Delta}^\a_d := \sqrt{d}\cdot\frac{\omega_{\Lambda^\a_d}-\omega_{\Lambda_\infty}}{2} \]
and the observables
\[ \langle\Delta^\a_d,\, x^\ell \rangle := \int_{\R} x^\ell\,\Delta^\a_{d}(x) dx,\quad
\ell\in\Z_{\ge 1}. \]

Then the joint distribution of the random variables $\langle\Delta^\a_d,\, x^\ell \rangle$, 
for $\ell=1,2,\cdots$, converges weakly to a (non-centered) Gaussian distribution.

In other terms, the sequence of random functions ${\Delta}^\a_{d}$ converges to some (non-centered) Gaussian random vector $\Delta_\infty$ valued in the space $(\R[x])'$ of distributions, the dual
space to polynomials, and the convergence holds in distribution as $d\to\infty$.
\end{theorem}

Informally, the LLN and CLT above describe the first and second order asymptotics of the profile $\omega_{\Lambda^\a_d}$ of a large random anisotropic partition:
\[ \omega_{\Lambda^\a_d} = \omega_{\Lambda_\infty} + \frac{2}{\sqrt{d}}\cdot\Delta_\infty + o\left( \frac{1}{\sqrt{d}} \right)\!,\text{ as $d\to\infty$}. \]

\subsection{Examples of characters with the AFP and the corresponding probability measures}\label{sec:examples}

\subsubsection{Jack--Plancherel measure}\label{ex:Plancherel}
Let
\begin{equation}\label{eq:regular}
\chi_d(\mu) := \begin{cases} 1, & \text{if $\mu=(1^d)$},\\
0, & \text{otherwise}, \end{cases}    
\end{equation}
be the normalized character of the regular representation of the symmetric group $\Sym{d}$.
For $\al=1$, the associated probability measure $\mathbb{P}^{(\al=1)}_{\chi_d}$ 
is the celebrated \emph{Plancherel measure} studied in~\cite{VershikKerov1977,LoganShepp1977,BaikDeiftJohansson1999,BorodinOkounkovOlshanski2000,Okounkov2000,IvanovOlshanski2002}.
For any $\alpha>0$, one can show (see~\cite{DolegaSniady2019}) that $\chi_d$ is a Jack character and the corresponding probability measure is the \emph{Jack--Plancherel measure} with formula 
\[  \mathbb{P}^\a_{\chi_d}(\la) = \frac{\al^d\,d!}{j_\la^\a}, \]
where
\begin{equation}\label{j_lambda}
j^\a_\la := \prod_{(i, j)\in\la}
\big(\alpha(\lambda_i-j)+(\lambda_j'-i)+1\big)
\big(\alpha(\lambda_i-j)+(\lambda_j'-i)+\alpha\big),
\end{equation}
and $\la' = (\la'_1\ge\la'_2\ge\cdots)$ is the \emph{conjugate} of the partition of $\la$, whose Young diagram (in French convention) is obtained from reflecting the Young diagram of $\la$ across the line $y=x$. The quantity $j^\a_\la$ originates from symmetric function theory, see e.g.~\cref{j_lambda_2} below.

Note that the cumulants associated to the Jack characters $\chi_d$ are given by 
\[\kappa^{\chi_d}_n(k_1,\dots,k_n) = \begin{cases}1, &\text{when } n=1 \text{ and } k_1=1,\\
0, & \text{otherwise},\end{cases}\]
so the sequence $(\chi_d)_{d\ge 1}$ has the enhanced AFP: indeed, all the constants in \eqref{eq:SecondCumu} and \eqref{eq:refined-asymptotics-characters-2} vanish, i.e.~$v_k=v_k'=v_{(k| l)}=0$, for all $k, l\ge 2$.
Hence, the theorems from~\cref{sec:thms_DS}, as well as our main theorems from~\cref{sec:Duality}, 
can be applied to this example, as long as the asymptotic condition~\eqref{eq:double-scaling-refined} 
on $\al$ is satisfied. The regimes of high temperature ($g>0,\,\al=\Theta(d)$), low temperature 
($g<0,\,\al=\Theta(d^{-1})$), and fixed temperature ($g=0,\,\al=\Theta(1)$) are all covered.

\subsubsection{Jack--Schur--Weyl measure}\label{ex:Schur-Weyl}

Let $N\in\R$ be a real parameter and set 
\begin{equation}\label{eq:Schur-Weyl}
\chi_d(\mu) := N^{-\|\mu\|}, 
\end{equation}
For $N\in\Z_{\ge 1}$, this is the normalized $\Sym{d}$-character 
considered in Schur--Weyl's construction, where $\Sym{d}$ acts naturally on 
$(\C^N)^{\otimes d} = \C^N\otimes\cdots \otimes \C^N$. Using~\cite[Ch.~VI,~(10.25)]{Macdonald1995}
for the value of the principal specialization of a Jack symmetric function, one finds:
\begin{equation}\label{eq:Schur-WeylDef}
\mathbb{P}^\a_{\chi_d}(\lambda) =
\frac{d!\,\sqrt{\alpha}^{\,d}}{N^d\cdot j^\a_\la}\cdot
\prod_{(i, j)\in\lambda}
\big(N\sqrt{\alpha}+(j-1)\alpha-(i-1)\big).
\end{equation}

For $\alpha=1$, the measure $\mathbb{P}^{(\alpha=1)}_{\chi_d}$ is the
\emph{Schur--Weyl measure} studied by Biane~\cite{Biane2001}, who
proved the LLN for the corresponding large random
Young diagrams and found an explicit formula for the limit shape.
In addition, for any $\al>0$, the measure~\eqref{eq:Schur-WeylDef} 
becomes the Jack--Plancherel measure in the limit $N\to\infty$. 
Finally, the Jack--Schur--Weyl measure can be obtained as 
a limit of the $z$-measures $M^{(d)}_{z,z',\theta}$ 
studied by Borodin and Olshanski, namely
\[ \mathbb{P}^\a_{\chi_d}(\lambda) =
\lim_{z' \to\infty}M^{(d)}_{N\sqrt{\theta},z',\theta}(\lambda),\]
holds with $\theta = \alpha^{-1}$ (see~\cite[Equation~(1.1)]{BorodinOlshanski2005}).

We remark that formula \eqref{eq:Schur-WeylDef} defines a probability measure 
whenever $N\sqrt{\alpha}\in\Z_{\ge 1}\,$ or $\,-N/\sqrt{\al}\in\Z_{\ge 1}$; 
in these cases, each $\chi_d$ is a Jack character by definition. 
If $N\sqrt{\alpha}\in\Z_{\ge 1}$ (resp.~$-N/\sqrt{\al}\in\Z_{\ge 1}$), 
then $\mathbb{P}^\a_{\chi_d}$ is supported on Young diagrams $\la\in\Y_d$ 
with $\ell(\la)\le N\sqrt{\al}$ (resp.~$\ell(\la')\le -N/\sqrt{\al}$).

The functions $\chi_d$ considered in this example are
multiplicative, i.e.~for $\mu^1,\dots,\mu^n \in\Y_{\leq d}$ 
such that $\mu^1\cdots\mu^n\in\Y_{\leq d}$, one has
$\chi_d(\mu^1\cdots\mu^n) = \chi_d(\mu^1)\cdots\chi_d(\mu^n)$.
Then all the cumulants $\kappa^{\chi_d}_n(k_1,\dots,k_n)$ of order $n \geq 2$ vanish, 
and in particular \eqref{eq:SecondCumu} is satisfied with all $v_{(k| l)}=0$; 
then $(\chi_d)_{d\ge 1}$ has the enhanced AFP if \eqref{eq:refined-asymptotics-characters-2} holds.
As a result, all theorems from~\cref{sec:thms_DS} and~\cref{sec:Duality} can be applied if each $\chi_d$ is a Jack character, and if \eqref{eq:double-scaling-refined} and 
\eqref{eq:refined-asymptotics-characters-2} are satisfied. We discuss all three possible regimes:

\smallskip

$\bullet$ \emph{High temperature:} Let $g, c>0$, and $w,g'\in\R$ be arbitrary parameters. 
Find sequences $\al=\al(d)>0,\ N=N(d)\in\frac{1}{\sqrt{\al}}\cdot\Z_{\ge 1}$ such that
\begin{equation*}
N\sqrt{\al} = \frac{\sqrt{\al d}}{c} + w\sqrt{\al} + o(\sqrt{d}),\qquad
\sqrt{\al} = g\sqrt{d} + g' + o(1),\ \text{ as }d\to\infty.
\end{equation*}
This is possible by first choosing $\al$ satisfying the second condition, 
and then choosing a positive integer $N\sqrt{\al}$ satisfying the first condition; the latter 
amounts to choosing an integer in an interval of length $o(\sqrt{d})$, which is always feasible for large $d$.
With these conditions on $\al, N$, the asymptotic relation \eqref{eq:double-scaling-refined} 
holds, and moreover 
\[ \chi_d(k) \ d^{\frac{k-1}{2}} = c^{k-1} - (k-1)wc^k\cdot d^{-\frac{1}{2}} + o(d^{-\frac{1}{2}}), \]
so \eqref{eq:refined-asymptotics-characters-2} holds too, with $v_k=c^{k-1},\, v_k'=-(k-1)wc^k$.

\smallskip

$\bullet$ \emph{Low temperature:} Let $g, c<0$, and $w,g'\in\R$ be arbitrary parameters. Find sequences $\al=\al(d)>0,\ N=N(d)\in -\sqrt{\al}\cdot\Z_{\ge 1}$ such that
\begin{equation*}
-\frac{N}{\sqrt{\al}} = -\frac{\sqrt{d}}{c\sqrt{\al}} - \frac{w}{\sqrt{\al}} + o(\sqrt{d}),\qquad \frac{1}{\sqrt{\al}} = -g\sqrt{d} - g' + o(1),\ \text{ as }d\to\infty.
\end{equation*}
With these asymptotics on $\al, N$, both \eqref{eq:double-scaling-refined} and 
\eqref{eq:refined-asymptotics-characters-2} hold with $v_k=c^{k-1},\, v_k'=-(k-1)wc^k$.

\smallskip

$\bullet$ \emph{Fixed temperature:} In this regime, we would like $\al(d)=\al>0$ to be fixed. 
Set $g=0,\, g'=\sqrt{\al}-\frac{1}{\sqrt{\al}}$, 
so that \eqref{eq:double-scaling-refined} holds. Let $c>0,\, w\in\R$ be arbitrary parameters.

The enhanced AFP property requires finding $N=N(d)\in\R$ such that either $N\sqrt{\al}\in\Z_{\ge 1},\ N\sqrt{\al} = 
\frac{\sqrt{\al d}}{c}+w\sqrt{\al}+o(1)$, or $-\frac{N}{\sqrt{\al}}\in\Z_{\ge 1},\ 
-\frac{N}{\sqrt{\al}} = -\frac{\sqrt{d}}{c\sqrt{\al}}-\frac{w}{\sqrt{\al}}+o(1)$. 
However, it is impossible to choose an integer in an interval of
length $o(1)$. Nevertheless, we can pick $N,\al$ s.t.~$N\sqrt{\al}\in\Z_{\ge 1},\ N\sqrt{\al} = \frac{\sqrt{\al d}}{c}+O(1)$, or $-\frac{N}{\sqrt{\al}}\in\Z_{\ge 1},\ -\frac{N}{\sqrt{\al}} =
-\frac{\sqrt{d}}{c\sqrt{\al}}+O(1)$, which guarantees that the AFP
condition is satisfied with $v_k=c^{k-1}$. The LLN from~\cref{theo:ShapeDS} only requires
the AFP, so we can apply it to our example. The $c$--dependent limit
shape was found by Biane~\cite{Biane2001} in the special case $\al=1$; it coincides with the celebrated Vershik--Kerov/Logan--Shepp 
limit shape in the limit $c\to 0$.

Moreover, the proof of~\cref{theo:CovDS} shows that the weaker
conditions imposed by the AFP guarantee that the 
cumulants of order $\ge 3$ vanish for the random variables $\langle\Delta^\a_d, x^m\rangle$, 
and the limits $\lim_{d\to\infty}{\Cov( \langle\Delta^\a_d, x^{k-2}\rangle,\, \langle\Delta^\a_d, x^{l-2}\rangle )}$ 
exist and have explicit combinatorial formulas. The
difference between the AFP and the enhanced AFP only affects the
existence of the limit
$\lim_{d\to\infty}{\E[\langle\Delta^\a_d, x^m\rangle]}$, which is not
guaranteed by the weaker assumption of the AFP. Note that the special
case $\alpha=1$ discussed in~\cite{Sniady2006c} omits this problem by
proving the CLT for the centered observables
(see~\cite[Corollary 3.3]{Sniady2006c}). Our case of a general fixed $\alpha>0$
covers this theorem, but we also prove a
stronger result: our choice of $N,\al$, e.g.~satisfying 
$N\sqrt{\al}\in\Z_{\ge 1},\ N\sqrt{\al} = \sqrt{\al d}/c+O(1)$,
implies that we can restrict to a subsequence $(d_k)_k$ such that 
$N(d_k)\sqrt{\al}=\sqrt{\al d_k}/c+w\sqrt{\al}+o(1)$,
for some fixed $w\in\R$. Then the convergence of the random
function $\Delta^\a_d$ to the Gaussian process $\Delta_{\infty}$
stated in~\cref{theo:CovDS} still holds true but with the sequence
$(\Delta^\a_d)_{d\in\Z_{\geq 1}}$ replaced by $(\Delta^\a_{d_k})_{k\in\Z_{\geq 1}}$.

\subsubsection{Conditional Jack--Thoma measure}\protect\footnote{In connection with this example, see also~\cref{ex:CDMPois} and~\cref{plancherel_jack} below.}\label{ex:CDM}
Let $(v_k)_{k\ge 2}$ be any sequence of reals satisfying \eqref{eq:what-is-m}, 
and for any integer $k\geq 2$, choose a sequence $(v^{(d)}_k)_{d\ge 1}$ such that
\begin{equation}\label{asymptotic_vd}
v^{(d)}_{k}\cdot d^{\frac{k-1}{2}} = v_k + o(d^{-\frac{1}{2}}), \text{ as }d\to\infty.
\end{equation}
Then consider the following function on $\Y_d$:
\begin{equation*}
\chi_d(\mu) := \prod_{{k} \geq 2}\left(v^{(d)}_{{k}}\right)^{m_{k}(\mu)}.
\end{equation*}
These functions are multiplicative; together with~\eqref{asymptotic_vd}, we conclude 
that $(\chi_d)_{d\ge 1}$ satisfies the enhanced AFP with $v_k'=v_{(k|l)}=0$.
We prove in \cref{prop:Depoissonization}  that the associated measure can be expressed as 
\[ \mathbb{P}^\a_{\chi_d}(\lambda) = \frac{\al^d\,d!}{u^d}\cdot\frac{J_\la^\a(u\cdot\vv_{u,\al})}{j^\a_\la}, \]
where $u>0$ is arbitrary, and $J_\la^\a(u\cdot \vv_{u,\al})$ is defined as 
the image of $J_\la^\a$ under the specialization (unital algebra homomorphism) 
$\rho\colon\SSym\rightarrow\C$ from the $\R$-algebra of symmetric 
functions $\SSym$ that maps the power sums $p_k\,$ to 
$\,u\cdot (u/\sqrt{\al})^{k-1}\cdot v_k^{(d)}$.

In order to apply our main results, including those from~\cref{sec:Duality}, 
we must impose that $\al$ satisfies \eqref{eq:double-scaling-refined} 
and that each $\chi_d$ is a Jack character. 
One contribution of this paper is the classification of the \emph{totally Jack positive specializations} 
in~\cref{thm_jack_spec} that allows to find large families of examples of parameters 
$\al,\, v_k^{(d)}$ that fit into this setting. 
By~\cref{prop:CDMPar2} and~\cref{prop:Depoissonization}, 
we are able to consider all three limit regimes as follows.

\smallskip

$\bullet$ \emph{High temperature:} Consider any parameters $g>0$ and $a_1\ge a_2\ge\cdots\ge 0$ 
such that $\sum_{i=1}^\infty{a_i}\leq 1$.
Set $\al = g^2d$, so that \eqref{eq:double-scaling-refined} holds with $g'=0$. Also set 
\begin{equation}\label{vkd_1}
v_k^{(d)} = \left(\frac{g\sqrt{d}}{\lceil gd\rceil}\right)^{k-1}\cdot\sum_{i=1}^\infty{a_i^k},\ 
\text{ for }k\ge 2.
\end{equation}
Then $(\chi_d)_{d\ge 1}$ are Jack characters with the enhanced AFP 
and $v_k =\!\sum_{i=1}^\infty{a_i^k}$.

\smallskip

$\bullet$ \emph{Low temperature:} Consider any parameters $g<0$ and $a_1\ge a_2\ge\cdots\ge 0$ 
such that $\sum_{i=1}^\infty{a_i}\leq 1$. Set $\alpha = \frac{1}{g^2d}$, so that \eqref{eq:double-scaling-refined} holds with $g'=0$. Also set 
\begin{equation}\label{vkd_2}
v_k^{(d)} = \left(\frac{g\sqrt{d}}{\lceil -gd\rceil}\right)^{k-1}\cdot\sum_{i=1}^\infty{a_i^k},\ 
\text{ for }k\ge 2.
\end{equation}
Then, due to \cref{prop:DualityAFP}, $(\chi_d)_{d\ge 1}$ are Jack characters with the enhanced AFP 
and $v_k = (-1)^{k-1}\sum_{i=1}^\infty{a_i^k}$.

$\bullet$ \emph{Fixed temperature:} For this regime, $\al>0$ is fixed, so 
$g=0,\, g'=\sqrt{\al}-\frac{1}{\sqrt{\al}}$. 
We prove in \cref{prop:Depoissonization} that
$\chi_d$ is a Jack character whenever 
$(\sqrt{\al}, \sqrt{\al}\,v_2^{(d)}, \sqrt{\al}\,v_3^{(d)},\dots)$
determines an $\al$-Jack-positive specialization, as in \cref{alpha_jack_positive}.
Then \cref{thm:KOO}, due to
Kerov--Okounkov--Olshanski~\cite{KerovOkounkovOlshanski1998},
characterizes the valid sequences $(v_k^{(d)})_k$. 
Take for instance any reals $a_1 \geq a_2 \geq \dots \geq 0$ and $b_1 \geq b_2 \geq
\dots \geq 0$ such that $\sum_{i=1}^\infty (a_i+b_i) \leq 1$, 
and a sequence of integers $n=n(d)$ such that $n\le\sqrt{\al d}$ 
and $n=\sqrt{\al d}+O(1)$; then set:
\[v_{k}^{(d)} := \frac{n}{\sqrt{\al}}\cdot\sum_{i=1}^\infty\left(\frac{a_i}{\sqrt{d}}\right)^{k} + 
\frac{n}{\sqrt{\al}}\cdot (-\al)^{1-k}\sum_{i=1}^{\infty}\left(\frac{b_i}{\sqrt{d}}\right)^k,\ 
\text{ for } k\ge 2.\]
Then it turns out that all results from~\cref{sec:Duality} are true, except for the existence of the limits 
$\lim_{d\to\infty}{\E[\langle\Delta^\a_d, x^m\rangle]}$; for that, one must 
restrict to subsequences $(d_k)_k$ with
$n(d_k)=\sqrt{\al d_k}+o(1)$,
cf.~\cref{ex:Schur-Weyl} and \cref{rem:CompWithSniady}.
The parameters $v_k$ are $v_k =
\sum_{i=1}^\infty{a_i^k}+(-\al)^{1-k}\sum_{i=1}^\infty{b_i^k}$.
\begin{remark}
We point out that the scaling of the random partitions in our LLN is different from that in \cite{KerovOkounkovOlshanski1998,Vershik1974}. Indeed, these other articles discuss a LLN in which the parameters $a_1,a_2,\cdots$ (resp.~$b_1,b_2,\cdots$) appear as the rates of linear growth of rows (reps.~columns) of the random partitions. In contrast, our random partitions $\la^{(d)}\in\Y_d$ are \emph{balanced} in the sense that both the number of rows and columns are asymptotically proportional to $\sqrt{d}$.
\end{remark}
\begin{remark}
  \label{rem:CompWithSniady}
In general, whenever we replace
the assumption of the enhanced AFP by the (weaker) AFP in
\cref{theo:CovDS}, we have its weaker version which says that the
family of centered random variables $x_\ell := X_\ell-\E X_\ell$, where
$X_\ell := \langle\Delta^\a_d, x^\ell\rangle$, converges jointly to a
Gaussian distribution in the weak topology of probability
measures. The special case $\alpha=1$ recovers the result of
Śniady~\cite{Sniady2006c}. Similarly, by replacing
\eqref{eq:refined-asymptotics-characters-2} in the definition of the
enhanced AFP by a weaker condition $\chi_d(k) \ d^{\frac{k-1}{2}}  =
v_k+ O(d^{-\frac{1}{2}})$
that we encountered in all our examples, then \cref{theo:CovDS} holds
true after restricting to a subsequence.
\end{remark}

\section{Second model of random Young diagrams: Jack measures}\label{sec:Jack-deformed-2}

In this section, we review the construction from \cite{Moll2023} that yields 
probability measures on the infinite set $\Y := \bigsqcup_{d \geq 0}\Y_d$ of all
Young diagrams and it relies on the Cauchy identity for Jack symmetric functions; 
the result is a natural $1$-parameter generalization of Okounkov's Schur measures~\cite{Okounkov2001}.
We also isolate a special family of Jack measures, to be called 
\emph{Jack--Thoma measures}, that are closely related to the 
model of random Young diagrams discusssed in~\cref{sec:Jack-deformed-1}, 
and to the work of Kerov--Okounkov--Olshanski \cite{KerovOkounkovOlshanski1998}.
We prove the LLN and CLT for this new family of measures.

\subsection{Definitions and preliminaries}
Recall the Cauchy identity for Jack symmetric functions~\cite[Ch.~VI.10]{Macdonald1995}:
\begin{equation}\label{eq:Cauchy}
1 + \sum_{d=1}^\infty{t^d \sum_{\la\in\Y_d}{\frac{J_\la^\a(\pp)J_\la^\a(\qq)}{\| J_\la^\a \|^2}}} = 
1 + \sum_{d=1}^\infty{t^d \sum_{\mu\in\Y_d}{\frac{p_\mu q_\mu}{\al^{\ell(\mu)}z_\mu}}} = 
\exp\bigg(\sum_{k\geq 1}\frac{t^{k} p_{k} q_{k}}{{k}\alpha}\bigg),
\end{equation}
where $\pp = (p_1,p_2,\dots)$ and $\qq = (q_1,q_2,\dots)$ are two
independent families of variables which play the role of power-sums. 
We employed the usual notations: $p_\mu := \prod_{i\ge 1}{p_{\mu_i}}$, 
$J_\lambda(\pp) := \sum_{\mu\in\Y_d}\theta^\a_\mu(\lambda)p_\mu$,
and $q_\mu, J_\lambda(\qq)$ are defined similarly.
The norm $\|\cdot\|$ in the LHS of \eqref{eq:Cauchy} is with respect to the $\alpha$-deformed Hall scalar product, defined on the $\R$-algebra of symmetric functions $\SSym$ by declaring 
the basis $\{p_{\mu}\}_{\mu\in\Y}$ to be orthogonal, and
\begin{equation}\label{eq:powersumnorm}
\| p_\mu \|^2 = \alpha^{\ell(\mu)}z_\mu,\quad\forall\,\mu\in\Y.
\end{equation}
It is known that the Jack symmetric functions $J^\a_\la$ are orthogonal and have norms 
\begin{equation}\label{j_lambda_2}
\| J^\a_\la \|^2 = j^\a_\la,
\end{equation}
where $j^\a_\la$ is given in~\eqref{j_lambda}.
Note that $j^\a_\la\in\Z_{\ge 0}[\alpha]$, in particular $j^\a_\la>0$, whenever $\al>0$.

\begin{definition}[Second model of random Young diagrams: Jack measures]\label{def:second}
Given a fixed $\alpha>0$, let $\rho_1,~\rho_2: \Symm\to\mathbb{C}$ be any two
unital $\R$-algebra homomorphisms such that
\begin{align}
\rho_1\left(J_{\la}^\a\right)
  \overline{\rho_2\left(J_\la^\a\right)} &\geq 0, \text{ for all } \lambda\in\Y, \label{eqn:positive}\\
\sum_{k=1}^{\infty}{\frac{\rho_1(p_k) \overline{\rho_2(p_k)}}{k}} &<
  \infty. \label{eqn:bounded}
\end{align}
Then define the \textbf{Jack measure} $M^\a_{\rho_1, \rho_2}$ as the
probability measure on $\Y$ given by
\begin{equation}\label{eqn:Jack_measure}
M_{\rho_1, \rho_2}^\a(\la) := \frac{\rho_1\left(J_{\la}^\a\right)
  \overline{\rho_2\left(J_\la^\a\right)}}
{j^\a_\la}\exp\left(-\sum_{k=1}^{\infty}{\frac{\rho_1(p_k) \overline{\rho_2(p_k)}}{k\alpha}} \right),
\quad \la\in\Y.
\end{equation}
\end{definition}

The normalization $\,\sum_{\la\in\Y}{M^\a_{\rho_1, \rho_2}(\la)} = 1$ 
follows from \eqref{eq:Cauchy}, while the
nonnegativity $M^\a_{\rho_1, \rho_2}(\la) \geq 0$ is a direct
consequence of the assumptions \eqref{eqn:positive}, \eqref{eqn:bounded}.

\subsection{Examples of Jack measures}\label{sec:examples_2}

\subsubsection{Principal series Jack measures}\label{ex:Moll}
The measure $M_{\rho_1, \rho_2}^\a$ was studied by the
third author in the special case when $\rho := \rho_1 = \rho_2$,
see~\cite{Moll2023}. Note that~\eqref{eqn:positive} is automatically satisfied 
when the specializations $\rho_1, \rho_2$ are equal.
The main results of~\cite{Moll2023} are the LLN and CLT for $M_{\rho,\rho}^\a$ 
under a natural assumption that guarantees \eqref{eqn:bounded}, namely 
$|\rho(p_{k})| \leq A\cdot r^{k}$, for all ${k}\ge {1}$, and some constants $A >0,\, r<1$.

\subsubsection{Poissonized Jack--Plancherel measures}\label{ex:PlancherelPois}
  Let us look at the special case of \cref{ex:Moll} when $\rho (p_{k})= u\cdot\delta_{1,k}$, for some $u\in\C$, and $\delta_{1, k}$ is the Kronecker delta.
  This is the so-called \emph{Plancherel specialization}.
From the fact that the coefficient of $p_1^{|\la|}$ in the expansion of $J^\a_\la$ is $1$ (see \cite[Ch.~VI,~(10.29)]{Macdonald1995}), one finds the explicit formula
\begin{equation*}
M_{\rho, \rho}^\a(\la) = 
\exp\left(-\frac{|u|^2}{\alpha}\right) \frac{|u|^{2|\la|}}{j^\a_\la}.
\end{equation*}
Observe that $M_{\rho, \rho}^\a$ is
  the Poissonization of the Jack--Plancherel measures from
  \cref{ex:Plancherel} with the Poisson parameter equal to
  $\frac{|u|^2}{\alpha}$. In other words,
  \[ M^\a_{\rho, \rho}(\lambda) = \exp\left(-\frac{|u|^2}{\alpha}\right)
    \frac{|u|^{2d}}{\al^d d!}\cdot \PP_{\chi_d}^\a(\lambda),\quad\forall\, \la\in\Y_d,\ d\in\Z_{\ge 0},\]
where $\PP^\a_{\chi_d}$ is the probability measure associated
with the character $\chi_d$ from \cref{ex:Plancherel}.

\subsubsection{Poissonized Jack--Schur--Weyl measures}\label{ex:PrincipPlancherelPois}
  Consider a generalization of the previous example, when
  $\rho_1$ is the \emph{principal specialization} $\rho_1 (p_k) = u\cdot c^{k-1}$, while $\rho_2$ is still the Plancherel specialization $\rho_2 (p_k)= u\cdot\delta_{1,k}$. 
  This example is not an instance of \cref{ex:Moll}, but
  it clearly satisfies~\eqref{eqn:bounded}. 
One can also check that \eqref{eqn:positive} is satisfied whenever $u\in\R, c\in\R^*$, and 
$u/c\in\Z_{\ge 1}$ or $-u/(c\al)\in\Z_{\ge 1}$.
Indeed, by \cite[Ch.~VI,~(10.25)]{Macdonald1995}, one finds
\begin{equation*}
M_{\rho_1,\rho_2}^\a(\la) = \exp\left(-\frac{u^2}{\alpha}\right)\cdot
\frac{(uc)^{|\la|}}{j_\la^\a} \cdot 
\prod_{(i, j)\in\la}{\left(\frac{u}{c}+\alpha(j-1)-(i-1)\right)}.
\end{equation*}
If $u/c\in\Z_{\ge 1}$ (resp.~$-u/(c\al)\in\Z_{\ge 1}$), then 
$M_{\rho_1,\rho_2}^\a$ is supported on partitions $\la$ 
with $\ell(\la)\le u/c$ (resp.~$\ell(\la')\le -u/(c\al)$).
Moreover, $M_{\rho_1,\rho_2}^\a$ is the 
Poissonization of the Jack--Schur--Weyl measures $\PP^\a_{\chi_d}$ from
\cref{ex:Schur-Weyl} with $N=\frac{u}{c\sqrt{\alpha}}$ and 
Poisson parameter $\frac{u^2}{\alpha}$.

\subsubsection{Jack--Thoma measures}\label{ex:CDMPois}
\footnote{See also \cref{plancherel_jack} below.}
The example from~\cref{ex:PrincipPlancherelPois} is a special case of the following general framework 
that will be the main object of interest later in this section. Let $\vv = (v_1, v_2, \cdots)\in\R^{\infty}$ 
and $u>0$. Define $M^\a_{u; \vv} $ as the Jack measure 
$M^\a_{\rho_1, \rho_2}$ with specializations 
$\rho_1(p_k) = u\cdot v_k$ and $\rho_2(p_k) = u\cdot\delta_{1,k}$.
If we denote $u\cdot\vv := (u\cdot v_1,u\cdot v_2,\cdots)$ and
$J_{\la}^\a(u\cdot\vv) := \rho(J_{\la}^\a)$, for the specialization $\rho(p_{k}) := u\cdot v_k$, then
\begin{equation*}
M^\a_{u; \vv}(\la) = 
\exp\left(-\frac{u^2\cdot v_1}{\al}\right)\frac{J_{\la}^\a(u\cdot\vv)\cdot u^{|\la|}}{j^\a_\la}.
\end{equation*}
This is a probability measure if and only if
$J_\la^\a(u\cdot\vv) \geq 0$, for all $\lambda\in\Y$.
We find in~\cref{sec:JackPos} some choices of $u,\vv$ that imply this positivity assumption, 
see e.g.~\cref{prop:CDMPar2}.
We prove in~\cref{prop:Depoissonization} that if $v_1=1$, the measure $M^\a_{u;\vv}$ is the Poissonization of the measures $\PP^\a_{\chi_d}$ from \cref{ex:CDM} with Poisson parameter $\frac{u^2}{\alpha}$.

\subsection{Transition measures}\label{sec:transition_measure}

Given any Young diagram $\la\in\Y$, we use Kerov's idea~\cite{Kerov1993} to 
associate to $\la$ a probability measure $\mu_\lambda$ on $\R$, 
called the \emph{transition measure of $\la$}; it is characterized by the following identity for its Cauchy transform:
\begin{equation}\label{eq:Cauchy-St}
G_{\mu_\la}(z) = \int_\R \frac{d\mu_\la(x)}{z-x} = \frac{1}{z+\ell(\la)}\prod_{i=1}^{\ell(\la)}\frac{z+i-\la_i}{z+i-1-\la_i}
= \frac{\prod_{i=1}^k(z-y_i)}{\prod_{i=1}^{k+1}(z-x_i)},
\end{equation}
where $x_1,\dots, x_{k+1}$ ($y_1,\dots,y_k$, respectively) are the local
minima (maxima, respectively) of the profile $\omega_{\lambda}$, see~\cref{fig:french-fries}.
More generally, for any interlacing sequence $x_1<y_1<x_2<\cdots<y_k<x_{k+1}$ of real numbers, there exists a unique probability measure whose Cauchy transform is the rational function given by the RHS of~\eqref{eq:Cauchy-St} and, in particular, it has mean $\sum_{i=1}^{k+1}{x_i} - \sum_{i=1}^k{y_i}$.
This allows us to define the transition measure $\mu_{T_{w, h}\la}$ of any anisotropic diagram $T_{w, h}\la$ by the same equation~\eqref{eq:Cauchy-St}, where the $x_i$ and $y_i$ are the local minima and maxima of $T_{w, h}\la$, respectively.
Note that the mean of any transition measure is zero, since $\sum_{i=1}^{k+1}{x_i} = \sum_{i=1}^k{y_i}$.

Going even further, for any infinite interlacing sequence of real numbers $x_1<y_1<x_2<y_2<\cdots$, or $x_1>y_1>x_2>y_2>\cdots$, where we assume that $\lim_{k\to\infty}|y_k| = \infty$, and
\begin{equation}\label{eq:KerovCondInfinite}
\sum_{i=1}^\infty\frac{x_{i+1}-y_i}{y_i} < \infty
\end{equation}
(we remove from the sum the term corresponding to $y_i = 0$, if it exists), one can show that (see~\cite[Section 1.3 and Section 2.1]{Kerov1998}):

\begin{itemize}

	\item the infinite product
\[
G(z) := \prod_{i=1}^\infty \frac{z-y_i}{z-x_i}
\]
converges for $z\ne x_i$, $i=1,2\dots$, the zeroes of $G(z)$ are located at $y_1,y_2,\dots$, and its poles are located at $x_1,x_2,\dots$;

	\item the partial fraction decomposition
\[
G(z) = \sum_{i=1}^\infty\frac{\mu_i}{z-x_i}
\]
is such that $\mu_i>0$, for all $i\ge 1$, and $\sum_{i=1}^\infty{\mu_i}=1$.
As a result,
\begin{equation}\label{eq:TransitionSupport}
\mu := \sum_{i=1}^\infty \mu_i\delta_{x_i}
\end{equation}
defines a probability measure on $\R$ whose Cauchy transform is precisely $G_{\mu}(z)=G(z)$.

\end{itemize}

\subsection{Observables of measures with all finite moments}

We will use several sequences of observables on the set of probability measures on $\R$ with finite moments of all orders.
Besides the \emph{moments}, we will use the \emph{Boolean cumulants}, the \emph{free cumulants} and the \emph{fundamental functionals of shape}.
If $\mu$ is a probability measure on $\R$ with finite moments of all orders, then let us denote its $\ell$-th moment by $M^\mu_\ell := \int_{\R}{x^\ell \mu(dx)}$. 
The sequence of moments $(M^\mu_\ell)_{\ell\geq 1}$ can also be described by expanding 
the Cauchy transform of $\mu$ in a neighborhood of infinity:
\begin{equation}\label{eq:Moment}
G_{\mu}(z) = \int_\R \frac{d\mu(x)}{z-x} = 
\frac{1}{z}+\sum_{\ell \geq 1}M^\mu_\ell z^{-\ell-1}.
\end{equation}

  The associated sequences of Boolean cumulants $(B^\mu_\ell)_{\ell
    \geq 1}$, free cumulants $(R^\mu_\ell)_{\ell
    \geq 1}$, and fundamental functionals of shape $(S^\mu_\ell)_{\ell
    \geq 1}$ can be defined by the following relations between their generating functions:
  \begin{align}
          B_{\mu}(z) &:= \sum_{\ell \ge
                   1} B^\mu_\ell \, z^{-\ell+1} = z-\frac{1}{G_\mu(z)},
                         \label{eq:Boolean}\\
      K_{\mu}(z) &:=  \sum_{\ell \ge
                   1} R^\mu_\ell \, z^{\ell-1} = G^{-1}_{\mu}(z)-\frac{1}{z},
                             \label{eq:Free}\\
    S_{\mu}(z) &:=  \sum_{\ell \ge
                 1} S^\mu_\ell \, z^{-\ell} = \log \left(z G_{\mu}(z)\right),
     \label{eq:Stransform}
\end{align}
where $G^{-1}_{\mu}(z) $ denotes the (formal) compositional inverse
of $G_\mu$. Observe that the first elements in these sequences 
of observables are the same: $B^\mu_1 = R^\mu_1 = S^\mu_1 = M^\mu_1$.

\begin{proposition}
  \label{prop:relations}
  For any integer $\ell \geq 1$,
  \begin{align}
    M^\mu_\ell &= \sum_{n \geq 1}\sum_{k_1,\dots, k_n \geq
    1\atop k_1+\cdots k_n = \ell}B^\mu_{k_1}\cdots B^\mu_{k_n},\label{eq:Moment-Boolean}\\
    S^\mu_\ell &= \sum_{n \geq 1}\frac{1}{n}\sum_{k_1,\dots, k_n \geq
    1\atop k_1+\cdots k_n = \ell}B^\mu_{k_1}\cdots B^\mu_{k_n}. \label{eq:S-Boolean}
    \end{align}
  \end{proposition}

\begin{proof}
In order to prove \eqref{eq:Moment-Boolean}, we use~\eqref{eq:Boolean}:
\[ M^\mu_\ell = [z^{-\ell-1}]G_\mu(z) = [z^{-\ell-1}]\frac{1}{z-B_{\mu}(z)}
= [z^{-\ell}]\frac{1}{1-z^{-1}B_{\mu}(z)} = [z^{-\ell}] \sum_{n\ge 1}(z^{-1}B_{\mu}(z))^n,\]
which equals the right hand side of~\eqref{eq:Moment-Boolean}. Similarly, to prove~\eqref{eq:S-Boolean}, we use \eqref{eq:Boolean} and \eqref{eq:Stransform} to get
\[ S^\mu_\ell = [z^{-\ell}] \log \left(z G_{\mu}(z)\right) 
= [z^{-\ell}]\log \bigg(\frac{1}{1- z^{-1}B_{\mu}(z)}\bigg)
= [z^{-\ell}]\sum_{n\ge 1}{\frac{1}{n}(z^{-1}B_\mu(z))^n},\]
which equals the right hand side of \eqref{eq:S-Boolean}.
\end{proof}

The moments, Boolean cumulants, free cumulants and fundamental functionals of shape associated with a transition measure $\mu_\lambda$ will be denoted by $M_\ell(\lambda)$, $B_\ell(\lambda)$, $R_\ell(\lambda)$, $S_\ell(\lambda)$, respectively, and we think of them as functions on the set of Young diagrams.
We will also use the analogous notations if the Young diagram $\la$ is replaced by an anisotropic diagram $T_{w, h}\la$.
We remark that, for all $X = M,B,R,S$, and all anisotropic diagrams $T_{w, h}\la$, we have
\begin{equation}\label{eq:values_X}
X_1(T_{w,h}\la) = 0,\qquad X_2(T_{w,h}\la) = wh|\la|;
\end{equation}
see~\cite[Sec.~2.6]{Kerov1993transition} for more details.
Finally, we have the following \emph{scaling property}:
\begin{equation}\label{eq:scaling0}
X_\ell(T_{cw, ch}\la) = c^\ell X_\ell(T_{w, h}\la), \textrm{ for all }c>0 \textrm{ and }\ell\in\Z_{\ge 1}.
\end{equation}

In this section, we will be interested in the following scaling of random partitions:
\begin{equation}\label{eq:Lambda2}
\Lambda_{(\alpha;u)} := T_{\frac{\alpha}{u},\frac{1}{u}}\lambda,
\end{equation}
for some $u>0$; note that this rescaled partition has boxes whose ratio of dimensions is $\alpha$.
The parameter $u$ will be chosen later, depending on the random partition ensemble, in such a way that the area of the rescaled partition $\Lambda_{(\alpha;u)}$ is equal to $1$.
With this new notation, the scaling property implies in particular that:
\begin{equation}\label{eq:scaling}
X_\ell(\Lambda_{(\alpha;u)}) = u^{-\ell}X_\ell(\Lambda_{(\alpha;1)}).
\end{equation}

\subsection{The Markov--Krein correspondence}\label{subsub:Markov-Krein}

The relation between the profile $\omega_\Lambda$ and the transition measure 
$\mu_\Lambda$ of an anisotropic diagram $\Lambda$ is given by \eqref{eq:Cauchy-St} and this can be restated in the following form, known as the \emph{Markov--Krein correspondence}:
\begin{equation}\label{eq:KreinMarkov}
\frac{1}{z}\exp\int_\R\frac{1}{x-z}\left(\frac{\omega_\Lambda(x)-|x|}{2}\right)'dx 
= \int_\R \frac{d\mu_\Lambda(x)}{z-x}.
\end{equation}

Applying~\eqref{eq:Stransform} to the above equation, we obtain the following formula for the fundamental functional of shape:
    \begin{equation}
    \label{eq:SFormula}
   S_n(\Lambda) = -\int_\R x^{n-1}\left(\frac{\omega_\Lambda(x)-|x|}{2}\right)'
    dx = (n-1)\int_\R x^{n-2}\left(\frac{\omega_\Lambda(x)-|x|}{2}\right)dx.
  \end{equation}

The Markov--Krein correspondence works in a much wider generality, and we briefly state it here in the most general framework that we are aware of.\footnote{Typically the Markov--Krein correspondence is stated for probability measures with compact support, e.g.~\cite{Kerov1993}. We follow the more general framework from~\cite{Meliot2017}, which is based on Kerov's work~\cite{Kerov1998}}

  We call \emph{generalized continuous Young diagram} a continuous function
  $\omega\colon \R \to \R$ with Lipschitz constant $1$ and such that
  the following integrals converge:
  \begin{equation}
    \label{eq:GenContYoung}
    \int_{-\infty}^{0}\frac{1+\omega'(x)}{1-x}dx < \infty; \ \ \ \
    \int_{0}^{\infty}\frac{1-\omega'(x)}{1+x}dx < \infty.
  \end{equation}

For instance, if $\omega(x)$ is a zig-zag line with slopes $\pm 1$, such that $\omega(x)=-x$ whenever $x\ll 0$, and whose local minima and maxima are achieved at points in the interlacing sequence $x_1 < y_1 < x_2 < y_2 <\cdots$, then the condition~\eqref{eq:GenContYoung} is equivalent to the condition~\eqref{eq:KerovCondInfinite} (see~\cite[Section 1.3]{Kerov1998}).

  We denote by $\Y^1$ the set of equivalence classes of generalized
  continuous Young diagrams, where $\omega_1 \sim \omega_2 \iff
  \omega_1' \equiv \omega_2'$. The space of generalized continuous Young diagrams $\Y^1$
  is a Polish space (separable and metrizable complete space) with the
  following topology: we say that $\omega_n \to \omega$ if the following integrals are uniformly bounded:
    \[ \sup_{n\in\Z_{\ge 0}}\left( \int_{-\infty}^{0}\frac{1+\omega_n'(x)}{1-x}dx \right) < \infty; \ \ \ \
    \sup_{n\in\Z_{\ge 0}}\left(\int_{0}^{\infty}\frac{1-\omega_n'(x)}{1+x}dx\right) < \infty,  \]
and there exist representatives $\omega_n(x)$ and
    $\omega(x)$ such that for any $N >0$ one
    has $\lim_{n \to \infty}\|\omega_n-\omega\|_{L^{\infty}([-N,N])}
      = 0$.

We define two other Polish spaces (see~\cite{Meliot2017} for the details): $\mathscr{M}^{1}$ the space of
probability measures on the real line $\R$ with the topology of weak
convergence and $\mathscr{N}^{1}$ the space of holomorphic functions $h$
on the Poinc\'are half-space
\[ \mathbb{H} := \{z \in \C\colon \Im(z) > 0\}\]
which take values with negative imaginary parts and such that $\lim_{y
\to +\infty} \mathbf{i} yh(\mathbf{i} y) = 1$ if $\mathbf{i} = \sqrt{-1}$ (the topology is given by the proper convergence
i.e.~$h_n \to h$ if $h_n$ converges uniformly to $h$ on all compact
subsets of $\mathbb{H}$ and
\[ \lim_{N \to \infty}\sup_{n \in \Z_{\ge 0}\atop y\geq N}| \mathbf{i}yh_n( \mathbf{i} y)-1|=0). \]

Then, the following Markov--Krein correspondence holds true.

\begin{theorem}{\cite[Theorem 7.25]{Meliot2017}}
  \label{theo:Meliot}
  There is a sequence of homeomorphisms
  \[ \omega \in \Y^1 \leftrightarrow G_\omega \in \mathscr{N}^{1}
    \leftrightarrow \mu_\omega \in \mathscr{M}^{1}\]
  given by:
\begin{equation}\label{MK_moments}
G_\omega(z) = \int_\R \frac{d\mu_\omega(x)}{z-x} = \frac{1}{z}\exp\int_\R\frac{1}{x-z}\left(\frac{\omega(x)-|x|}{2}\right)' dx.
\end{equation}
\end{theorem}

Next, we will study the typical shape of the rescaled diagram $\Lambda_{(\alpha;u)}$ (recall the rescaling~\eqref{eq:Lambda2}) when $\lambda$ is sampled from the Jack--Thoma measure, and the Poisson parameter tends to infinity.
We prove that in various regimes such a typical shape $\Lambda_{g;\vv}$ exists, the fluctuations around this shape are Gaussian and we describe the limit shape via the Markov--Krein correspondence (we give an explicit combinatorial formula for the associated Cauchy transform in terms of certain lattice paths).

\subsection{Łukasiewicz paths and Łukasiewicz ribbon paths}\label{Luk_sec}

An \emph{excursion} $\Gamma$ of length $\ell$ is a sequence of points $\Gamma = (w_0,\dots,w_{\ell})$ on $(\Z_{\geq 0})^2$ such that $w_j = (j,y_j)$, with $y_0=y_\ell =0$, $y_1,\dots,y_{\ell-1} \in \Z_{\geq 0}$, and it is uniquely
encoded by the sequence of its \emph{steps} $e_j := w_{j}-w_{j-1} = 
(1,y_{j}-y_{j-1})$. Informally speaking, an excursion is a directed lattice path
with steps of the form $(1,k), k \in\Z$, starting at
$(0,0)$, finishing at $(\ell,0)$, and that stays in the first
quadrant. A step $e=(1,0)$ is called a \emph{horizontal step}, while
steps $e=(1,k),\, k \in \Z_{>0}$ (resp.~$k \in \Z_{<0}$), are
called \emph{up steps} (resp.~\emph{down steps}) \emph{of degree $k$}, 
see~\cref{fig:Lukasiewicz} for examples.

For a given excursion $\Gamma = (w_0,\dots,w_{\ell})$, the set of
points $\SSS(\Gamma):=\{w_1,w_2,\dots,w_\ell\}$ (not counting the origin $w_0=(0, 0)$) 
naturally decomposes into pairwise disjoint subsets
consisting of those points which are preceded by a horizontal, down step or up step:
\[\SSS(\Gamma) = \SSS_{\tiny\rightarrow}(\Gamma)\cup
  \SSS_{\tiny\searrow} (\Gamma)\cup \SSS_{\tiny\nearrow}(\Gamma).\]
Denote by $\SSS_n(\Gamma)$, $n\in\Z\setminus\{0\}$, 
the set of points preceded by a step $(1, n)$, i.e.~$\SSS_n(\Gamma)$ 
is the set of points preceded by an up step 
of degree $n$, if $n>0$, or by a down step of degree $-n$, if $n<0$.
Also denote by $\SSS^i(\Gamma)$ the subset of points with second coordinate equal to $i$, i.e.
\[ \SSS^i(\Gamma) := \{w_j = (j,y_j)\in\SSS(\Gamma) \,\mid\, y_j=i\}.\]
We say that points $w_j\in\SSS^i(\Gamma)$ are at \emph{height $i$}.
The set $\SSS^0(\Gamma)$ will play an especially prominent role 
(e.g.~see \cref{theo:CLTCDM} below); 
note that the origin $w_0=(0,0)$ does not belong to $\SSS^0(\Gamma)$.

Define the sets $\SSS^i_{\tiny\rightarrow}(\Gamma),
\SSS^i_{\tiny\searrow}(\Gamma), \SSS^i_{\tiny\nearrow}(\Gamma),
\SSS^i_n(\Gamma)$, $n\in\Z\setminus\{0\}$, analogously.

\begin{figure}[tbp]
\begin{center}
    \includegraphics[width = 0.99\linewidth]{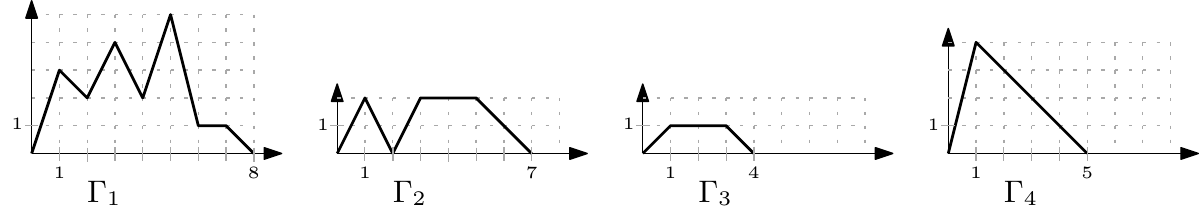}
	\end{center}
\caption{Four excursions $\Gamma_1,\Gamma_2,\Gamma_3$ and
$\Gamma_4$. Note that $\Gamma_3$ and $\Gamma_4$ are Łukasiewicz paths,
but $\Gamma_1$ and $\Gamma_2$ are not.}
\label{fig:Lukasiewicz}
\end{figure}

Summing up, we have the following natural decompositions into disjoint subsets:
\[ \SSS(\Gamma) = \bigcup_{i=0}^\infty \,\Bigl\{ \SSS^i_{\tiny\rightarrow}(\Gamma)\cup
  \SSS^i_{\tiny\searrow} (\Gamma)\cup \SSS^i_{\tiny\nearrow} (\Gamma) \Bigr\} =
\bigcup_{i=0}^\infty \,\Bigl\{\SSS^i_{\tiny\rightarrow}(\Gamma)\cup \bigcup_{n=1}^\infty
\bigl( \SSS^i_{-n} (\Gamma)\cup \SSS^i_n(\Gamma) \bigr) \Bigr\}.\]
For an ordered tuple $\vec{\Gamma} = (\Gamma_1,\dots,\Gamma_k)$ of $k$ excursions 
$\Gamma_i$, we will treat $\vec{\Gamma}$ itself as the excursion
obtained by concatenating $\Gamma_1,\dots,\Gamma_k$.
We define $\SSS(\vec{\Gamma}),
\SSS_{\tiny\rightarrow}^i(\vec{\Gamma}),
\SSS_n^i(\vec{\Gamma})$, etc.~in the same way as before, by treating 
$\vec{\Gamma}$ as the excursion obtained from the concatenation of $\Gamma_1,\dots,\Gamma_k$.
For example, if $\vec{\Gamma} = (w_0, w_1, \dots, w_\ell)$, then 
$\SSS(\vec{\Gamma}) = \{w_1, \dots, w_\ell\}$.
We say that \emph{$p = (w_i,w_j)$ is a pairing of degree $n > 0$} if $w_i \in \SSS_{-n}(\vec{\Gamma}), w_j\in\SSS_n(\vec{\Gamma})$, and $w_i$ appears
before $w_j$ in $\vec{\Gamma}$, i.e.~$i<j$. 
See~\cref{fig:Lukasiewicz} and~\cref{fig:LukasiewiczEx} for an example.

\begin{figure}[tbp]
\begin{center}
    \includegraphics[width = 0.99\linewidth]{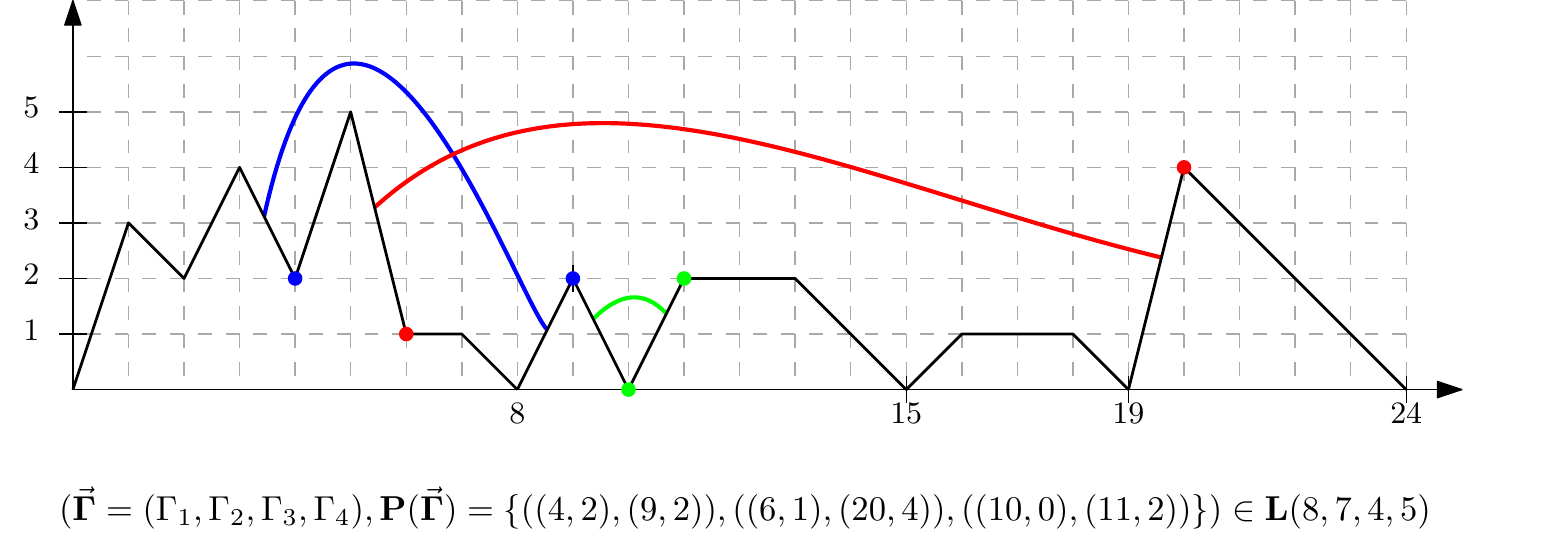}
\end{center}
\caption{A ribbon path $\ribbon = (\Gamma_1,\Gamma_2,\Gamma_3,\Gamma_4)$ on $4$ sites of lengths $8, 7, 4, 5$, which is also a Łukasiewicz ribbon path belonging to $\mathbf{L}(8, 7, 4, 5)$.
The three pairs of dots of different colors (red, green and blue)
are examples of pairings, and the colored arcs join 
the pairs of preceding steps of the same degree. The degrees of the red, green and blue
pairings are $4$, $2$ and $2$, respectively.}
\label{fig:LukasiewiczEx}
\end{figure}

By definition, a \emph{ribbon path on $k$ sites of lengths $\ell_1,\dots,\ell_k$} 
is a pair $\ribbon = (\vec{\Gamma},\,\mathbf{P}(\ribbon))$ consisting of an ordered tuple 
$\vec{\Gamma}$ of $k$ excursions $\Gamma_1,\dots,\Gamma_k$ of lengths $\ell_1, \dots, \ell_k$, 
respectively, and a set $\mathbf{P}(\ribbon)$ of pairwise disjoint
pairings $p_1, \dots, p_q$ on $\vec{\Gamma}$. This notion was first introduced by the third author in \cite{Moll2023}. We denote by $\mathbf{P}_n(\ribbon)\subset\mathbf{P}(\ribbon)$
the subset of pairings of degree $n$.
For any $n\in\Z\setminus\{0\}$, define 
$\SSS_n(\ribbon) := \SSS_n(\vec{\Gamma})\setminus\mathbf{P}_{|n|}(\ribbon)$ 
as the set obtained by removing the points belonging to $\mathbf{P}_{|n|}(\ribbon)$ from 
$\SSS_n(\vec{\Gamma})$; likewise, $\SSS_n^i(\ribbon) := \SSS_n^i(\vec{\Gamma}) \setminus \mathbf{P}_{|n|}(\ribbon)$.
Finally, define $\SSS(\ribbon) := \SSS(\vec{\Gamma})$ and $\SSS_{\to}^i(\ribbon) := \SSS_{\to}^i(\vec{\Gamma})$.
Then we have the following decomposition:\footnote{In this paragraph, 
we abused the notation: the set of pairings $\mathbf{P}_n(\ribbon)$ contains pairs of
  distinct points $(w_i,w_j)$, but we implicitly treated such pairs as the 2-element sets $\{w_i, w_j\}$, for simplicity of notation.}
\begin{equation}\label{decompos_S}
\SSS(\ribbon) = \bigcup_{n=1}^\infty\mathbf{P}_n(\ribbon)\cup \bigcup_{i=0}^\infty\,\Bigl\{ \SSS^i_{\tiny\rightarrow}(\ribbon)\cup
  \bigcup_{n=1}^\infty \bigl( \SSS^i_{-n} (\ribbon)\cup \SSS^i_{n}(\ribbon) \bigr) \Bigr\}.
\end{equation}

\begin{example}
The decomposition \eqref{decompos_S} for the ribbon path $\ribbon$ from \cref{fig:LukasiewiczEx} is given by:
\begin{align*}
\mathbf{P}_2(\ribbon) &= \bigl\{\bigl( (4,2),(9,2) \bigr), \bigl((10,0),(11,2)\bigr)\bigr\},\quad
\mathbf{P}_4(\ribbon) = \bigl\{ \bigl( (6,1),(20,4) \bigr) \bigr\},\\
\SSS^0_{-1}(\ribbon) &= \{(8,0), (15, 0), (19, 0), (24, 0)\},\\
\SSS^1_{\tiny\rightarrow}(\ribbon) &= \{(7,1),(17,1),(18,1)\},\quad
\SSS^1_{-1}(\ribbon) = \{(14,1),(23,1)\},\quad
\SSS^1_{1}(\ribbon) = \{(16,1)\},\\
\SSS^2_{\tiny\rightarrow}(\ribbon) &= \{(12,2),(13,2)\},\quad
\SSS^2_{-1}(\ribbon) = \{(2,2),(22,2)\},\\
\SSS^3_{-1}(\ribbon) &= \{(21,3)\},\quad \SSS^3_{3}(\ribbon) = \{(1,3)\},\quad
\SSS^4_{2}(\ribbon) = \{(3,4)\},\quad \SSS^5_{3}(\ribbon) = \{(5, 5)\},
\end{align*}
and the remaining sets in the decomposition from \cref{decompos_S} are
empty.
\end{example}

\begin{definition}[Łukasiewicz paths and Łukasiewicz ribbon paths]
{\ } \\
\indent (i) An excursion $\Gamma$ without horizontal steps at height $0$ 
and whose down steps all have degree $1$ 
is called a \textbf{Łukasiewicz path}.
In symbols, if $\SSS^0_{\tiny\rightarrow}(\Gamma)=\SSS_{-n}(\Gamma)=\emptyset$, for all $n\ge 2$, then $\Gamma$ is a Łukasiewicz path.
We denote $\mathbf{L}_0(\ell)$ the set of Łukasiewicz paths of length $\ell$.

(ii) A ribbon path $\ribbon$ without horizontal steps at height $0$ and 
whose non-paired down steps all have degree $1$ is called 
a \textbf{Łukasiewicz ribbon path}.
In other words, if $\SSS^0_{\tiny\rightarrow}(\ribbon)=\SSS_{-n}(\ribbon) = \emptyset$, for all $n\ge 2$, then $\ribbon$ 
is a Łukasiewicz ribbon path.
We denote $\mathbf{L}(\ell_1,\dots,\ell_k)$ the set of 
Łukasiewicz ribbon paths on $k$ sites of lengths $\ell_1,\dots,\ell_k$.
\end{definition}

Łukasiewicz paths are classical objects in combinatorics, see e.g.~\cite{FlajoletSedgewick2009}, 
but observe that, contrary to the literature, our definition \emph{does not allow horizontal steps at height $0$}.

\subsection{The LLN and CLT for Jack--Thoma measures}\label{subsec:LLNCLTCDM}

From~\cref{ex:CDMPois}, recall:

\begin{definition}\label{plancherel_jack}
Let $\alpha > 0$, $u>0$, $\vv=(v_1, v_2,\cdots)\in\R^\infty$ be such that $J_\la^\a(u\cdot\vv)\ge 0$, for all $\la\in\Y$. Then the \textbf{Jack--Thoma measure}, denoted  by $M^\a_{u;\vv}$, is the probability measure on the set of all Young
diagrams $\Y$ defined by
\begin{equation}\label{eq:DefOneJack}
M^\a_{u;\vv}(\la) := \exp\left(-\frac{u^2\cdot v_1}{\alpha}\right)\frac{J_\la^\a(u\cdot\vv)\cdot
u^{|\la|}}{j^\a_\la},\quad\forall\,\la\in\Y.
\end{equation}
Recall that $j^\a_\la$ is the quantity from \cref{j_lambda}.
\end{definition}

\begin{remark}
For any fixed $\alpha>0$, the parameters $u,\vv$, such that $J_\la^\a(u\cdot\vv)\ge 0$, for all $\la\in\Y$, are classified by \cref{thm:KOO}, due to Kerov--Okounkov--Olshanski.
This theorem generalizes Thoma's classification of extremal characters
of the infinite symmetric group~\cite{Thoma1964}, which is equivalent to the $\alpha=1$ case; 
this connection to Thoma's theorem motivated the name of the measures $M^\a_{u;\vv}$.
\end{remark}

\begin{example}\label{exam:size_jack_thoma}
As $M^\a_{u;\vv}$ is a special case Jack measure, we will be able to calculate various statistics of $M^\a_{u;\vv}$-random partitions by known techniques; see \cref{sec:proofs_CDM}.
For example, the simplest equation coming from \cref{theo:Expectations} shows that the expected value of $B_2(\Lambda_{(\alpha;u)})$, where $\Lambda_{(\alpha;u)}=T_{\frac{\al}{u},\frac{1}{u}}\la$ and $\la$ is $M^\a_{u;\vv}$-distributed, is equal to $v_1$.
Since $B_2(\Lambda_{(\alpha;u)}) = \frac{\alpha|\la|}{u^2}$, due to~\eqref{eq:values_X}, it follows that the expected value of the size of a $M^\a_{u;\vv}$-distributed random partition $\la$ is equal to:
\[
\E^\a_{u; \vv}\big[ |\la| \big] = \frac{u^2v_1}{\al}.
\]
\end{example}

From now on, let us assume in this section that $(\alpha(d))_{d\ge 1}$, $(u(d))_{d\ge 1}$ are
sequences of positive real numbers. Let $\alpha$ and $u$ denote 
the general $d$-dependent terms in the corresponding sequences: $\alpha = \alpha(d),\, u = u(d)$. We will impose the following assumption:

\begin{assumption}[Assumptions on parameters $g,g',\vv$, and sequences $\al(d),\,u(d)$]\label{main_assumption}
The sequences $(\alpha(d))_{d\ge 1}$, $(u(d))_{d\ge 1}$ of positive reals and the infinite tuple $\vv=(v_1,v_2,\cdots)\in\R^\infty$, with $v_1>0$, are such that 
$M^\a_{u;\vv}(\la)\geq 0$, for all $\la\in\Y,\,d\in\Z_{\ge 1}$.
Moreover, assume that there exist $g,g'\in\R$ such that one of the following two asymptotic conditions holds:

\smallskip

(i) Weak version:
\begin{equation}\label{eq:assumptions'_weak}
\frac{\al}{u^2} = o(1),\quad \frac{\al-1}{u} = g+o(1),\quad \text{as }d\to\infty.
\end{equation}

(ii) Refined version:
\begin{equation}\label{eq:assumptions'}
\frac{\al}{u^2} = \frac{1}{d} + o\left(\frac{1}{d}\right),\quad
\frac{\al-1}{u} = g+\frac{g'}{\sqrt{d}}+o\left(\frac{1}{\sqrt{d}}\right),\quad \text{as }d\to\infty.
\end{equation}
\end{assumption}

\begin{remark}\label{rem:Assumpt}
The weak version of \cref{main_assumption} is the minimal requirement needed for the LLN in~\cref{theo:LLNCDM} below, while the refined version is needed for the CLT in~\cref{theo:CLTCDM}. These conditions should be compared to the AFP and enhanced AFP from \cref{def:approx-factorization-charactersA}.
\end{remark}

In the next section, we investigate the existence of fixed parameters $g, g', \vv$ and 
sequences $\al(d),\, u(d)$ satisfying the refined version of \cref{main_assumption}, see e.g.~\cref{prop:CDMPar2}.

Our main theorems in this section study the limits of the rescaled random Young diagrams $\Lambda_{(\alpha;u\sqrt{v_1})}$, given by~\eqref{eq:Lambda2} and sampled w.r.t.~the Jack--Thoma measures, as $d\to\infty$.
Let us mention that the choice of scaling is so that the area of the rescaled Young diagram $\Lambda_{(\alpha;u\sqrt{v_1})}$ is equal to $1$, as seen from \cref{exam:size_jack_thoma}.
Moreover, note that the area of a rescaled box in $\Lambda_{(\alpha;u\sqrt{v_1})}$ equals $\,\frac{\al}{u^2v_1}$, 
whereas the difference between the box's width and height is $\,\frac{\al-1}{u\sqrt{v_1}}$; these points show that \cref{main_assumption} is geometrically motivated.
The following LLN and CLT are detailed versions of Theorems \ref{theo:LLNCDM_intro} and \ref{theo:CLTCDM_intro} from the introduction.

\begin{theorem}[Law of Large Numbers]\label{theo:LLNCDM}

Fix $g\in\R,\,\vv=(v_1, v_2,\cdots)\in\R^\infty$, and suppose that the weak version of \cref{main_assumption} is satisfied.
Then there exists a deterministic function $\omega_{\Lambda_{g;\vv}}\colon\R\to\R_{\ge 0}$ such that 
\begin{equation}\label{eq:limitCDM}
\lim_{d\to\infty}\omega_{\Lambda_{(\alpha;u\sqrt{v_1})}} = \omega_{\Lambda_{g;\vv}},
\end{equation}
with respect to the supremum norm, in probability.
In other words, for each $\epsilon>0$,
\[
\lim_{d\to\infty}M^\a_{u;\vv}\big( \|\omega_{\Lambda_{(\alpha;u\sqrt{v_1})}}-\omega_{\Lambda_{g;\vv}} \|_\infty > \epsilon \big) =0.
\]

The associated transition measure (via the Markov--Krein correspondence) $\mu_{g; \vv} := \mu_{\Lambda_{g;\vv}}$ is the unique probability measure on $\R$ with moments given by:
\begin{equation}\label{mu_moments}
\int_\R{x^\ell \mu_{g; \vv}(dx)} = v_1^{-\frac{\ell}{2}}\cdot\sum_{\Gamma\in\mathbf{L}_0(\ell)}
\,\prod_{i\ge 1} (i\cdot g)^{|\SSS_{\tiny\rightarrow}^i(\Gamma)|}\,v_i^{\,|\SSS_i(\Gamma)|},\quad\forall\,\ell\in\Z_{\ge 1}.
\end{equation}
\end{theorem}

\begin{theorem}[Central Limit Theorem]\label{theo:CLTCDM}
Fix $g,g'\in\R,\, \vv = (v_1, v_2,\cdots)\in\R^\infty$ with $v_1>0$, and suppose that the refined version of \cref{main_assumption} is satisfied.
Consider the sequence of random functions
\[ \Delta^\a_{u;\vv} :=
\frac{u}{\sqrt{\al}}\bigg(\frac{\omega_{\Lambda_{(\al;u\sqrt{v_1})}} - \omega_{\Lambda_{g;\vv}}}{2}\bigg). \]
and the polynomial observables $\displaystyle\langle \Delta^\a_{u;\vv},\, x^\ell \rangle := \int_{\R}{x^\ell \, \Delta^\a_{u;\vv}(x)dx}$, for $\ell\in\Z_{\ge 0}$.
Then the joint distribution of the random variables
$\langle \Delta^\a_{u;\vv},\, x^{\ell} \rangle$, for $\ell\in\Z_{\ge 0}$, converges weakly, as $d \to \infty$, to a Gaussian distribution with mean vector and covariance matrix given by the following explicit polynomials in $g,g',v_1,v_2,\dots$ with nonnegative coefficients:
\begin{equation}\label{MeanGaussian_2}
\lim_{d\to\infty}\E^\a_{u;\vv}\!\left[ \langle \Delta^\a_{u;\vv},\, x^{\ell-2} \rangle \right] 
= \frac{v_1^{-\frac{\ell}{2}}}{(\ell-1)}\!\times\!\frac{g'\cdot \partial}{\partial g}
\!\left( \sum_{\Gamma\in\mathbf{L}_0(\ell)}{\frac{1}{|\SSS^0(\Gamma)|}\prod_{i\ge 1}(i\cdot g)^{|\SSS_{\tiny\rightarrow}^i(\ribbon)|}\, v_i^{\,|\SSS_i(\ribbon)|}} \right)\!,
\end{equation}
\begin{multline}\label{CovarianceGaussian_2}
\lim_{d\to\infty}\Cov^\a_{u;\vv}\left( \langle \Delta^\a_{u;\vv},\, x^{k-2} \rangle, 
\langle \Delta^\a_{u;\vv},\, x^{l-2} \rangle \right) = \\
= \frac{v_1^{-\frac{k+l}{2}}}{(k-1) (l-1)}\sum_{\ribbon = (\Gamma_1, \Gamma_2)\in\mathbf{L}^{\conn}(k, l)\atop |\mathbf{P}(\ribbon)|=1}\frac{1}{|\mathbf{S}^0(\Gamma_1)||\mathbf{S}^0(\Gamma_2)|}\prod_{i\ge 1} 
(i\cdot g)^{|\SSS_{\tiny\rightarrow}^i(\ribbon)|}\,
i^{\,|\mathbf{P}_i(\ribbon)|} \, v_i^{\,|\SSS_i(\ribbon)|}.
\end{multline}
These formulas are valid for all $\ell,k, l = 2, 3, \cdots$. In~\eqref{CovarianceGaussian_2}, the sum is over ribbon paths $\ribbon=(\Gamma_1,\Gamma_2)\in\mathbf{L}(k,l)$ with exactly one pairing between a vertex from $\SSS(\Gamma_1)$ and a vertex from $\SSS(\Gamma_2)$ (for the definition of more general versions of the set $\mathbf{L}^{\conn}(k, l)$, see~\cref{subsubsec:Cumulants_step2}).
\end{theorem}

\begin{remark}
The third author studied in~\cite{Moll2023} a process analogous to
$\Delta^\a_{u;\vv}$ for Jack measures $M^\a_{\rho_1, \rho_2}$ with
identical specializations $\rho_1=\rho_2$ and proved a CLT.
His formula for the covariances of the Gaussian fluctuations are given 
via a `welding operator' $\mathcal{K}$ applied to quantities
determined by the LLN; for more general cases, he also gave certain formula involving signs, which
is very different from ours, e.g.~observe that \cref{CovarianceGaussian_2} 
is a cancellation-free polynomial with nonnegative coefficients.
Our arguments also work verbatim for Jack
measures under the assumptions of Moll, and provide explicit formulas akin to \eqref{CovarianceGaussian_2}
after replacing Łukasiewicz ribbon paths by all ribbon paths.
\end{remark}

\subsection{Proofs of Theorems \ref{theo:LLNCDM} and \ref{theo:CLTCDM}}\label{sec:proofs_CDM}

We use the following statistic on ribbon paths:
\[  F_{a,b,\vv}(\ribbon) := \prod_{i\ge 1}(i\cdot
  a)^{\!|\SSS^i_{\tiny\rightarrow}(\ribbon)|}(i\cdot
  b)^{\!|\mathbf{P}_i(\ribbon)|}\,v_i^{\,|\SSS_i(\ribbon)|},\ \text{ for variables } a,b.  \]

First we study the combinatorics of Nazarov--Sklyanin operators, which have been used before for random partitions in \cite{Huang2021,Moll2023}; however, our setting and results are new.

\subsubsection{Step 1 -- Expressing observables via Nazarov--Sklyanin operators}

Consider the infinite row vector
$P = (P_{1,{k}})_{{k} \in \Z_{\geq 1}}$ and column vector $P^\dagger =
(P^\dagger_{{k},1})_{{k} \in \Z_{\geq 1}}$, where $P_{1,{k}} := p_{{k}}$, and
$P^\dagger_{{k}, 1} := \alpha \cdot k\cdot\frac{\partial}{\partial p_{k}} =: p_{-{k}}$, are regarded as operators on the algebra of symmetric
functions $\Symm = \R[p_1, p_2, \cdots]$.
Let $L = (L_{i,j})_{i,j \in \Z_{\geq 1}}$ be the infinite matrix defined by
$L_{i,j} := p_{j-i}+\delta_{i,j}\,i(\alpha-1)$, for all $i, j\in\Z_{\geq 1}$, with the convention that $p_0:=0$:
$$
L = \begin{bmatrix} 
(\alpha-1) & p_1 & p_2 & \dots\\
p_{-1} & 2(\alpha-1) & p_1 & \ddots\\
p_{-2} & p_{-1} & 3(\alpha-1) & \ddots\\
\vdots & \ddots & \ddots & \ddots
\end{bmatrix}.
$$

\begin{theorem}[Nazarov--Sklyanin \cite{NazarovSklyanin2013}]\label{theo:Nazarov-Sklyanin}
For all $\ell\in\Z_{\geq 0}$ and all $\la\in\Y$:
\begin{equation*}
PL^\ell P^\dagger J_\lambda(\pp) = B_{\ell+2}(\Lambda_{(\alpha;1)}) \cdot J_\lambda(\pp).
\end{equation*}
\end{theorem}

\begin{proof}
  Nazarov and Sklyanin proved~\cite[Theorem 2]{NazarovSklyanin2013}
  that
  \[ \sum_{k \geq 0}z^{-k-1}PL^k P^\dagger J_\lambda(\pp)
    = \left(z-(z+\ell(\lambda))\prod_{i=1}^{\ell(\lambda)}\frac{z+i-1-\alpha\lambda_i}{z+i-\alpha\lambda_i}\right)J_\lambda(\pp).\]
Comparing the RHS of the equation
  above with the identity~\eqref{eq:Cauchy-St} and using
  \eqref{eq:Boolean} we can rewrite this RHS as
\[ \left(z-\frac{1}{G_{\mu_{T_{\alpha,1}\lambda}}(z)}\right)J_\lambda(\pp)
    = B_{\mu_{\Lambda_{(\alpha;1)}}}(z) J_\lambda(\pp),\]
which finishes the proof. 
\end{proof}

\begin{theorem}\label{theo:Expectations}
Fix arbitrary $\ell_1,\dots,\ell_n\in\Z_{\ge 1}$. The following expectation of products of Boolean cumulants, with respect to the Jack--Thoma measure $M_{u;\vv}^\a$, is given by:
\begin{equation}\label{eq:ExpectOfBoolean}
\E^\a_{u; \vv}\left(B_{\ell_1}(\Lambda_{(\al;u)})\cdots B_{\ell_n}(\Lambda_{(\al;u)})\right)
= \sum_{\ribbon\in\mathbf{L}(\ell_1,\dots,\ell_n)\atop |\SSS^0(\ribbon)|=n}
{F_{\,\frac{\al-1}{u},\,\frac{\al}{u^2},\,\vv}(\ribbon)}.
\end{equation}
\end{theorem}

\begin{proof}
First, observe that if some $\ell_i=1$, then $B_{\ell_i}(\Lambda_{(\alpha; u)}) = 0$, so the LHS of \eqref{eq:ExpectOfBoolean} vanishes. On the other hand, the set $\mathbf{L}(\ell_1,\dots, 1, \dots,\ell_n)$ is empty, because the only excursion of length $1$ is $((0, 0), (1, 0))$, but Łukasiewicz ribbon paths are not allowed to have horizontal steps at height $0$.
As a result, the RHS side of \eqref{eq:ExpectOfBoolean} also vanishes.

In the remainder of the proof, assume that $\ell_1, \dots, \ell_n\ge 2$.
Note that
    \begin{align*}
      &\mathbb{E}^\a_{u;\vv}\left(B_{\ell_1}(\Lambda_{(\alpha;u)})\cdots
        B_{\ell_n}(\Lambda_{(\alpha;u)})\right) =\\
      &=\exp\left(-\frac{u^2\cdot v_1}{\alpha}\right) 
\sum_{\lambda}\frac{\left(B_{\ell_1}(\Lambda_{(\alpha;u)})\cdots
        B_{\ell_n}(\Lambda_{(\alpha;u)})\right) \cdot
      J_{\la}^\a(u\cdot\vv)\cdot u^{|\la|}}{j^\a_\la}\\
      &=\exp\left(-\frac{u^2\cdot v_1}{\al}\right) u^{-\ell_1-\cdots-\ell_n} 
PL^{\ell_1-2}P^\dagger\cdots PL^{\ell_n-2}P^\dagger
\sum_{\lambda}\frac{J_\la^\a(\pp)\cdot u^{|\la|}}
        {j^\a_\la}\bigg|_{\pp = u\cdot\vv}\\
      &=\exp\left(-\frac{u^2\cdot v_1}{\al}\right)
        u^{-\ell_1-\cdots-\ell_n}PL^{\ell_1-2}P^\dagger\cdots PL^{\ell_n-2}P^\dagger\exp\left(\frac{u\cdot p_1}{\alpha}\right)\bigg|_{\pp = u\cdot\vv},
\end{align*}
where the first equality comes from the definition of the expectation and the definition~\eqref{eq:Lambda2} of $\Lambda_{(\alpha;u)}$, the second one comes from \cref{theo:Nazarov-Sklyanin} and~\eqref{eq:scaling}, and the last equality follows from the Cauchy identity~\eqref{eq:Cauchy}.
The final desired equality
\begin{equation}\label{desired_concatenation}
e^{-\frac{u^2\cdot v_1}{\al}}\cdot
u^{-\ell_1-\cdots-\ell_n}\cdot PL^{\ell_1-2}P^\dagger\cdots PL^{\ell_n-2}P^\dagger\, 
\!\!\left(e^{\frac{u\cdot p_1}{\al}}\right)\!\bigg|_{\pp = u\cdot\vv}
\!=\!\sum_{\ribbon\in\mathbf{L}(\ell_1,\dots,\ell_n),\atop |\SSS^0(\ribbon)|=n}
{\!F_{\,\frac{\al-1}{u},\,\frac{\al}{u^2},\,\vv}(\ribbon)}
\end{equation}
follows from the following combinatorial interpretation of each operator $PL^{\ell_i-2}P^\dagger$ in terms of excursions: for any integer $\ell\ge 2$, we have
\[ PL^{\ell-2}P^\dagger = \sum_{\Gamma\colon \ell(\Gamma)=\ell, \atop |\SSS^0(\Gamma)|=1}{L(\Gamma)}, \]
where the sum is over excursions $\Gamma = (w_0=(0,y_0),\dots,w_\ell=(\ell,y_\ell))$ of length $\ell$, 
and $L(\Gamma)$ is the operator $L(\Gamma) := p_{y_1}L_{y_1,y_2}\cdots
L_{y_{\ell-2},y_{\ell-1}}p_{-y_{\ell-1}}$, and the condition $|\SSS^0(\Gamma)| = 1$ indicates that the excursion touches the $x$-axis only at the beginning $w_0=(0,0)$ and the end $w_\ell=(\ell,0)$.
The product of operators $PL^{\ell_1-2}P^\dagger\cdots
PL^{\ell_n-2}P^\dagger$ in the LHS of \eqref{desired_concatenation}
then equals by concatenation to a sum of operators $L(\Gamma_1)\cdots L(\Gamma_n)$ over excursions $\vec{\Gamma} =
(\Gamma_1,\dots,\Gamma_n)$ of length
$\ell_1+\cdots+\ell_n$ touching the $x$-axis at exactly the $n$ prescribed
positions: $\ell_1,\ell_1+\ell_2,\dots,\ell_1+\cdots+\ell_n$.

Next, we would like to express each operator $L(\Gamma_1)\cdots L(\Gamma_n)$ in the
normal ordered form, i.e.~we would like to express it as a linear
combination of the operators of the form $p_{i_1}\cdots p_{i_k}
p_{-j_1}\cdots p_{-j_{k'}}$, where $i_1,\dots,i_k, j_1,\dots,j_{k'}
\geq 1$. Note that the commutation relation
$[p_j,p_i] = \delta_{i+j,0}\cdot i\al$, equivalently $p_jp_i=p_ip_j+\delta_{i+j,0}\cdot i\al$, allows to move all the
operators $p_{-i}$ with a negative index to the right, and the only
nontrivial commutation relation is $p_{-i}p_i=p_ip_{-i}+i\al$, which occurs when we try to move $p_{-i}$ past $p_i$.
These two operators are associated with a
down step and a following up step of the same increments, and we
assign to such pair a pairing. Marking all the possible pairings in
$\vec{\Gamma}$ and remembering that each of them produces two
terms, namely $p_i p_{-i}$ and  $i \al$, leads to the following combinatorial
interpretation of the operator $L(\Gamma_1)\cdots L(\Gamma_n)$
in normal ordered form:
\begin{equation}\label{eq:normalOrder}
L(\Gamma_1)\cdots L(\Gamma_n) = \sum_{\mathbf{P}(\ribbon)}
\prod_{i\ge 1}(i\cdot (\al-1))^{\!|\SSS^i_{\tiny\rightarrow}(\ribbon)|}(i\cdot\al)^{\!|\mathbf{P}_i(\ribbon)|}\,p_i^{\,|\SSS_i(\ribbon)|}
\prod_{i\ge 1}\,p_{-i}^{\,|\SSS_{-i}(\ribbon)|},
\end{equation}
where the sum is over all possible sets of pairings.
Finally, note that
\[
\prod_{i\ge 1}\,p_{-i}^{\,|\SSS_{-i}(\ribbon)|}\!\left(e^{\frac{u\cdot p_1}{\al}}\right)\!\bigg|_{\pp = u\cdot\vv}  = \begin{cases}
e^{\frac{u^2\cdot v_1}{\al}} u^{\, |\SSS_{\tiny\searrow}(\ribbon)|},&\text{ if } \SSS_{\tiny\searrow}(\ribbon) = \SSS_{-1}(\ribbon),\\
0,&\text{ otherwise. }
\end{cases}
\]
In particular, the ribbon paths $\ribbon$ with non-paired down steps
of degree $\ge 2$ give null contributions, therefore the LHS of
\eqref{desired_concatenation} can be rewritten as a sum over the set
$\mathbf{L}(\ell_1,\dots,\ell_n)$ of Łukasiewicz ribbon paths with the
associated operators written in the normal ordered form as in \eqref{eq:normalOrder}.
For $\ribbon=(\vec{\Gamma},\,\mathbf{P}(\ribbon))\in\mathbf{L}(\ell_1,\dots,\ell_n)$,
with $\vec{\Gamma}=(\Gamma_1,\dots,\Gamma_n)$, one has that 
$L(\Gamma_1)\cdots L(\Gamma_n)\left(e^{\frac{u\cdot p_1}{\al}}\right)\!\big|_{\pp = u\cdot\vv}$ is equal to
\begin{multline*}
\sum_{\mathbf{P}(\ribbon)} \prod_{i\ge 1}(i\cdot (\al-1))^{\!|\SSS^i_{\tiny\rightarrow}(\ribbon)|}(i\cdot\al)^{\!|\mathbf{P}_i(\ribbon)|}\,p_i^{\,|\SSS_i(\ribbon)|}\prod_{i\ge 1}\,p_{-i}^{\,|\SSS_{-i}(\ribbon)|}
\!\left(e^{\frac{u\cdot p_1}{\al}}\right)\!\bigg|_{\pp = u\cdot\vv}\\
= \sum_{\mathbf{P}(\ribbon)} \prod_{i\ge 1}(i\cdot (\al-1))^{\!|\SSS^i_{\tiny\rightarrow}(\ribbon)|}(i\cdot\al)^{\!|\mathbf{P}_i(\ribbon)|}\,v_i^{\,|\SSS_i(\ribbon)|}\cdot e^{\frac{u^2\cdot v_1}{\al}}u^{|\SSS_{\tiny\searrow}(\ribbon)|+|\SSS_{\tiny\nearrow}(\ribbon)|}\\
= \sum_{\mathbf{P}(\ribbon)} F_{\,\frac{\al-1}{u},\,\frac{\al}{u^2},\,\vv}(\ribbon)\cdot
e^{\frac{u^2\cdot v_1}{\al}} u^{|\SSS_{\tiny\searrow}(\ribbon)|+|\SSS_{\tiny\rightarrow}(\ribbon)|+|\SSS_{\tiny\nearrow}(\ribbon)|+2|\mathbf{P}(\ribbon)|}\\
= \sum_{\mathbf{P}(\ribbon)} F_{\,\frac{\al-1}{u},\,\frac{\al}{u^2},\,\vv}(\ribbon)\cdot e^{\frac{u^2\cdot v_1}{\al}}u^{\ell_1+\cdots+\ell_n}.
\end{multline*}
This leads to the expression in the RHS of \eqref{desired_concatenation}, as claimed.
\end{proof}

\subsubsection{Step 2 -- Formulas for cumulants of observables of Jack--Thoma measures}\label{subsubsec:Cumulants_step2}

      Let $X_1,\dots,X_n$ be random variables on the same probability space. 
	We recall that the classical (joint) cumulant $\kappa_n(X_1,\dots,X_n)$ is 
	defined combinatorially through the following moment-cumulant formula:
\begin{equation}
  \label{eq:cumu}
  \E(X_1\cdots X_n) := \sum_{\pi \in \SP([1\,..\,n])}\prod_{B \in
      \pi}\kappa(X_b\colon b \in B),
\end{equation}
where $[1\,..\,n] := \{1,2,\dots,n\}$, and for any finite set $X$ we denote by
$\SP(X)$ the set of \emph{set-partitions} of $X$, 
namely $\pi\in\SP(X)$ means that $\pi$ is a set of disjoint subsets of $X$ (called the \emph{blocks of $X$}) 
whose union is the whole set $X$. 

Next, we give a slightly refined notion of 
\emph{connected ribbon paths}, originally introduced in~\cite{Moll2023}.
Let $\ribbon = (\Gamma_1, \dots, \Gamma_n)$ be a ribbon path on $n$ sites and $\pi \in \SP([1\,..\,n])$
be a set-partition. We define $G^\pi_{\ribbon}$ as the following simple graph: the vertex 
set is $\pi$ (so that vertices are indexed by the blocks of $\pi$) and an edge joins 
$B,B' \in \pi$ if and only if there exists at least one pairing $p = (v,w)\in\mathbf{P}(\ribbon)$ 
such that $v\in \SSS(\Gamma_b), w \in \SSS(\Gamma_{b'})$, 
or $w\in \SSS(\Gamma_b), v \in\SSS(\Gamma_{b'})$, 
for some $b \in B,\, b' \in B'$.
We say that a ribbon path $\ribbon$ is \emph{$\pi$-connected} if the associated graph $G^\pi_{\ribbon}$ is connected. If $\pi = \{\{1\},\dots,\{n\}\}$ and $\ribbon$ is $\pi$-connected, we simply say that $\ribbon$ is \emph{connected}. 
See \cref{fig:CumulantGraphs} for some examples.

\begin{figure}[tbp]
\begin{center}
\includegraphics[width = 0.8\linewidth]{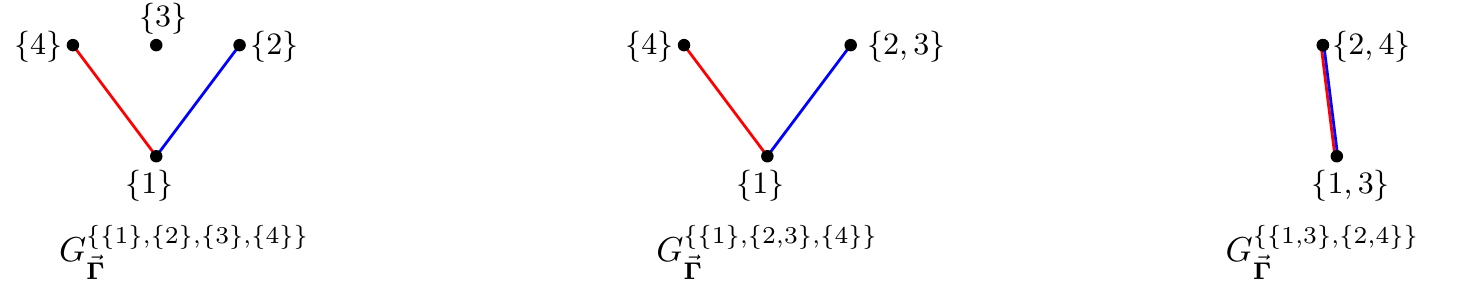}
\end{center}
\caption{The first graph depicts $G^{\{\{1\}, \{2\}, \{3\}, \{4\}\}}_{\ribbon}$ for the ribbon path
  $\ribbon$ from \cref{fig:LukasiewiczEx}; observe that 
$\ribbon$ is not connected. The next two graphs $G^\pi_{\ribbon}$ 
correspond to other set-partitions $\pi \in \SP([1\,..\,4])$; in both cases, $\ribbon$ is $\pi$-connected.}
\label{fig:CumulantGraphs}
\end{figure}

Recall that $\mathbf{L}(\ell_1,\dots,\ell_n)$ denotes the set of
Łukasiewicz ribbon paths on $n$ sites of lengths
$\ell_1,\dots,\ell_n$. We denote by $\mathbf{L}^{\pi-\conn}(\ell_1,\dots,\ell_n)$
(resp.~$\mathbf{L}^{\conn}(\ell_1,\dots,\ell_n)$) the subset of 
$\mathbf{L}(\ell_1,\dots,\ell_n)$ of those Łukasiewicz ribbon paths 
that are $\pi$-connected (resp.~connected).

\begin{proposition}\label{prop:kumuLukasiewicz}
    Fix arbitrary $\ell^1_1,\dots,\ell^1_{k_1}, \ell^2_1,\dots,\ell^2_{k_2},
    \dots,  \ell^n_1,\dots,\ell^n_{k_n}\in\Z_{\ge 1}$. 
We have the following formula for a cumulant with respect to the Jack--Thoma measure $M^\a_{u;\vv}$:
\begin{equation}
\kappa^\a_{u;\vv}\!\left(\!B_{\ell^1_1}(\Lambda_{(\al;u)})\cdots
      B_{\ell^1_{k_1}}\!\!(\Lambda_{(\al;u)}),\dots, B_{\ell^n_1}(\Lambda_{(\alpha;u)})\cdots
      B_{\ell^n_{k_n}}\!\!(\Lambda_{(\al;u)})\!\right) \!=\!
\sum_{\ribbon}{F_{\frac{\al-1}{u},\frac{\al}{u^2},\vv}(\ribbon)}.
\end{equation}
The sum above is over $\ribbon\in\mathbf{L}^{\pi-\conn}(\ell^1_1,\dots,\ell^1_{k_1},\dots, \ell^n_1,\dots,\ell^n_{k_n})$, with $|\SSS^0(\ribbon)| = k_1+\cdots+k_n$, and where $\pi\in\SP([1\,..\,k_1+\cdots+k_n])$ is the set-partition with $n$ blocks defined by $\pi := \{[1\,..\,k_1],[k_1+1\,..\,k_1+k_2],\dots,[k_1+\cdots+k_{n-1}+1\,..\,k_1+\cdots+k_{n}]\}$, and we employed the usual notation $[a\,..\,b] := \{a, a+1, \dots, b\}$, for any integers $a\le b$.
\end{proposition}

\begin{proof}
Fix the set-partition $\pi :=
\{[1\,..\,k_1],[k_1+1\,..\,k_1+k_2],\dots,[k_1+\cdots+k_{n-1}+1\,..\,k_1+\cdots+k_{n}]\}$ of $[1\,..\,k_1+\cdots+k_n]$.
	Let $\ribbon\in\mathbf{L}(\ell^1_1,\dots,\ell^1_{k_1},\dots,
    \ell^n_1,\dots,\ell^n_{k_n})$ be any Łukasiewicz ribbon path.
	Then we can define $\sigma = \sigma(\ribbon)\in\SP(\pi)$ as the set-partition whose blocks are
    sets of vertices of the connected components of
    $G^{\,\pi}_{\ribbon}$. For each $B \in \sigma$, let
    $\ribbon_B$ be the concatenation of excursions 
    $\Gamma_b$, $b \in \bigcup B$, with respect to the natural order. Since $\SP(\pi)$ is isomorphic
    to $\SP([1\,..\,n])$, we will denote by $\tilde{B}$ the image of a
    block $B \in \sigma$ by the natural map
    $[k_1+\cdots+k_{i-1}+1\,..\,k_1+\cdots+k_i] \to i$.
	By construction, it is evident that each 
    $\ribbon_B$ is a well-defined, $\pi|_B$-connected Łukasiewicz
    ribbon path $\ribbon_B \in
    \mathbf{L}^{\pi|_B-\conn}(\ell^b_1,\dots,\ell^b_{k_b}\colon b \in \tilde{B})$.

    Conversely, let $\sigma\in\SP(\pi)$ be arbitrary, and for any block $B\in\sigma$ 
    assume that we have a $\pi|_B$-connected Łukasiewicz ribbon path $\ribbon_B \in
    \mathbf{L}^{\pi|_B-\conn}(\ell^b_1,\dots,\ell^b_{k_b}\colon b \in \tilde{B})$.
	Then one can construct a (non-necessarily $\pi$-connected)
    Łukasiewicz ribbon path $\ribbon\in\mathbf{L}(\ell^1_1,\dots,\ell^1_{k_1},\dots,
    \ell^n_1,\dots,\ell^n_{k_n})$ as a concatenation of
    the excursions of lengths $\ell_1^1,\dots,\ell_{k_n}^n$ that make up the $\ribbon_B$'s,
	and the set of pairings of $\ribbon$ is defined as the union of sets of pairings for $\ribbon_B$'s, for 
	$B\in\sigma$.

Note that $F_{\,\frac{\al-1}{u},\,\frac{\al}{u^2},\,\vv}(\ribbon) = \prod_{B\in\sigma}{F_{\,\frac{\al-1}{u},\,\frac{\al}{u^2},\,\vv}(\ribbon_B)}$.
Using \cref{theo:Expectations}, this gives:
      \begin{multline}
    \mathbb{E}^\a_{u;\vv}\left(B_{\ell^1_1}(\Lambda_{(\alpha;u)})\cdots
      B_{\ell^1_{k_1}}(\Lambda_{(\alpha;u)})\cdots B_{\ell^n_1}(\Lambda_{(\alpha;u)})\cdots
      B_{\ell^n_{k_n}}(\Lambda_{(\alpha;u)})\right)\\
    =  \sum_{\sigma \in \SP(\pi)}\prod_{B \in
    \sigma}\sum_{\,\ribbon_B \in
    \mathbf{L}^{\pi|_B-\conn}(\ell^b_1,\dots,\ell^b_{k_b}\colon b \in
    \tilde{B})\atop |\SSS^0(\ribbon_B)| = \sum_{b \in \tilde{B}}k_b}{F_{\,\frac{\al-1}{u},\,\frac{\al}{u^2},\,\vv}(\ribbon_B)}.
\end{multline}
Comparing this equation with \eqref{eq:cumu} that defines cumulants, we conclude the proof.
\end{proof}

\begin{corollary}\label{cor_formulas_cums}
For arbitrary $\ell_1,\dots,\ell_n\in\Z_{\ge 1}$, we have 
\begin{equation}\label{eq:Cumu-Ss}
\kappa^\a_{u; \vv} \left(S_{\ell_1}(\Lambda_{(\alpha;u)}),\dots, S_{\ell_{n}}
(\Lambda_{(\alpha;u)})\right) =
\sum_{\ribbon\in\mathbf{L}^{\conn}(\ell_1,\dots,\ell_{n})\atop \ribbon = (\Gamma_1, \dots,\Gamma_n)}
\,\prod_{j=1}^n \frac{1}{|\mathbf{S}^0(\Gamma_j)|}\cdot F_{\,\frac{\al-1}{u},\,\frac{\al}{u^2},\,\vv}(\ribbon).
\end{equation}
\end{corollary}
\begin{proof}
It follows immediately from \cref{prop:kumuLukasiewicz} by using
multilinearity of the cumulant $\kappa^\a_{u;\vv}$ and \eqref{eq:S-Boolean}.
\end{proof}

\subsubsection{Step 3 -- Concluding the proofs of Theorems \ref{theo:LLNCDM} and \ref{theo:CLTCDM}}\label{subsubsec:Cumulants}

We start with the following:

\begin{claim}\label{claimA}
Any Łukasiewicz ribbon path $\ribbon\in\mathbf{L}(\ell_1,\dots,\ell_n)$ stays below height $\ell_1+\cdots+\ell_n$, i.e.~$\SSS^i(\ribbon) = \emptyset$, whenever $i \geq \ell_1+\cdots+\ell_n$.
Moreover, $\mathbf{L}(\ell_1,\dots,\ell_n)$ is finite.
\end{claim}
\begin{proof}
The first statement easily follows from the fact that the non-paired
down steps of any Łukasiewicz ribbon path have degree $1$, and we leave the details to the reader. The first statement implies that the set $\mathbf{L}(\ell_1,\dots,\ell_n)$ has cardinality upper bounded by $\big(2(\ell_1+\cdots+\ell_n)+1\big)^{\ell_1+\cdots+\ell_n}$,	and in particular is finite.
\end{proof}

\vspace{3pt}
\textbf{\emph{End of proof of~\cref{theo:LLNCDM}.}}
For the LLN, i.e.~\cref{eq:limitCDM}, we first prove a general result that might be useful for studying the asymptotic behavior of random partitions in other contexts.

Let $\mathcal{M}_{< \infty} (\R)$ be the set of positive finite measures on
$\R$ that have finite moments of all orders. Recall that $M_\ell(\mu) := \int_\R{x^\ell d\mu(x)}$ denotes the $\ell$-th moment of the measure $\mu$.

\begin{theorem}[\cite{Billingsley1995}, Thm.~30.2]\label{theo:Billingsley}
Let $(\mu_d)_{d \geq 1} \subset \mathcal{M}_{< \infty}(\R)$ be any sequence and $\mu\in\mathcal{M}_{<\infty}(\R)$ be a measure that is uniquely determined by its moments. Suppose that $\lim_{d \to \infty} M_\ell(\mu_d) =  M_\ell(\mu)$ for every $\ell \geq 0$, then $\mu_d$ converges to $\mu$ weakly.
\end{theorem}

\begin{remark}
The above theorem is stated in~\cite{Billingsley1995} for
probability measures, but the proof in~\cite{Billingsley1995}
works for $\mathcal{M}_{<\infty}(\R)$ in the same way.
\end{remark}

\begin{theorem}\label{theo:GeneralLLN}
  Let $(f_d)_{d \geq 1}$ be a sequence of random $C$-Lipschitz
  functions $f_d\colon\R\to\R_{\ge 0}$ such that $f_d(x)dx \in \mathcal{M}_{<
    \infty}(\R)$. Suppose that
\begin{enumerate}[label=(\alph*)]
    \item for every $\eps>0$, there exists $N>0$ such that $\lim_{d\to\infty}\PP_d\left( \|f_d\|_{L^\infty(\R\setminus [-N,N])}>\epsilon\right) = 0$,
    \item there exists a measure $\mu\in\mathcal{M}_{<\infty}(\R)$, uniquely determined by its moments, such that $f_d(x)dx$ converges to $\mu$ in moments, in probability, i.e.~for every $\eps>0$ and $\ell\in\Z_{\ge 0}$, we have
\[ \lim_{d\to\infty}{\PP_d\left(\left| \int_{\R}{x^\ell f_d(x)dx} - M_\ell(\mu) \right|>\eps\right)} = 0. \]
\end{enumerate}
Then, $(f_d)_{d \geq 1}$ has a deterministic limit $f\colon\R\to\R_{\ge 0}$ in the supremum norm, in probability:
\begin{equation*}
\forall \epsilon > 0, \qquad \lim_{d \to \infty}\PP_d\left( \|f_d - f\|_{\infty}> \epsilon\right) = 0.
\end{equation*}
The function $f$ is $C$-Lipschitz, and is a probability density of the measure $\mu$.
\end{theorem}

\begin{proof}
\emph{Step 1.} Let $\eps>0,\, \psi\in C_b(\R)$ be arbitrary.
In this step, we show that (b) implies 
\begin{equation}\label{limit_observable}
\lim_{d\to\infty}{\PP_d\left(\left|\int_\R{\psi(x)f_d(x)dx} - \int_\R{\psi(x)d\mu(x)} \right|\leq\eps \right)} = 1.
\end{equation}

Consider the following metric $\pi$ on $\mathcal{M}_{<\infty}(\R)$:
\[ \pi(\nu_1,\nu_2) := \sum_{i=0}^\infty\frac{\min\left( 1,\, |M_i(\nu_1)-M_i(\nu_2)| \right)}{2^i}. \]
From the definition of $\pi$, it follows that for any $\eps>0$, there exist $\eps'>0,\, k\in\Z_{\ge 1}$ such that:
\[  \text{ if $\nu\in\mathcal{M}_{<\infty}(\R)$ satisfies }\ |M_1(\nu) - M_1(\mu)|, \dots, |M_k(\nu) - M_k(\mu)| \le\eps' \Longrightarrow \pi(\nu, \mu) \le \eps. \]

As a result, $\pi$ metrizes convergence in moments, so that the assumption (b) is equivalent to the following condition:
for every $\eps>0$, we have
\begin{equation}\label{limit_mom_metric}
\lim_{d\to\infty}{\PP_d\left( \pi(f_d(x)dx,\mu)\le\eps \right)} = 1.
\end{equation}

A classical theorem of Le\'vy says that the weak convergence in
$\mathcal{M}_{< \infty}(\R)$ is metrizable (see~\cite{Varadarajan1958}
for a more general result); let $\tilde{\pi}$ denote the corresponding
metric. In other terms, for any sequence of measures 
$(\nu_n)_{n\ge 1}\subset\mathcal{M}_{<\infty}(\R)$ one has $\nu_n\to\mu$ weakly, if and only if
$\tilde{\pi}(\nu_n,\mu) \to 0$, as $n \to \infty$. Consequently, \eqref{limit_observable} would be implied if for every $\eps'>0$, we have
\begin{equation}\label{limit_levy_metric}
\lim_{d\to\infty}{\PP_d\left( \tilde{\pi}(f_d(x)dx,\mu)\le\eps' \right)} = 1.
\end{equation}
Note that \cref{theo:Billingsley}
implies that for any $\eps'>0$, there exists $\eps>0$ such that
$\pi(\nu,\mu)\le\eps \Longrightarrow\tilde{\pi}(\nu,\mu)\le\eps'$.
Therefore $\PP_d\left( \pi(f_d(x)dx,\mu)\le\eps \right)\le\PP_d\left(
  \tilde{\pi}(f_d(x)dx,\mu)\le\eps' \right)$, and
\eqref{limit_mom_metric} implies the desired limit~\eqref{limit_levy_metric}. 

\smallskip
\emph{Step 2.} Fix $x_0 \in \R$ and let $\psi_\epsilon^{x_0}\in C_c(\R)$ be a nonnegative function supported on $[x_0-\eps, x_0+\eps]$ with $\int_\R \psi_\eps^{x_0}(x) dx = 1$.
  For any $C$-Lipschitz function $g$, one has
\begin{equation*}
\left|g(x_0)-\int_\R{\psi_\eps^{x_0}(x)g(x)dx}\right| \leq \!\int_\R{\psi_\eps^{x_0}(x)|g(x_0)-g(x)|dx} \leq C\eps,
\end{equation*}
therefore
\begin{equation*}
\left|g(x_0) - \!\int_\R \psi_\eps^{x_0}(x) d\mu(x)\right| \leq 
C\epsilon+\left|\int_\R{\psi_\eps^{x_0}(x)g(x)dx} - \int_\R{\psi_\epsilon^{x_0}(x)d\mu(x)}\right|.
\end{equation*}
When we take $g = f_d$ we obtain
\begin{multline*}
\PP_d\left( \left|f_d(x_0) - \int_\R \psi_\eps^{x_0}(x) d\mu(x)\right|\le (C+1)\eps\!\right)\\
\geq \PP_d\left(\left|\int_\R\psi_\eps^{x_0}(x)f_d(x)dx - \int_\R \psi_\eps^{x_0}(x) d\mu(x)\right|\le \eps \right).
\end{multline*}
Step 1 shows that this RHS above converges to $1$, as $d\to\infty$.
Consequently, for any $\eps>0$, we have:
$\lim_{d\to \infty}\PP_d\left( |f_d(x_0) - \int_\R \psi_\epsilon^{x_0}(x) d\mu(x)| \leq
(C+1)\epsilon\right) = 1$. Passing to the limit $\epsilon \to 0$, we get that
  \[
    \lim_{d \to \infty} f_d(x_0) =
    f(x_0) := \lim_{\epsilon \to 0}\int_\R \psi_\epsilon^{x_0}(x)
    d\mu(x),\]
  where the first equality holds in probability.
  Since $x_0$ was arbitrary, we have a pointwise convergence $f_d \to
  f$, in probability. This pointwise convergence can be extended to the
  uniform convergence on any
   closed interval, in probability, by a standard
   argument of covering a closed interval by finitely many small
   intervals, and the fact that that the supremum norm of $|f_d-f|$ on a
   small interval is small by the Lipschitz condition. Therefore for every $N, \eps >0$
\[ \lim_{d \to \infty}\PP_d\left( \|f_d - f\|_{L^\infty[-N,N]}> \eps\right) = 0.\]

Finally, fix any $\eps>0$ and take $N>0$ as in condition (a) from the hypothesis; then 
\[ \PP_d\left( \|f_d - f\|_{\infty} > \epsilon\right) \leq \PP_d\left( \|f_d - f\|_{L^\infty[-N,N]} > \epsilon\right) +
  \PP_d\left( \|f_d - f\|_{L^\infty(\R\setminus[-N,N])} >
    \epsilon\right) \to 0 \]
as $d\to\infty$. Since the functions $f_d$ are $\ge 0$ and $C$-Lipschitz, $f$ has the same properties.

\smallskip
\emph{Step 3.} It remains to show that $f$ is the density of $\mu$. As
$\mu$ is uniquely determined by its moments, it suffices to verify
$\int_{\R}{x^\ell f(x)dx} = M_\ell(\mu)$, for all $\ell\in\Z_{\ge
  0}$. Fix $\ell\ge 0$ and $R>0$,
\begin{gather*}
\left|\int_{|x|\ge R}{\!x^\ell f_d(x)dx}\right| \le
\int_{|x|\ge R}{\!|x|^\ell f_d(x)dx} \le
\frac{1}{R^{\ell+2}}\!\int_{|x|\ge R}{\!x^{2\ell+2}f_d(x)dx}\\
\le \frac{1}{R^{\ell+2}}\int_\R{x^{2\ell+2}f_d(x)dx}\le
\frac{1}{R^{\ell+2}} \left| \int_\R{x^{2\ell+2}f_d(x)dx} - M_{2\ell+2}(\mu) \right| + \frac{M_{2\ell+2}(\mu)}{R^{\ell+2}}.
\end{gather*}
This inequality and condition (b) show that for any $\eps>0$, there exists $R_0>0$ such that 
\begin{equation}
  \label{eq:PomocDensity}
  \lim_{d\to\infty}{\PP_d\left( \left|\int_{|x|\ge R}{x^\ell
          f_d(x)dx}\right| < \eps\right)} = 1, \quad \forall\,
  R>R_0.
  \end{equation}
Note that $\left| M_\ell(\mu) - \int_{|x|\le R}{x^\ell f(x)dx}
\right|$ is bounded from above by
\[ \bigg|M_\ell(\mu) - M_\ell(f_d(x)dx)\bigg| + \left|\int_{|x|\ge R}{x^\ell
          f_d(x)dx}\right|+\left| \int_{-R}^{R}x^\ell(f_d(x)-f(x))dx\right|.\]
For any $\epsilon >0$, there exists $R_0>0$ such that for every $R > R_0$
the probability $\PP_d$ that each of these terms is smaller than
$\epsilon$ tends to $1$, as $d \to \infty$. Indeed, for the first term this is condition
(b), for the second this is~\eqref{eq:PomocDensity}, and for the third
one this is convergence of $f_d\to f$ uniformly on $[-R, R]$ in
probability, proved in Step 2. Taking the limit as $R\to\infty$, it follows that the integral $\int_\R{x^\ell f(x)dx}$ converges and equals $M_\ell(\mu)$, as desired.
\end{proof}

\medskip

We are going to prove~\cref{theo:LLNCDM} by
applying~\cref{theo:GeneralLLN} to the collection of functions
\[f_{(\al;u\sqrt{v_1})}(x) :=
\frac{\omega_{\Lambda_{(\al;u\sqrt{v_1})}}(x)-|x|}{2},\quad x\in\R, \]
that are parametrized by an integer $d\ge 1$ (recall that $\al = \al(d),\, u = u(d)$).
In our setting, $\la$ is an $M^\a_{u;\vv}$-distributed random partition, and 
$\Lambda_{(\al;u\sqrt{v_1})} := T_{\frac{\al}{u\sqrt{v_1}}, \frac{1}{u\sqrt{v_1}}}\la$, therefore $(f_{(\al;u\sqrt{v_1})})_{d\ge 1}\colon\R\to\R_{\ge 0}$ is a collection of random $1$--Lipschitz compactly supported functions. Let us verify the two conditions required by \cref{theo:LLNCDM}.

\smallskip
\emph{Step 1: For every $\epsilon > 0$, there exists $N > 0$ such that
\[ \lim_{d \to\infty}{M^\a_{u;\vv}\left(\|f_{(\al;u\sqrt{v_1})}\|_{L^\infty(\R\setminus [-N,N])} > \eps\right)} = 0. \]}

Let $U = U(d) := \frac{u^2\cdot v_1}{\al}$. Note that if the point $(x_0,y_0)$ 
in the first quadrant belongs to the rescaled diagram $\Lambda_{(\al;u\sqrt{v_1})}$, then the whole rectangle
$(0,0)$--$(x_0,0)$--$(x_0,y_0)$--$(0,y_0)$ is contained in
$\Lambda_{(\al;u\sqrt{v_1})}$, therefore $x_0\cdot y_0 \leq |\Lambda_{(\al;u\sqrt{v_1})}| = \frac{|\lambda|}{U}$. Applying this inequality to $x_0=(\omega_{\Lambda_{(\al;u\sqrt{v_1})}}(x) + |x|)/2\,$ and $y_0=(\omega_{\Lambda_{(\al;u\sqrt{v_1})}}(x) - |x|)/2$, we obtain
\[ |x| \leq \omega_{\Lambda_{(\al;u\sqrt{v_1})}}(x) \le \sqrt{x^2+4\cdot\frac{|\lambda|}{U}},\]
and consequently
\begin{equation}\label{eq:IneqForF}
0 \leq f_{(\al;u\sqrt{v_1})}(x)\le
\frac{\sqrt{x^2+4\cdot \frac{|\lambda|}{U}} - |x|}{2}\le
\frac{\sqrt{x^2+8e} - |x|}{2}=
\frac{4e}{\sqrt{x^2+8e} + |x|},
\end{equation}
whenever $|\la|\le 2e\cdot U$. Note that for every $\eps > 0$, there exists $N>0$ such that the RHS of
\eqref{eq:IneqForF} is smaller or equal to $\eps$, for all $|x|\ge N$. Therefore it is enough to prove that
    \begin{equation}
      \label{eq:WhatIWant}
      \lim_{d \to \infty}M^\a_{u;\vv}\left(|\lambda| > 2e\cdot U\right) = 0.
    \end{equation}

    By looking at the coefficient of $t^i$ in the Cauchy identity~\eqref{eq:Cauchy}, we deduce that the measure $M^\a_{u;\vv}$ gives $\Y_i$ a mass of $e^{-U}\cdot \frac{U^i}{i!}$, therefore for any $k\in\Z_{\ge 1}$, one has
\[M^\a_{u;\vv}\left(|\lambda| > k\right) \leq
   \exp(-U)\sum_{i=k}^\infty\frac{U^i}{i!} \leq \frac{U^k}{k!} =
   O\left(\left(\frac{e\cdot U}{k}\right)^{\!k}\cdot k^{-\frac{1}{2}}\right)\!,\text{ as }k\to\infty,\]
where the last equality follows from the Stirling's formula. 
Observe that our assumptions \eqref{eq:assumptions'_weak} imply that $U\to\infty$ as $d\to\infty$, therefore if we choose $k := \lfloor 2e\cdot U\rfloor$, then also $k\to\infty$ as $d\to\infty$.
We conclude that \eqref{eq:WhatIWant} holds true, and this finishes the proof of this step.

\smallskip
\emph{Step 2: There exists a probability measure $P_\infty\in\mathcal{M}_{<\infty}(\R)$ such that $f_{(\al;u\sqrt{v_1})}(x)dx$ converges to $P_\infty$ in moments, in probability.}

Consider the function $\mathbb{E}^\a_{u;\vv}[f_{(\al;u\sqrt{v_1})}(x)]$, which for simplicity we denote $\mathbb{E}[f(x)]$. \cref{eq:SFormula} shows that, for any $\ell\in\Z_{\ge 0}$, we have
\begin{equation}\label{moments_and_S}
\int_{\R}{x^{\ell}\,\mathbb{E}[f(x)]dx} = 
\mathbb{E}\left[ \!\int_{\R}{x^{\ell} f(x) dx} \right] \!=\! 
\frac{\mathbb{E}\left[S_{\ell+2}(\Lambda_{(\al; u\sqrt{v_1})}) \right]}{\ell+1}.
\end{equation}
When $\ell=0$, this shows $\int_{-\infty}^{\infty}{\E[f(x)]dx} = \E[S_2(\Lambda_{(\al; u\sqrt{v_1})})] = v_1^{-1}\cdot\E[S_2(\Lambda_{(\al; u)})]$ and this equals $1$, by~\cref{eq:Cumu-Ss}. Therefore $\E[f(x)]$ is the density of a probability measure,\footnote{We point out that this is why we've considered the rescaled partition $\Lambda_{(\al; u\sqrt{v_1})}$ with $u\sqrt{v_1}$ instead of $u$.} to be denoted by $P_d$.

Since $\mathbf{L}^{\conn}(\ell) = \mathbf{L}(\ell)$ is finite and $S_{\ell}(\Lambda_{(\al; u\sqrt{v_1})}) = 
v_1^{-\ell/2}\cdot S_{\ell}(\Lambda_{(\al; u)})$, then \cref{eq:Cumu-Ss} shows that the following limits exist:
\begin{multline}\label{lth_moment}
\lim_{d\to\infty} \E^\a_{u;\vv}\left[ S_{\ell}(\Lambda_{(\alpha;u\sqrt{v_1})}) \right] =
v_1^{-\frac{\ell}{2}}\lim_{d\to\infty}\kappa^\a_{u;\vv}\left( S_{\ell}(\Lambda_{(\al;u)}) \right) \\
= v_1^{-\frac{\ell}{2}}\sum_{\Gamma\in\mathbf{L}_0(\ell)} 
	\frac{1}{|\SSS^0(\Gamma)|}\prod_{i\geq 1}  (i\cdot g)^{|\SSS^i_{\tiny\rightarrow}(\Gamma)|}
	\, v_i^{\,|\SSS_i(\Gamma)|},
\end{multline}
for all $\ell\in\Z_{\ge 1}$. Note that since $\alpha/u^2 \to 0$, then the only 
paths that survive from \eqref{eq:Cumu-Ss} are the ones with $|\mathbf{P}(\ribbon)|=0$; 
this explains why after taking the limit as $d\to\infty$, the sum above is over 
$\Gamma\in\mathbf{L}_0(\ell)$ (with no pairings) and not over $\Gamma\in\mathbf{L}(\ell)$.
Similarly, note that any Łukasiewicz ribbon path $\ribbon\in\mathbf{L}^{\conn}(\ell_1,\ell_2)$ 
must have at least $1$ pairing for the connectivity condition to hold,
i.e. $|\mathbf{P}(\ribbon)| = \sum_{i\ge 1}{|\mathbf{P}_i(\ribbon)|} \ge 1$, 
therefore \cref{eq:Cumu-Ss} also implies
\begin{equation}\label{lth_variance}
\lim_{d\to\infty} \text{Var}^\a_{u; \vv}\left[ S_{\ell}(\Lambda_{(\al;u\sqrt{v_1})}) \right] =
v_1^{-\ell}\lim_{d\to\infty} 
{\kappa^\a_{u;\vv}\left(S_{\ell}(\Lambda_{(\al;u)}), S_{\ell}(\Lambda_{(\al;u)}) \right)} = 0.
\end{equation}

From \eqref{lth_moment}, \eqref{lth_variance}, and Chebyshev's inequality, 
we have that for any $\epsilon > 0$ and any $\ell\in\Z_{\ge 1}$, 
\begin{equation}\label{s_ell_prob}
\lim_{d\to\infty}{ M^\a_{u;\vv}\left( | S_\ell(\Lambda_{(\al;u\sqrt{v_1})}) - s_\ell | \ge\epsilon \right) } = 0,
\end{equation}
where 
\begin{equation}\label{s_ell}
s_\ell := v_1^{-\frac{\ell}{2}}\sum_{\Gamma\in\mathbf{L}_0(\ell)} 
\frac{1}{|\SSS^0(\Gamma)|}\prod_{i\geq 1} (i\cdot g)^{|\SSS^i_{\tiny\rightarrow}(\Gamma)|}
\, v_i^{\,|\SSS_i(\Gamma)|}.
\end{equation}

Next, observe that \eqref{moments_and_S}, \eqref{lth_moment} and
\eqref{lth_variance} imply that the moments and variances of the
measures $P_d$ are uniformly bounded. Thus $(P_d)_{d\ge 1}$ is a tight
sequence of probability measures, and thus it contains a subsequence
$(P_{d_k})_{k\ge 1}$ that converges weakly to some probability measure
$P_\infty$. The limit \eqref{lth_moment} shows further that the
$\ell$-th moment of $P_\infty$ is $\frac{s_{\ell+2}}{\ell+1}$, for all
$\ell\ge 0$. Finally \eqref{moments_and_S} and \eqref{s_ell_prob}
imply that for any $\eps>0$ and any $\ell\in\Z_{\ge 1}$, we have that
\[  \lim_{d\to\infty}{ M^\a_{u;\vv}\left( \left| \int_\R{x^\ell f_{(\al; u\sqrt{v_1})}(x)dx} - \int_\R{x^\ell P_\infty(dx)} \right| \ge\eps \right)} = 0. \]
In \cref{sec:LimitShape}, we show that $P_\infty$ is uniquely
determined by its moments (see~\cref{cor:UniqueP}), which finishes the proof.

\smallskip
\emph{Step 3: Conclusion of the argument.} 

Steps 1 and 2 prove that the two conditions in \cref{theo:GeneralLLN} are satisfied for the sequence $(f_{(\al;u\sqrt{v_1})})_{d\ge 1}$ and $\mu = P_\infty$.
Consequently, there exists a deterministic function $f_{\Lambda_{g;\vv}}\colon\R\to\R_{\ge 0}$, which is a probability density function of $P_\infty$, such that for every $\epsilon > 0$:
\[ \lim_{d \to \infty}M^\a_{u;\vv} \left(
\|f_{(\al;u\sqrt{v_1})} - f_{\Lambda_{g; \vv}}\|_{\infty} > \eps \right) =0.\]
Consequently, setting $\omega_{\Lambda_{g;\vv}}(x) := 2f_{\Lambda_{g;\vv}}(x)+|x|$, we conclude the convergence of $\omega_{\Lambda_{(\al;u\sqrt{v_1})}}$ to
$\omega_{\Lambda_{g;\vv}}$ in the supremum norm, in probability.

Finally, the fact that formula \eqref{mu_moments} gives the moments of the 
associated transition measure $\mu_{g;\vv}$ follows from Markov--Krein correspondence (\cref{theo:Meliot}) and \cref{prop:relations}.

\medskip

\textbf{\emph{End of proof of~\cref{theo:CLTCDM}.}} For the CLT, let us first prove the following limits:
\begin{align}
    \lim_{d\to\infty}\ &\kappa^\a_{u;\vv}
    \left(\langle \Delta^\a_{u;\vv}, x^{k-2} \rangle, \langle \Delta^\a_{u;\vv} ,x^{l-2} \rangle\right) =  \frac{v_1^{-\frac{k+l}{2}}}{(k-1) (l-1)}\times\label{eq:2InFinalProof}\\
\times&\sum_{\ribbon=(\Gamma_1, \Gamma_2)\in\mathbf{L}^{\conn}(k,l)\atop |\mathbf{P}(\ribbon)|=1}\frac{1}{|\mathbf{S}^0(\Gamma_1)||\mathbf{S}^0(\Gamma_2)|}
	\prod_{i\ge 1} (i\cdot g)^{|\SSS_{\tiny\rightarrow}^i(\ribbon)|}\,
	i^{\,|\mathbf{P}_i(\ribbon)|}\, v_i^{\,|\SSS_i(\ribbon)|},\quad \forall\, k, l \ge 2,\nonumber\\
        \lim_{d\to\infty}\ &\kappa^\a_{u;\vv}\left(\langle \Delta^\a_{u;\vv} ,x^{\ell_1-2} \rangle,\dots, \langle \Delta^\a_{u;\vv} ,x^{\ell_n-2} \rangle\right) = 0,\quad \forall\, n>2,\ \forall\, \ell_1, \cdots, \ell_n \ge 2.
\label{eq:3InFinalProof}
\end{align}

In order to compute
  \[\lim_{d\to\infty}\kappa^\a_{u; \vv}
	\left(\langle \Delta^\a_{u;\vv} ,x^{\ell_1-2} \rangle,\dots,
    \langle \Delta^\a_{u;\vv} ,x^{\ell_n-2} \rangle\right)\] for $n \geq 2$,
  first use the fact that the cumulant $\kappa^\a_{u;\vv} $ of order $n \geq 2$ is invariant by any translation
 of any argument by a scalar, so that using also \eqref{eq:SFormula} we
 have the following equalities
\begin{multline}\label{kappas_cums}
\kappa^\a_{u; \vv}\left(\langle \Delta^\a_{u;\vv} ,x^{\ell_1-2} \rangle,\dots,
    \langle \Delta^\a_{u;\vv} ,x^{\ell_n-2} \rangle\right)\\
= \left( \frac{u}{\sqrt{\al}} \right)^{\!n}\cdot\,
	\kappa^\a_{u;\vv}\left(\langle f_{(\alpha;u\sqrt{v_1})}, x^{\ell_1-2} \rangle,\dots,
    \langle f_{(\alpha;u\sqrt{v_1})}, x^{\ell_n-2} \rangle\right)\\
= \left( \frac{u}{\sqrt{\al}} \right)^{\!n}\cdot\prod_{i=1}^n\frac{1}{(\ell_i-1)}\cdot
	\kappa^\a_{u;\vv}\left(S_{\ell_1}(\Lambda_{(\alpha;u\sqrt{v_1})}),\dots,
    S_{\ell_n}(\Lambda_{(\alpha;u\sqrt{v_1})})\right).
\end{multline}
Next, use the scaling property~\eqref{eq:scaling} that shows 
$S_\ell(\Lambda_{(\al;u\sqrt{v_1})}) = v_1^{-\frac{\ell}{2}}\cdot S_\ell(\Lambda_{(\al;\sqrt{v_1})})$, 
and then~\eqref{eq:Cumu-Ss} from \cref{cor_formulas_cums}.
Note that the sum in the last equation is over connected Łukasiewicz ribbon paths 
$\ribbon\in\mathbf{L}^{\conn}(\ell_1,\dots,\ell_{n})$ on $n$ sites, and 
each such $\ribbon$ must have at least $(n-1)$ pairings for it to be connected,
i.e. $|\mathbf{P}(\ribbon)| = \sum_{i\ge 1}{|\mathbf{P}_i(\ribbon)|} \ge n-1$.
As a result:
\begin{multline}\label{kappa_conn}
\kappa^\a_{u; \vv}\left(S_{\ell_1}(\Lambda_{(\alpha;u\sqrt{v_1})}),\dots,
    S_{\ell_n}(\Lambda_{(\alpha;u\sqrt{v_1})})\right) = 
  \bigg(\frac{\alpha}{u^2}\bigg)^{\!n-1}\cdot v_1^{-\frac{\ell_1+\cdots+\ell_n}{2}}\times\\
  \times\Bigg(\sum_{\ribbon\in\mathbf{L}^{\conn}(\ell_1,\dots,\ell_{n}),\atop
  |\mathbf{P}(\ribbon)|=n-1}\prod_{j=1}^n \frac{1}{|\mathbf{S}^0(\Gamma_j)|}\prod_{i \geq 1} 
\bigg ( i\cdot\frac{\alpha-1}{u}\bigg)^{\!|\SSS^i_{\tiny\rightarrow}(\ribbon)|}i^{\,|\mathbf{P}_i(\ribbon)|}\, v_i^{\,|\SSS_i(\ribbon)|} +\frac{\alpha}{u^2}\cdot X_n \Bigg),
\end{multline}
where $X_n$ is some polynomial in $\QQ[\frac{\alpha}{u^2}, \frac{\alpha-1}{u},v_1,v_2,\dots]$.
By combining \eqref{kappas_cums} and \eqref{kappa_conn}, we have an expression of the form
\begin{equation*}
\kappa^\a_{u;\vv}\left(\langle \Delta^\a_{u;\vv} ,x^{\ell_1-2} \rangle,\dots,
\langle \Delta^\a_{u;\vv} ,x^{\ell_n-2} \rangle\right)
= \bigg(\frac{\sqrt{\alpha}}{u}\bigg)^{\!n - 2}\cdot\frac{v_1^{-\frac{\ell_1+\cdots+\ell_n}{2}}}{\prod_{i=1}^n{(\ell_i-1)}}\cdot 
\left( Y_n + \frac{\alpha}{u^2}\cdot X_n \right),
\end{equation*}
where $Y_n$ is the sum in the RHS of \cref{kappa_conn}.
Notice that $X_n$ and $Y_n$ have finite limits as $d\to\infty$: in fact, just replace $\frac{\al-1}{u} \mapsto g,\, \frac{\al}{u^2} \mapsto 0$ in their expressions.
By taking the $d\to\infty$ limit of this equation, we obtain 
\eqref{eq:2InFinalProof} and \eqref{eq:3InFinalProof}.

To conclude the theorem, we show that $\,\lim_{d\to\infty}{\kappa^\a_{u;\vv}(\langle \Delta^\a_{u;\vv} ,x^{\ell-2} \rangle)}$ is equal to the RHS of~\cref{MeanGaussian_2}.
Let us use the notations inside the proof of~\cref{theo:LLNCDM}. By definition, 
$\Delta^\a_{u;\vv} 
= \frac{u}{2\sqrt{\al}}( \omega_{\Lambda_{(\al;\,u\sqrt{v_1})}} - \omega_{\Lambda_{g;\vv}} )
= \frac{u}{\sqrt{\al}}( f_{(\al;\,u\sqrt{v_1})} - f_{g;\vv} )$. 
Together with \eqref{eq:SFormula} and \eqref{eq:scaling}, we obtain
\begin{align}
\kappa^\a_{u;\vv}\left(\langle \Delta^\a_{u;\vv}, x^{\ell-2} \rangle\right) &=
\frac{u}{\sqrt{\al}}\cdot\left\{ \kappa^\a_{u;\vv}\left(\langle f_{(\al;u\sqrt{v_1})}, x^{\ell-2} \rangle\right) -
\langle f_{g;\vv}, x^{\ell-2} \rangle \right\} \nonumber\\
&= \frac{u}{\sqrt{\al}}\cdot \left\{ 
\frac{v_1^{-\ell/2}}{\ell - 1}\cdot\mathbb{E}^\a_{u;\vv}\left[ S_{\ell}(\Lambda_{(\al; u)}) \right] -
\frac{s_\ell}{\ell-1} \right\}. \label{diff_moments}
\end{align}

 By \eqref{eq:Cumu-Ss} and \eqref{s_ell}, 
 the expression inside brackets in \eqref{diff_moments} can be written as a sum of 
 weighted ribbon paths in $\mathbf{L}(\ell)$. 
 Let $U$ be the sum over those $\ribbon\in\mathbf{L}(\ell)$ 
 with $|\mathbf{P}(\ribbon)| \ge 1$, and $V$ the sum over those 
 with $|\mathbf{P}(\ribbon)| = 0$.
 Note that $U = \frac{\alpha}{u^2}\cdot U'$, for some 
 $U'\in\QQ[\frac{\alpha-1}{u}, \frac{\alpha}{u^2}, v_1^{-1/2}, \{v_i\}_{i\ge 1}]$. 
 On the other hand, $V = \widetilde{V}(\frac{\al-1}{u}) - \widetilde{V}(g)$, 
where $\widetilde{V}(g)$ is the 
polynomial expression for $\frac{s_\ell}{\ell-1}$, see \eqref{s_ell}, and therefore 
 $V=\left( \frac{\alpha-1}{u} - g \right)\cdot V'$, for some 
 $V'\in\QQ[\frac{\alpha-1}{u}, g, v_1^{-1/2}, \{v_i\}_{i\ge 1}]$.

Then the expression in brackets equals $\frac{\alpha}{u^2}\cdot U' 
 + \left( \frac{\al-1}{u} - g \right)\cdot V'$, and it follows that
\begin{equation}\label{CLT_bound}
\kappa^\a_{u;\vv}\left(\langle \Delta^\a_{u;\vv}, x^{\ell-2} \rangle\right) 
= \frac{\sqrt{\al}}{u}\cdot U' + \frac{u}{\sqrt{\al}}\left( \frac{\al-1}{u} - g \right)\cdot V'.
\end{equation}
Both $U'$ and $V'$ have finite limits as $d\to\infty$ and moreover, by the refined version of \cref{main_assumption}, $\frac{\sqrt{\al}}{u} = \frac{1}{\sqrt{d}}(1+o(1)),\, \frac{u}{\sqrt{\al}} = \sqrt{d}(1+o(1)),\, \frac{\al-1}{u}-g=\frac{g'}{\sqrt{d}}(1+o(1))$; they imply that 
$\,\lim_{d\to\infty}{\kappa^\a_{u;\vv}(\langle \Delta^\a_{u;\vv} ,x^{\ell-2} \rangle)} = g'\cdot \lim_{d\to\infty}{V'}$. But we also have $\lim_{d\to\infty}{V'} = \lim_{\frac{\al-1}{u}\to g}{(\widetilde{V}(\frac{\al-1}{u})-\widetilde{V}(g))(\frac{\al-1}{u}-g)^{-1}}=\frac{\partial}{\partial g}\widetilde{V}(g)$.
Hence, we have finally obtained the RHS of \cref{MeanGaussian_2}, and have thus concluded the proof of the CLT.

\section{Parameters for Jack--Thoma measures}\label{sec:JackPos}

In this section, we investigate possible choices of parameters $u>0,\, \vv\in\R^{\infty}$, for which the Jack--Thoma measures $M^\a_{u;\vv}$ from \cref{plancherel_jack} are well-defined probability measures, both for fixed $\alpha>0$, and for all $\alpha>0$.
Moreover, we find large families of parameters for which \cref{main_assumption} is satisfied.

\subsection{Jack-positive specializations}
\label{subsec:JackPositive}

A \emph{specialization} of the $\R$-algebra of symmetric functions $\Symm$ is by definition a unital algebra
homomorphism $\rho: \Symm\to\R$.
Since $\Symm = \R[p_1, p_2, \cdots]$, then 
$\rho$ is uniquely determined by the values $\rho(p_1), \rho(p_2), \cdots$.
If we set $v_k := \rho(p_k)$, for all $k\in\Z_{\ge 1}$, and $\vv:=(v_1, v_2, \cdots)$,
we say that \emph{$\vv$ determines the specialization $\rho$}.
For any $f\in\Symm$, we denote its image under $\rho$ by either $\rho(f)$ or $f(\vv)$.

\begin{definition}\label{alpha_jack_positive}
The specialization $\rho:\Symm\to\R$ is an \textbf{$\pmb{\alpha}$-Jack-positive specialization} if
\begin{equation*}
\rho\left( J_{\la}^\a \right) \ge 0\quad\forall\,\la\in\Y.
\end{equation*}
\end{definition}

Observe that the previous definition depends on the value of $\alpha > 0$.

\begin{theorem}[Kerov--Okounkov--Olshanski~\cite{KerovOkounkovOlshanski1998}]\label{thm:KOO}
Assume that $\alpha > 0$ is fixed. The set of $\alpha$-Jack-positive specializations 
$\rho: \Symm\to\R$ is in bijection with the \textbf{Thoma cone}
\begin{multline*}
\Omega = \{ (a, b, c)\in \R^{\infty}\times \R^{\infty}\times \R \mid
a = (a_1, a_2, \cdots),\quad b = (b_1, b_2, \cdots),\\
a_1 \ge a_2 \ge \cdots \ge 0,\quad b_1 \ge b_2 \ge \cdots \ge 0,\quad
\sum_{i=1}^{\infty}{(a_i + b_i)} \le c \}.
\end{multline*}
For any $\omega = (a, b, c)\in\Omega$, the corresponding $\alpha$-Jack-positive 
specialization $\rho_{\omega}\!: \Symm\to\R$ is defined by
\begin{equation*}
\rho_{\omega}(p_1) = c,\qquad
\rho_{\omega}(p_k) = \sum_{i=1}^{\infty}{a_i^k} + 
(-\alpha)^{1 - k}\sum_{i=1}^{\infty}{b_i^k},\quad k=2,3,\cdots.
\end{equation*}
\end{theorem}

\begin{example}\label{ex:plancherel}
Given $u>0$, the \emph{Plancherel $u$-specialization} $\tau_u: \Symm\to\R$ is 
defined by
\begin{equation*}
\tau_u(p_1) = u,\qquad \tau_u(p_2) = \tau_u(p_3) = \cdots = 0.
\end{equation*}
This is the $\alpha$-Jack-positive specialization corresponding to $c=u$, and $a_i=b_i=0,\ \forall\, i\ge 1$, in Theorem \ref{thm:KOO}.
\end{example}

\begin{example}\label{exam:integer_pos}
For any infinite tuple $\vv=(v_1, v_2, \cdots)\in\R^{\infty}$ and $u\in\R$, recall that $u\cdot\vv := (uv_1, uv_2, \cdots)$. 
We claim that if $\vv=(v_1, v_2, \cdots)\in\R^{\infty}$ determines an
$\alpha$-Jack-positive specialization and if $m\in\Z_{>0}$, then $m\cdot\vv$ determines
another $\alpha$-Jack-positive specialization. In fact, let $(a, b, c)\in\Omega$,\,
$a = (a_1, a_2, \cdots)$,\, $b = (b_1, b_2, \cdots)$, be the triple corresponding
to $\vv$, as given by Theorem \ref{thm:KOO}.
Then $m\cdot\vv$ corresponds to the triple
$(a^{\otimes m}, b^{\otimes m}, m\cdot c)\in\Omega$, where
\begin{equation*}
a^{\otimes m} := (\underbrace{a_1, \cdots, a_1}_{m \text{ times}},
\underbrace{a_2, \cdots, a_2}_{m \text{ times}}, a_3, \cdots),
\quad b^{\otimes m} := (\underbrace{b_1, \cdots, b_1}_{m \text{ times}},
\underbrace{b_2, \cdots, b_2}_{m \text{ times}}, b_3, \cdots).
\end{equation*}
\end{example}

\begin{example}\label{exam:integer_pos_2}
For any infinite tuple $\vv=(v_1, v_2, \cdots)\in\R^{\infty}$, define
\begin{equation*}
\pm\vv := (v_1, -v_2, v_3, -v_4, \cdots).
\end{equation*}
We claim that if $\pm\vv$ determines a
$(1/\alpha)$-Jack-positive specialization, and if $m\in\Z_{>0}$,
then $(\alpha \cdot m)\cdot\vv$ determines an $\alpha$-Jack-positive specialization.
Recall that the automorphism $\omega_{1/\alpha}$ of $\Symm$ defined by
$\omega_{1/\alpha}(p_r) = (-1)^{r-1}\alpha^{-1} p_r$ has the following
property~\eqref{eq:OmegaDuality}:
\begin{equation*}
\omega_{1/\alpha}(J^{(1/\alpha)}_{\mu}) = \alpha^{-|\mu|} J^\a_{\mu'}
\quad \forall\, \mu\in\Y.
\end{equation*}
We show that $J_\la^\a((\al\cdot m)\cdot\vv) \ge 0$, for all $\la\in\Y$. 
\begin{align*}
J_{\la}^\a((\al\cdot m)\cdot\vv)
&= J_{\la}^\a(\al\cdot m \cdot v_1,\, \alpha \cdot m \cdot v_2,\, \alpha \cdot m \cdot v_3,\, \alpha \cdot m \cdot v_4, \cdots)\\
&= \alpha^{|\la|}\cdot (\omega_{1/\alpha} J_{\la'}^{(1/\alpha)})(\alpha \cdot m \cdot v_1,\, \alpha \cdot m \cdot v_2,\, \alpha \cdot m \cdot v_3,\, \alpha \cdot m \cdot v_4, \cdots)\\
&= \alpha^{|\la|}\cdot J_{\la'}^{(1/\alpha)}(m \cdot v_1,\, -m \cdot v_2,\, m \cdot v_3,\, -m \cdot v_4, \cdots)\\
&= \alpha^{|\la|}\cdot J_{\la'}^{(1/\alpha)}(m\cdot(\pm\vv)).
\end{align*}
Since $\pm\vv$ determines a $(1/\alpha)$-Jack positive
specialization, and $m\in\Z_{\ge 1}$, then $m\cdot (\pm\vv)$
determines a $(1/\alpha)$-Jack positive specialization, by Example
\ref{exam:integer_pos}. Therefore $J_\la^\a((\al\cdot m)\cdot\vv) = \alpha^{|\la|}\cdot J_{\la'}^{(1/\alpha)}(m\cdot(\pm\vv)) \ge 0$, as claimed.
\end{example}

\subsection{Totally Jack-positive specializations}

\begin{definition}
The specialization $\rho:\Symm\to\R$ is a \textbf{totally Jack-positive specialization} if
\begin{equation*}
\rho\left( J_{\la}^\a \right) \ge 0,\quad \forall\,\la\in\Y,\ \forall\,\alpha > 0.
\end{equation*}
\end{definition}

Note that the condition is imposed for all $\alpha > 0$.

\begin{theorem}[Classification of totally Jack-positive specializations]\label{thm_jack_spec}
The specialization $\rho: \Symm\to\R$ is totally Jack-positive if and only if there exist 
$c\ge 0,\ a_1 \ge a_2 \ge \cdots \ge 0$, with $\sum_{i=1}^{\infty}{a_i} \le c$, such that
\begin{equation*}
\rho(p_1) = c,\qquad \rho(p_k) = \sum_{i=1}^{\infty}{a_i^k},\quad \forall\, k=2, 3, \cdots.
\end{equation*}
\end{theorem}

\begin{proof}
Let $\rho: \Symm\to\R$ be a specialization of $\Symm$.

\smallskip

\emph{Part 1.} In this first part of the proof, assume that $\rho$ is a totally Jack-positive specialization, and let us prove the existence of $c, a_1, a_2, \cdots\in\R_{\ge 0}$, satisfying the conditions of the theorem.

Set $c := \rho(p_1)$. The Jack symmetric function $J_{(1)}^\a$ associated to the
partition $(1)$ is $p_1$. Since $\rho$ is totally Jack-positive,
it follows that $c = \rho(J_{(1)}^\a) \ge 0$.
Theorem \ref{thm:KOO} shows that for any $\al > 0$, there exist reals
$a_{1;\,\al}\ge a_{2;\,\al}\ge \cdots \ge 0,\ b_{1;\,\al}\ge b_{2;\,\al}\ge \cdots \ge 0$,
with
\begin{equation}\label{bounded_seqs}
c \ge \sum_{i=1}^{\infty}{(a_{i;\,\al} + b_{i;\,\al})},
\end{equation}
such that
\begin{equation}\label{rho_pks}
\rho(p_k) = \sum_{i=1}^{\infty}{(a_{i;\,\al})^k} + (-\al)^{1-k}\sum_{i=1}^{\infty}{(b_{i;\,\al})^k},\quad\forall\, k\ge 2.
\end{equation}

By virtue of inequality \eqref{bounded_seqs}, both sequences $\{a_{i;\,\al}\}_{i=1}^{\infty}$
and $\{b_{i;\,\al}\}_{i=1}^{\infty}$ are uniformly bounded over all $\al > 0$.
Thus, using Cantor's diagonal extraction, we can find a sequence
$\{\alpha_{n}\}_{n\ge 1}$ such that the following limits exist
\begin{equation}\label{a_limit}
\lim_{n\to\infty}{ \al_{n} } = \infty,\qquad \lim_{n\to\infty}{a_{j;\,\al_{n}}} = a_j,\ \forall j\in\Z_{\ge 1}.
\end{equation}

Fix an arbitrary integer $k\ge 2$.
As each $\{b_{i;\,\al_{n}}\}_{i\ge 1}$ is a sequence of nonnegative reals whose 
sum is bounded above by $c$, we have
$$
0\le\sum_{i=1}^{\infty}{ (b_{i;\,\al_{n}})^k } \le
\left( \sum_{i=1}^{\infty}{ b_{i;\,\al_{n}} } \right)^k \le c^k,\quad \forall\, n\ge 1.
$$
Consequently, by \eqref{a_limit},
\begin{equation}\label{limit_bs}
\lim_{n\to\infty}{ (-\al_{n})^{1-k}\sum_{i=1}^{\infty}{(b_{i;\,\al_{n}})^k} } = 0.
\end{equation}
The previous reasoning also shows 
$0\le\sum_{i=1}^{\infty}{ (a_{i;\,\al_{n}})^k } \le c^k$, for all $n\ge 1$.
Then we can apply the dominated convergence theorem; due to \eqref{a_limit}, we have
\begin{equation}\label{limit_as}
\lim_{n\to\infty}{ \sum_{i=1}^{\infty}{(a_{i;\,\al_{n}})^k} } = \sum_{i=1}^{\infty}{a_i^k}.
\end{equation}
By combining the relations \eqref{rho_pks}, \eqref{limit_bs} and \eqref{limit_as}, we conclude
$$\rho(p_k) = \sum_{i=1}^{\infty}{a_i^k},\quad\forall\, k\ge 2.$$
Finally, since $a_{1;\,\al_{n}}\ge a_{2;\,\al_{n}}\ge\cdots\ge 0$, it follows by \eqref{a_limit} that
$a_1\ge a_2\ge \cdots\ge 0$. Likewise, $c \ge \sum_{i\ge 1}{ a_{i;\,\al_{n}} }$
implies $c\ge \sum_{i\ge 1}{a_i}$. This ends the proof of the ``only if" direction.

\smallskip

\emph{Part 2.}
In this second part, we prove the ``if'' direction of the theorem. Assume that 
$\rho(p_1) = c$, and $\rho(p_k) = \sum_{i=1}^{\infty}{a_i^k},\ k\ge 2$, 
for some parameters $c\ge 0,\ a_1 \ge a_2 \ge \cdots \ge 0$, with 
$\sum_{i=1}^{\infty}{a_i} \le c$. 
Then for any $\alpha>0$, we claim that $\rho$ is an $\alpha$-Jack-positive specialization.
Indeed, this is an immediate consequence of Theorem \ref{thm:KOO}
(the parameters $b_1, b_2, \cdots$ are all equal to zero). 
Hence $\rho$ is a totally Jack-positive specialization.
\end{proof}

\begin{definition}\label{def:admissible_pair}
A pair $(g, \vv)\in\R\times\R^\infty$ is called an \textbf{admissible pair} if one of the following two conditions is satisfied:
\begin{itemize}
	\item $g>0$, and $\vv = (v_j)_{j\ge 1}$ determines a totally Jack-positive specialization.
	\item $g<0$, and $\pm\vv = \left( (-1)^{j-1}v_j \right)_{j\ge 1}$ determines a totally Jack-positive specialization.
\end{itemize}
\end{definition}

\begin{proposition}\label{prop:CDMPar2}
Let $(g, \vv)$ be an admissible pair. Then there exist sequences $(\al(d))_{d\ge 1}$, $(u(d))_{d\ge 1}$ of positive real numbers such that the refined version of \cref{main_assumption} is satisfied with $g'=0$; explicitly:

\begin{enumerate}
	\item If $g > 0$, then we can choose $\alpha(d) = g^2d$,\, $u(d) = \lceil gd\rceil$.
	\item If $g < 0$, then we can choose $\alpha(d) =
          \frac{1}{g^2 d}$,\, $u(d) = \frac{\lceil -gd\rceil}{g^2 d}$.
\end{enumerate}
\end{proposition}
\begin{proof}
In both cases, we first observe that the conditions in
\eqref{eq:assumptions'} with $g'=0$ are readily verified. It remains to verify that $M^{(\alpha(d))}_{u(d);\,\vv}(\la)\geq 0$, for all $\la\in\Y,\,d\in\Z_{\ge 1}$. By~\cref{plancherel_jack} of Jack--Thoma measures, the inequality to verify is equivalent to:
\begin{equation}\label{ineq_jack}
J^{(\alpha(d))}_{\lambda}(u(d)\cdot\vv)\geq 0,\quad \forall\, \la\in\Y,\ \forall\, d\in\Z_{\ge 1}.
\end{equation}
Let us begin with case (1). Since $\vv$ is totally Jack-positive and $u(d) = \lceil gd\rceil \in\Z_{\ge 1}$, then $u(d)\cdot\vv$ is also totally Jack-positive, by virtue of \cref{exam:integer_pos}, and hence \eqref{ineq_jack} follows.

In case (2), since $\pm\vv$ is totally Jack-positive and $\lceil -gd\rceil\in\Z_{\ge 1}$, then $(\alpha(d)\lceil -gd\rceil)\cdot\vv$ is $\alpha(d)$-Jack-positive, by virtue of \cref{exam:integer_pos_2}. Notice that $\alpha(d)\lceil -gd\rceil =\frac{\lceil -gd\rceil}{g^2 d} = u(d)$, so the claim just verified is that $u(d)\cdot\vv$ is $\alpha(d)$-Jack-positive, which proves \eqref{ineq_jack}.
\end{proof}

As a consequence of this proposition, note that Theorems \ref{theo:LLNCDM} and \ref{theo:CLTCDM} can be applied to show limit theorems for the measures $M^{(\al(d))}_{u(d);\,\vv}$. 
In particular, for any admissible pair $(g,\vv)$, there exists a probability measure $\mu_{g;\vv}$ with moments given by \eqref{mu_moments}.
This proposition also played a role in the construction of the most general examples from Sections \ref{sec:examples} and \ref{sec:examples_2}.

\section{Universal formulas for the global asymptotics of Jack-deformed random Young diagrams}\label{sec:Duality}

Recall that \cref{theo:LLNDS} states that random Young diagrams
sampled w.r.t.~a family of characters $\chi_d$ with the AFP concentrate 
around a limit shape $\omega_{\Lambda_\infty}$.
In this section, we prove a \emph{universality} result: the formula \eqref{mu_moments} 
for the moments of the transition measure $\mu_{g;\vv}$ arising from the 
Jack--Thoma measures also computes the moments of the
transition measure $\mu_{\Lambda_\infty}$ associated to a general limit shape 
$\omega_{\Lambda_\infty}$ (and such moments uniquely determine the limit shape).
Further, we prove that the combinatorics of Łukasiewicz ribbon paths from \cref{sec:Jack-deformed-2} yields universal formulas for the covariances of 
Gaussian fluctuations around $\omega_{\Lambda_\infty}$.

\begin{theorem}[Universality of the transition measure of the limit shape]\label{theo:ShapeDS}
Assume that the sequence $\al=\al(d)$ satisfies \eqref{eq:double-scaling-refined}, 
for some $g,g'\in\R$. Further, assume that $\chi_d\colon\Y_d\to\R$ are Jack characters 
such that the sequence $(\chi_d)_{d\ge 1}$ fulfills the AFP and let $v_2,v_3,\dots$ 
be the associated parameters given by \cref{def:approx-factorization-charactersA}. 
Let $\omega_{\Lambda_\infty}$ be the corresponding limit shape given by
\cref{theo:LLNDS}. Then $\omega_{\Lambda_\infty}$ is exactly equal to $\omega_{\Lambda_{g;\vv}}$, the limit shape for Jack--Thoma measures obtained in~\cref{theo:LLNCDM}. 
Equivalently, the transition measure $\mu_{\Lambda_\infty}$ associated to 
$\omega_{\Lambda_\infty}$ is equal to $\mu_{g;\vv}$, and therefore it is 
uniquely determined by its moments 
\begin{equation}\label{mu_moments:DS}
\int_\R{x^\ell \mu_{\Lambda_\infty}(dx)} = \sum_{\Gamma\in\mathbf{L}_0(\ell)}
\prod_{i\ge 1}{(i\cdot g)^{|\SSS_{\tiny\rightarrow}^i(\Gamma)|}\, v_i^{\,|\SSS_i(\Gamma)|}},\quad\forall\,\ell\in\Z_{\ge 1},
\end{equation}
where we set $v_1:=1$.
\end{theorem}

\begin{remark}\label{rem:LukasiewiczNC}
In the special case of fixed $\alpha=1$, our formula~\eqref{mu_moments:DS} is equivalent to a result of Biane~\cite{Biane1998,Biane2001} that
  connects asymptotic representation theory of symmetric groups
 and free probability. He proved that
  the asymptotic behavior of characters with the AFP evaluated on 
  long cycles is given by free cumulants of the limiting
  transition measure $\mu_{\Lambda_\infty}$:
\begin{equation}\label{eq:Biane}
    R_{1}^{\mu_{\Lambda_\infty}}=0,\ \
    R_{i+1}^{\mu_{\Lambda_\infty}} = v_{i},\ \ \ i \geq 1.
\end{equation}
If $\alpha=1$ is fixed, then we are in the fixed temperature 
  regime and $g=0$. If we set $v_0 := 0$, our formula~\eqref{mu_moments:DS} 
	can be rewritten as 
\begin{equation}\label{moments_v}
M_{\ell}^{\mu_{\Lambda_\infty}} =
    \sum_{\Gamma\in\mathbf{L}_0(\ell)} \prod_{i\ge 0}v_i^{\,|\SSS_i(\Gamma)|}.
\end{equation}

On the other hand, there is a well known relation between the free cumulants 
and moments of a probability measure, see e.g.~\cite[Lecture 11]{NicaSpeicherBook}. 
It expresses the moment $M_{\ell}^{\mu_{\Lambda_\infty}}$ as a polynomial 
on the free cumulants $R_j^{\mu_{\Lambda_\infty}}$, in terms of the 
non-crossing set-partitions of $[1\,..\,\ell]$. By using the classical bijection between
non-crossing set-partitions of $[1\,..\,\ell]$ and Łukasiewicz paths of length $\ell$ 
(see~\cite[Lecture 9]{NicaSpeicherBook}), the resulting relations are 
\begin{equation}\label{moments_R}
M_{\ell}^{\mu_{\Lambda_\infty}} =
\sum_{\Gamma\in\mathbf{L}_0(\ell)} \prod_{i\ge 0}
\left(R_{i+1}^{\mu_{\Lambda_\infty}}\right)^{\!|\SSS_i(\Gamma)|}.
\end{equation}
By comparing the equations \eqref{moments_v} and \eqref{moments_R}, 
for all $\ell\in\Z_{\ge 1}$, we give a new proof of Biane's result~\eqref{eq:Biane}.
\end{remark}

\begin{remark}
In \cref{sec:JackPos} we showed that the Jack--Thoma measures 
are parametrized by $\Omega$, and \cref{ex:CDM} shows that
one can construct the associated Jack characters with the AFP with
the same sequence $\vv$. It is natural to ask if the
parameters $\vv$ must be the same in both models. The answer is no --- there are Jack characters with the
AFP whose limit shape will not appear as the limit shape of Jack--Thoma measures. 
The easiest example is given by the following construction
described by Śniady in~\cite{Sniady2006c}. Let $\alpha=1$, and
$\lambda_d$ be a sequence of Young diagrams of size $d$ whose shapes
tend to a fixed limit $\lambda$ (i.e. $\omega_{\Lambda^{(\alpha=1)}} \to \omega_\lambda$). Then $\chi_{\lambda^d}$ is a sequence
of characters with the AFP and $\Lambda_\infty = \lambda$. Now, one can choose $\lambda$ to be any
shape whose moments are not given by the Jack-positive specialization, for instance $\lambda$
can be a square of size $1$. We will discuss more examples of Jack
characters with the AFP in the forthcoming work.
\end{remark}

In this section, we shall use the notions from \cref{Luk_sec}.
Additionally, it will be useful to define the partition $\mu(\ribbon)$ that encodes the number of
up steps of each degree in a given Łukasiewicz ribbon path $\ribbon$ (recall that the paired up steps are not counted):
\[ \mu(\ribbon) := (1^{|\SSS_1(\ribbon)|}, 2^{|\SSS_2(\ribbon)|}, 3^{|\SSS_3(\ribbon)|},\dots). \]
Moreover, for any partition $\mu$, we use the standard notation $v_\mu := \prod_{i\geq 1}{v_{\mu_i}}$. 
In particular,
\[ v_{\mu(\ribbon)} := \prod_{j\geq 1}v_j^{|\SSS_j(\ribbon)|}. \]
We need some extra notations for the next theorem. If $\Gamma=(w_0, w_1,\dots, w_\ell)$, $\,w_i=(i,y_i)$, is a Łukasiewicz path, denote $\SSS_{\ne 0}(\Gamma) := \SSS(\Gamma)\setminus\SSS_0(\Gamma)$ the set of vertices $w_i,\,i\ge 1$, that are not preceded by a horizontal step.
Further, if $s=w_j\in\SSS_{\ne 0}(\Gamma)$, let $\deg(s)\in\{-1, 0, 1, 2, \dots\}$ be the difference $y_j - y_{j-1}$, i.e.~if the step preceding $s$ is an up step, then $\deg(s)$ is the degree of that up step, and if the step preceding $s$ is a down step, then $\deg(s) = -1$.

\begin{theorem}[Universality of the Gaussian fluctuations around the limit shape]\label{theo:CovDS}
Assume that $\al=\al(d)$ satisfies \eqref{eq:double-scaling-refined}, for some $g,g'\in\R$, also assume that $(\chi_d)_{d\ge 1}$ fulfills the enhanced AFP, and let $(v_k, v_k')_{k\ge 2},\, (v_{(k|l)})_{k, l\ge 2}$ be the associated parameters given by~\cref{def:approx-factorization-charactersA}.

Recall that \cref{theo:CLTDS} shows that the random functions $\Delta^\a_d = \frac{\sqrt{d}}{2}\cdot (\omega_{\Lambda^\a_d} - \omega_{\Lambda_\infty})$ converge to a Gaussian process $\Delta_\infty$, in the sense that $\langle\Delta^\a_d, x^m\rangle$ converge weakly to $\langle \Delta_\infty, x^m\rangle$, as $d\to\infty$.
Then the means and covariances of polynomial observables of $\Delta_\infty$ have explicit combinatorial formulas, given by the following polynomials in $g,g',\,(v_k, v_k')_k,\,(v_{(k|l)})_{k,l}\,$:
\begin{multline}\label{eq:ExpDS}
\E[\langle \Delta_\infty,\, x^{\ell-2} \rangle] = \frac{1}{\ell-1}
\!\left( g'\frac{\partial}{\partial g} + \sum_{i\ge 2}{v_i'\frac{\partial}{\partial v_i}} \right)
\!\sum_{\Gamma\in\mathbf{L}_0(\ell)}{\frac{1}{|\SSS^0(\Gamma)|}\prod_{i\ge 1}(i\cdot g)^{|\SSS_{\tiny\rightarrow}^i(\ribbon)|}\, v_i^{\,|\SSS_i(\ribbon)|}},
\end{multline}
\begin{multline}\label{eq:CovDS}
\Cov\left(\langle \Delta_\infty,\, x^{k-2}\rangle, \langle \Delta_\infty,\, x^{l-2}\rangle \right) = \frac{1}{(k-1)(l-1)}\times\\
\times\left( \sum_{\ribbon=(\Gamma_1, \Gamma_2)\in\mathbf{L}^{\conn}(k, l)\atop |\mathbf{P}(\ribbon)| = 1}\frac{v_{\mu(\ribbon)}}{|\mathbf{S}^0(\Gamma_1)|\cdot |\mathbf{S}^0(\Gamma_2)|}\cdot\prod_{i\geq 1} (i\cdot g)^{|\SSS^i_{\tiny\rightarrow}(\ribbon)|}\,i^{\,|\mathbf{P}_i(\ribbon)|} \,+\right.\\
\left. + \sum_{\ribbon=(\Gamma_1, \Gamma_2)\in\mathbf{L}(k, l)\atop |\mathbf{P}(\ribbon)| = 0} \sum_{s_1 \in \SSS_{\neq 0}(\Gamma_1)\atop s_2 \in \SSS_{\neq 0}(\Gamma_2)}
v_{(\deg(s_1)|\deg(s_2))}\cdot\frac{v_{\mu(\ribbon)\setminus\{ \deg(s_1), \deg(s_2) \}}}{|\mathbf{S}^0(\Gamma_1)|\cdot |\mathbf{S}^0(\Gamma_2)|}\prod_{i\geq 1} (i\cdot g)^{|\SSS^i_{\tiny\rightarrow}(\ribbon)|} \right),
\end{multline}
where $v_1:=1$. Both are valid for all $\ell,k,l \geq 2$. 
For~\eqref{eq:CovDS}, we additionally used the following conventions: 
$\,v_{(-1|-1)}:=-1$, and $\,v_{(-1|j)} = v_{(j|-1)} = v_{(1|j)} = v_{(j|1)} :=0$, for all $j \geq 1$.
\end{theorem}

\begin{remark}\label{funny_sum}
Let us make some clarifications about the last summation in the RHS of \eqref{eq:CovDS}.
Since $v_{(-1|j)} = v_{(j|-1)} = v_{(1|j)} = v_{(j|1)} := 0$, for all $j \geq 1$, the inner sum is over pairs $(s_1, s_2)$ with $\deg(s_1),\,\deg(s_2)\ge 2$, or pairs with $\deg(s_1)=\deg(s_2)=-1$.
In the former case, $v_{(\deg(s_1)|\deg(s_2))}$ is obtained from \cref{eq:SecondCumu} of \cref{def:approx-factorization-charactersA}, and the partition $\mu(\ribbon)\setminus\{\deg(s_1), \deg(s_2)\}$ has exactly two parts less than $\mu(\ribbon)$. On the other hand, if $\deg(s_1)=\deg(s_2)=-1$, then $v_{(\deg(s_1)|\deg(s_2))}\cdot v_{\mu(\ribbon)\setminus\{ \deg(s_1), \deg(s_2) \}} = v_{(-1|-1)}\cdot v_{\mu(\ribbon)} = -v_{\mu(\ribbon)}$.
\end{remark}

\begin{remark}
Similarly as~\cref{theo:ShapeDS} recovers Biane's result
from~\eqref{eq:Biane} in the special case of a fixed $\alpha=1$, \cref{theo:CovDS} recovers Śniady's result
from~\cite{Sniady2006c} (see also \cref{rem:CompWithSniady}). In
particular, the special case of our formula for the
covariance~\eqref{eq:CovDS} at $\alpha=1$ extends the formula for the case of the Plancherel measure found
in~\cite{IvanovOlshanski2002} to the general setting of random Young diagrams studied by Śniady
in~\cite{Sniady2006c}. Our formula does not appear
in~\cite{Sniady2006c}, where instead the
formulae for the covariance of free cumulants and characters was
given~\cite[Eq.~(3.6)]{Sniady2006c}; \eqref{eq:CovDS} can be used to
recover these formulae.
  \end{remark}

The proofs consist of two steps. In the first step, we show that the random Young diagrams 
sampled by the specific Jack characters constructed in \cref{ex:CDM} are the ``depoissonized'' 
versions of the random Young diagrams sampled with respect to the Jack--Thoma measures, 
and prove that they both have the same asymptotic behavior. 
The theme of \emph{depoissonization} has been explored many times in the literature,
e.g.~see~\cite{Johansson1998,BorodinOkounkovOlshanski2000}. 
In the second step, we show that the observables of interest in our model are polynomials in 
the variables $v_k,\, v_{(k|l)}$; this will finish the proof after combining it with the observation 
that if two polynomials agree on a large enough set, then the polynomials must coincide.

\subsection{Depoissonization of Jack--Thoma measures}

\begin{proposition}\label{prop:Depoissonization}
Let $\al>0,\,d\in\Z_{>0}$ be fixed.
Take any sequence $\vv = (1,v_2,v_3,\dots)\in\R^\infty$, and let 
$\chi_d:\Y_d\to\R$ be the function defined by $\chi_d(\mu) := v_\mu$, for all $\mu\in\Y_d$. 
Let $\mathbb{P}^\a_{\chi_d}$ be the (signed) measure on $\Y_d$ 
associated to $\chi_d$, according to \cref{def:first}.
Recall also that for any $u>0$, \cref{plancherel_jack} gives $M^\a_{u;\vv}$ 
as a (signed) measure on $\Y$.

The measure $\mathbb{P}^\a_{\chi_d}(\cdot)$ coincides with the 
measure $M^\a_{u;\vv_{u,\al}}(\cdot\,|\,\la\in\Y_d)$, 
conditioned on the size of $\la\in\Y_d$, where $u>0$ is arbitrary, and 
  \[ \vv_{u,\alpha} := 
  \left(\left(\frac{u}{\sqrt{\alpha}}\right)^{0}\!,\,v_2\cdot
  \left(\frac{u}{\sqrt{\alpha}}\right)^{1}\!,\, v_3\cdot
  \left(\frac{u}{\sqrt{\alpha}}\right)^{2},\dots\right).\]
In particular, this is a probability measure if the infinite tuple 
$(\sqrt{\alpha},\,\sqrt{\alpha}\cdot v_2,\,\sqrt{\alpha}\cdot v_3,\dots)$ 
determines an $\alpha$-Jack-positive specialization.
\end{proposition}

\begin{proof}
Using definition~\eqref{eq:character-Jack-unnormalized-zmiana} of the
irreducible Jack character $\chi^\a_\lambda $, the fact that the 
power sums $\{p_\mu\}_{\mu\in\Y}$ form an orthogonal basis of $\SSym$ 
with respect to the $\al$-deformed Hall scalar product, and formula~\eqref{eq:powersumnorm} 
for the norms of the power sums, we have
\begin{equation*}
\chi^\a_\lambda(\mu) = 	\al^{-(d - \ell(\mu))/2}\cdot
	\frac{z_\mu}{d!}\cdot\frac{\langle p_\mu, J_\la^\a \rangle}{\langle p_\mu, p_\mu\rangle} = \frac{\al^{-(d+\ell(\mu))/2}}{d!}\cdot\langle p_\mu,J^\a_\lambda\rangle,
\end{equation*}
for all $\la,\mu\in\Y_d$.
Then use this equation and the fact that the Jack symmetric functions $J^\a_\la$ form an orthogonal basis with norms given by \eqref{j_lambda_2} in order to obtain 
\begin{equation}\label{eq:IrrCharScalar}
p_\mu = \sum_{\la\in\Y_d}{ \frac{\langle p_\mu, J^\a_\la\rangle}{\langle J^\a_\la, J^\a_\la\rangle}\,
J^\a_\la } = \sum_{\la\in\Y_d}{ \frac{\langle p_\mu, J^\a_\la\rangle}{j^\a_\la}\, J^\a_\la } = 
\al^{(d+\ell(\mu))/2}\cdot d!\cdot\sum_{\la\in\Y_d}{ \frac{\chi^\a_\la(\mu)}{j^\a_\la}\, J^\a_\la }.
\end{equation}

Let $\rho:\SSym\to\R$ be the specialization determined by the infinite tuple $u\cdot\vv_{u,\al}$, i.e.
\[
\rho(p_k) := u\cdot (\vv_{u,\al})_k = u\cdot v_k\cdot\left( \frac{u}{\sqrt{\al}} \right)^{k-1}
= v_k\cdot u^k\cdot\al^{(1-k)/2},\quad \forall\, k\in\Z_{\ge 1},
\]
where we are setting $v_1:=1$.
Then $\rho(p_\mu)=v_{\mu}\cdot u^d\cdot\al^{(\ell(\mu)-d)/2}$ and, by our usual notation, $\rho(J^\a_\la)=J^\a_\la(u\cdot\vv_{u,\al})$.
Then applying $\rho$ to both sides of \cref{eq:IrrCharScalar}, using the assumption that $\chi_d(\mu) = v_\mu$, and rearranging terms we obtain
\begin{equation}\label{chi_expansion_1}
\chi_d(\mu) = \frac{\al^d d!}{u^d}\sum_{\la\in\Y_d}
\frac{J_{\la}^\a(u\cdot\vv_{u,\alpha})}{j^\a_\la}\cdot\chi^\a_\la(\mu).
\end{equation}
Recall that the Jack--Thoma measure $M^\a_{u; \vv_{u,\al}}$ from \cref{plancherel_jack} is the 
Jack measure with specializations $\rho_1(p_k)=u\cdot(\vv_{u,\al})_k$ and $\rho_2(p_k)=u\cdot\delta_{1,k}$, 
as observed in \cref{ex:CDMPois}.
Thus, as a direct consequence of Cauchy's identity~\eqref{eq:Cauchy}, the measure 
$M^\a_{u; \vv_{u,\al}}$ assigns $\Y_d\,$ a mass of 
$\,\exp\left(-\frac{u^2}{\al}\right)\cdot\frac{u^{2d}}{\al^d d!}$.
Hence, the Jack--Thoma measure, conditioned on the size of $\la\in\Y_d$, is given by:
\[ M^\a_{u;\vv_{u,\al}}(\la \,|\, \la \in \Y_d) = \frac{\al^d d!}{u^{2d}}\cdot
\frac{J_\la^\a(u\cdot\vv_{u,\al})\cdot u^d}{j^\a_\la}. \]
Plugging back into \eqref{chi_expansion_1} gives
\begin{equation}\label{chi_expansion_2}
\chi_d(\mu) = \sum_{\la\in\Y_d} M^\a_{u;\vv_{u,\al}}(\la \,|\, \la\in\Y_d) \
  \chi^\a_\lambda(\mu).
\end{equation}
This last equation is precisely the definition of the measure
$\PP^\a_{\chi_d}$ (compare \cref{def:first} and \eqref{chi_expansion_2}), therefore we conclude the desired equality of measures: $\PP^\a_{\chi_d}(\cdot) = M^\a_{u;\vv_{u,\al}}(\cdot \,|\, \la \in \Y_d)$.

Finally, if $(\sqrt{\alpha},\sqrt{\alpha}\cdot v_2, \sqrt{\alpha}\cdot v_3,\dots)$ 
determines an $\alpha$-Jack-positive specialization, then so does $u\cdot\vv_{u,\alpha}$, 
since they both belong to the same ray of the Thoma cone (see \cref{thm:KOO}).
Then \cref{plancherel_jack} shows that $M^\a_{u;\vv_{u,\al}}$ is a probability 
(not just a signed) measure, so the same is true for the conditional version 
$M^\a_{u;\vv_{u,\al}}(\cdot \,|\, \la \in \Y_d)$. The proof is then finished.
\end{proof}

\smallskip

Let $g\in\R\setminus\{0\}$ and $\vv = (v_1, v_2, \cdots)\in\R^\infty$ be such that 
$(g,\vv)$ is an admissible pair, in the sense of \cref{def:admissible_pair}, 
with $v_1=1$. Motivated by \cref{ex:CDM}, consider the following functions 
$\chi_{d;g,\vv}\colon\Y_d\to\R$:
\begin{equation}\label{eq:JackCDMHighLow}
\chi_{d; g, \vv}(\mu) := \left(\frac{g^2 d}{\big\lceil |g|d \big\rceil^2}\right)^{\!\frac{\|\mu\|}{2}}\cdot v_\mu.
\end{equation}

\begin{theorem}\label{theo:Depoissonization}
Let $(g,\vv)$ be an admissible pair with $v_1=1$, and let 
$(\chi_{d; g,\vv}\colon\Y_d \to \R)_{d\ge 1}$  be the sequence of functions 
given by \eqref{eq:JackCDMHighLow}. Then $\chi_{d; g,\vv}$ is a 
Jack character when $\al$ is equal to 
\begin{equation}\label{alpha_d}
\al = \begin{cases} g^2d, &\text{ if } g>0,\\
\frac{1}{g^2 d}, &\text{ if } g<0. \end{cases}
\end{equation}
This sequence $\al=\al(d)$ satisfies condition 
\eqref{eq:double-scaling-refined} with $g'=0$. 
Moreover, the sequence $(\chi_{d;g,\vv})_{d\ge 1}$ satisfies the enhanced AFP with $v_2, v_3, \cdots$ being the associated paramaters in 
\cref{eq:refined-asymptotics-characters} from 
\cref{def:approx-factorization-charactersA}., while the other parameters are $v_k'=v_{(k|l)}=0$, for all $k,l\ge 2$.

Let $\omega_{\Lambda_\infty}$ be the limit shape given by \cref{theo:LLNDS}, 
and let $\omega_{\Lambda_{g;\vv}}$ be the limit shape given by \cref{theo:LLNCDM} (for Jack--Thoma measures). 
Then 
\[ \omega_{\Lambda_\infty} = \omega_{\Lambda_{g;\vv}}.\]

Furthermore, the Gaussian process $\Delta_\infty$, given
by~\cref{theo:CLTDS} is determined by the explicit combinatorial formulas.
$\E[\langle \Delta_\infty,\, x^{k-2}\rangle]$, for all $k\ge 2$, is equal to the RHS of~\cref{MeanGaussian_2} where we set $v_1=1$. As for the covariances, we have the following formula, valid for all $k, l\ge 2$:
\begin{multline*}
\Cov\left( \langle \Delta_\infty,\, x^{k-2}\rangle, \langle \Delta_\infty,\, x^{l-2}\rangle \right) = \\
\frac{1}{(k-1)(l-1)}\Bigg(\sum_{\ribbon=(\Gamma_1, \Gamma_2)\in\mathbf{L}^{\conn}(k,l)\atop |\mathbf{P}(\ribbon)| = 1}\frac{v_{\mu(\ribbon)}}{|\mathbf{S}^0(\Gamma_1)|\cdot |\mathbf{S}^0(\Gamma_2)|}\prod_{i\geq 1} (i\cdot g)^{|\SSS^i_{\tiny\rightarrow}(\ribbon)|}
  \,i^{\,|\mathbf{P}_i(\ribbon)| } \\
  - \sum_{\ribbon=(\Gamma_1, \Gamma_2)\in\mathbf{L}(k, l)\atop
    |\mathbf{P}(\ribbon)| = 0}\frac{|\SSS_{-1}(\Gamma_1)|\cdot |\SSS_{-1}(\Gamma_2)|\cdot v_{\mu(\ribbon)}}{|\mathbf{S}^0(\Gamma_1)|\cdot |\mathbf{S}^0(\Gamma_2)|}\prod_{i
    \geq 1}(i\cdot g)^{|\SSS^i_{\tiny\rightarrow}(\ribbon)|}\Bigg).
\end{multline*}
\end{theorem}

\begin{proof}
It is evident that the sequence $(\al(d))_{d\ge 1}$ satisfies the 
assumption~\eqref{eq:double-scaling-refined}.
Let us verify that $\chi_{d; g,\vv}$ is a Jack character. Due to
\cref{prop:DualityAFP}, it is enough to verify it in case $g>0$. Let $u(d) := \lceil gd \rceil$.
As usual, denote $\al := \al(d),\, u:= u(d)$. 
\cref{prop:CDMPar2} shows that $M^\a_{u;\vv}$ is a probability measure. 
On the other hand, if we set $v^{(d)}_k := (\sqrt{\al}/u)^{k-1}\cdot v_k$, 
then the function $\chi_{d;g,\vv}$ can be expressed as
\begin{equation}\label{chi_dgv}
\chi_{d;g,\vv}(\mu) = v^{(d)}_\mu,\quad\textrm{for all }\mu\in\Y_d.
\end{equation}
Then \cref{prop:Depoissonization} implies that the measure 
$\mathbb{P}^\a_{\chi_{d;g,\vv}}(\cdot)$ associated to $\chi_{d;g,\vv}$ is equal 
to the conditional probability measure $M^\a_{u; \vv}(\cdot \,|\, \la\in\Y_d)$; 
therefore $\mathbb{P}^\a_{\chi_{d;g,\vv}}$ 
is also a probability measure, and consequently $\chi_{d;g,\vv}$ is a Jack character.

To check that $(\chi_{d;g,\vv})_{d\ge 1}$ satisfies the enhanced AFP,
note 
that each $\chi_{d;g,\vv}$ is multiplicative, meaning that 
$\chi_{d;g,\vv}(\mu^1\cdots\mu^k) = \chi_{d;g,\vv}(\mu^1)\cdots\chi_{d;g,\vv}(\mu^k)$. 
Therefore all cumulants $\kappa_n^{\chi_{d;g,\vv}}(k_1, \dots, k_n)$
with $n\ge 2$ vanish, in particular, \eqref{eq:SecondCumu} holds with
$v_{(k|l)}=0$. Recall from \cref{subsec:CharactersAFP} that in order
to make sense of the cumulants $\kappa_n^{\chi_{d;g,\vv}}(k_1, \dots, k_n)$,
for $k_1+\cdots+k_n < d$, we need to add an appropriate number of
$1$'s to construct partitions of size $d$. In particular,
\[ \kappa_1^{\chi_{d;g,\vv}}(k)\cdot d^{\frac{k-1}{2}} =
  \chi_{d;g,\vv}(k,1^{d-k}) \cdot d^{\frac{k-1}{2}} = 
v_k^{(d)}(v_1^{(d)})^{d-k}\cdot d^{\frac{k-1}{2}} =
(\sqrt{\al d}/u)^{k-1}\cdot v_k,\]
where the second equality is due to~\eqref{chi_dgv}. Therefore
\[ \kappa_1^{\chi_{d;g,\vv}}(k)\cdot
  d^{\frac{k-1}{2}}=v_k+o(d^{-1/2}),\]
for all $k\ge 2$, because $\sqrt{\al d}/u = 1+o(d^{-1/2})$ for our choices of $\al,\, u$. Then \eqref{eq:refined-asymptotics-characters-2} is satisfied with $v_k'=0$.
        
Denote the expectation w.r.t.~the conditional probability measure $M^\a_{u; \vv}(\cdot \,|\, \la\in\Y_d)$ by $\mathbb{E}^{\al; d}_{u; \vv}$. We can compute 
$\mathbb{E}^{\al,d}_{u;\vv}\left(B_{\ell_1}(\Lambda_{(\al;u)})\cdots
      B_{\ell_n}(\Lambda_{(\alpha;u)})\right)$ in the same way we
    computed $\mathbb{E}^\a_{u;\vv}\left(B_{\ell_1}(\Lambda_{(\al;u)})\cdots
   B_{\ell_n}(\Lambda_{(\alpha;u)})\right)$
   in the proof of~\cref{theo:Expectations}, except that 
   $\exp\!\left(-\frac{u^2}{\al}\right)$ and $\exp\!\left(\frac{u\cdot q_1}{\al}\right)$ 
	should be replaced by $d!\!\cdot\!\left(\frac{\alpha}{u^2}\right)^d$ and 
   $\left(\frac{u}{\alpha}\right)^d\!\cdot\!\frac{q_1^d}{d!}$, respectively.
	This gives 
\begin{multline}\label{exp_CMD_d}
\E^{\alpha;d}_{u;\vv}\left(B_{\ell_1}(\Lambda_{(\alpha;u)})\cdots
      B_{\ell_n}(\Lambda_{(\alpha;u)})\right) =\\
\sum_{\ribbon\in\mathbf{L}(\ell_1,\dots,\ell_n)\atop |\SSS^0(\ribbon)|=n} 
d_{\overline{|\SSS_{-1} (\ribbon)|}}\cdot
\bigg(\frac{\alpha}{u^2}\bigg)^{\!|\SSS_{-1} (\ribbon)|}\cdot {\bigone\bigtwo},
\end{multline}
where $d_{\overline{k}} := d(d-1)\cdots(d-k+1)$ is the falling factorial.
Since $d_{\overline{|\SSS_{-1} (\ribbon)|}} = d^{|\SSS_{-1} (\ribbon)|}(1 + O(d^{-1}))$, $\,\frac{\al-1}{u}= g(1+O(d^{-1}))$, and $\,\frac{\al}{u^2} = d^{-1}(1+O(d^{-1}))$,
	then the asymptotics of~\eqref{exp_CMD_d} is 
\begin{equation}\label{exp_d_prodB}
\E^{\al;d}_{u;\vv}\left(B_{\ell_1}
	(\Lambda_{(\al;u)})\cdots B_{\ell_n}(\Lambda_{(\al;u)})\right) =
	B^{\mu_{g;\vv}}_{\ell_1}\cdots B_{\ell_n}^{\mu_{g;\vv}}(1+O(d^{-1})).
\end{equation}
Recall that $\Lambda^\a_d = T_{\sqrt{\al/d},\,1/\sqrt{\al d}}{\,\la}\,$ and 
$\,\Lambda_{(\al; u)} = T_{\al/u,\,1/u}{\,\la}$, for any $\la\in\Y_d$. Then 
$\Lambda^\a_d = T_{u/\sqrt{\al d},\,u/\sqrt{\al d}}{\,\Lambda_{(\al; u)}}$, 
so that the scaling property \eqref{eq:scaling0} implies 
\[
B_{\ell}(\Lambda^\a_d) = \left(\frac{u}{\sqrt{\al d}}\right)^{\!\ell}\cdot B_{\ell}(\Lambda_{(\al;u)}), \ \text{ for all }\ell\in\Z_{\ge 1}.
\]
Since $\frac{u}{\sqrt{\al d}} = 1 + O(d^{-1})$ in our case, 
\cref{exp_d_prodB} shows 
\[ \lim_{d\to\infty}\E^{\al;d}_{u;\vv}
\left( B_{\ell_1}(\Lambda^\a_d) \cdots B_{\ell_n}(\Lambda^\a_d) \right)
= B_{\ell_1}^{\mu_{g;\vv}}\cdots B_{\ell_n}^{\mu_{g;\vv}}. \]
By \cref{prop:relations}, it follows that for any $\ell\in\Z_{\ge 1}$, we have
\begin{equation}\label{expected_moments}
\lim_{d\to\infty}\E^{\al;d}_{u;\vv}\left( M_\ell(\Lambda^\a_d) \right) = M_{\ell}^{\mu_{g;\vv}}.
\end{equation}
Recall that the transition measure $\mu_{\Lambda^\a_d}$ of $\Lambda^\a_d$ has moments $M_{\ell}(\Lambda^\a_d)$.
Since $\mu_{\Lambda^\a_d}$ is a random measure, we can define the average $\E^{\al;d}_{u;\vv}[\mu_{\Lambda^\a_d}]$, or simply $\E[\mu_{\Lambda^\a_d}]$; this is a deterministic probability measure with moments $\E^{\al;d}_{u;\vv}\left( M_\ell(\Lambda^\a_d) \right)$.
The limits \eqref{expected_moments} then show the convergence $\E[\mu_{\Lambda^\a_d}]\to\mu_{g;\vv}$ in the sense of moments. But since $\mu_{g;\vv}$ is uniquely determined by its moments, as shown in~\cref{lemma_moments}, then the convergence $\E[\mu_{\Lambda^\a_d}]\to\mu_{g;\vv}$ holds weakly, by virtue of \cref{theo:Billingsley}.

On the other hand, \cref{theo:LLNDS} proves that the profiles $\omega_{\Lambda^\a_d}$ converge to the limit shape $\omega_{\Lambda_\infty}$ uniformly, in probability. Because of the homeomorphism given by the Markov-Krein correspondence, namely \cref{theo:Meliot}, the previous convergence implies $\mu_{\Lambda^\a_d}\to\mu_{\Lambda_\infty}$ weakly, in probability. This implies the weak convergence $\E[\mu_{\Lambda^\a_d}]\to\mu_{\Lambda_\infty}$. But we already knew that $\E[\mu_{\Lambda^\a_d}]\to\mu_{g;\vv}$ weakly, therefore $\mu_{\Lambda_\infty} = \mu_{g;\vv}$. By the Markov-Krein correspondence, we conclude $\omega_{\Lambda_\infty} = \omega_{\Lambda_{g;\vv}}$, as desired.

\smallskip

Next compute $\Cov(\langle \Delta_\infty, x^{k-2}\rangle, \langle \Delta_\infty, x^{l-2}\rangle)=\lim_{d\to\infty}{\Cov^\a_{\chi_{d; g,\vv}}(\langle \Delta^\a_d, x^{k-2}\rangle, \langle \Delta^\a_d, x^{l-2}\rangle)}$.
We follow the same ideas as in \cref{theo:CLTCDM}, but now use \eqref{exp_CMD_d} instead of the formula from \cref{theo:Expectations}. 
The result will differ slightly because of the asymptotic identity: 
$d_{\overline{k+l}}-d_{\overline{k}}\cdot d_{\overline{l}} = 
-k\cdot l\cdot d^{k+l-1}+O(d^{k+l+2})$. 
This implies that the limiting covariance will be the same as the one found in \cref{theo:CLTCDM}, but corrected by the term which comes from the limit 
\[  \lim_{d \to \infty}\frac{d_{\overline{|\SSS_{-1}(\ribbon)|}} \,-\, d_{\overline{|\SSS_{-1}(\Gamma_1)|}}\cdot d_{\overline{|\SSS_{-1}(\Gamma_2)|}}}{d^{|\SSS_{-1}(\ribbon)|}} = 
-|\SSS_{-1}(\Gamma_1)||\SSS_{-1}(\Gamma_2)|  \]
that accompanies the Łukasiewicz ribbon paths $\ribbon$ without pairings, i.e.~with $\mathbf{P}(\ribbon) = \emptyset$:
\begin{multline*}
(k-1)(l-1)\cdot\lim_{d \to \infty} \!\Cov^\a_{\chi_{d; g,\vv}}
	\!\left( \langle \Delta^\a_d, x^{k-2}\rangle, \langle \Delta^\a_d, x^{l-2}\rangle \right) =\\
= \lim_{d \to \infty}{\!d\cdot \Cov^\a_{\chi_{d; g,\vv}}\!\left(S_k(\Lambda^\a_d), S_l(\Lambda^\a_d)\right)}
	= \lim_{d \to \infty}d\cdot \left(\frac{u}{\sqrt{\alpha\cdot d}}\right)^{\!\!k+l}\times\\
	\times\left\{\mathbb{E}^{\alpha;d}_{u;\vv}\left(S_{k}(\Lambda_{(\alpha;u)})
      S_l(\Lambda_{(\al;u)})\right) -\mathbb{E}^{\al;d}_{u;\vv}
	\left(S_k(\Lambda_{(\alpha;u)})\right)\cdot 
	\mathbb{E}^{\alpha;d}_{u;\vv}\left(S_l(\Lambda_{(\alpha;u)})\right)\!\right\} \\
	= \sum_{\ribbon=(\Gamma_1, \Gamma_2)\in\mathbf{L}^{\conn}(k,l)\atop 
					|\mathbf{P}(\ribbon)| = 1}
	\frac{v_{\mu(\ribbon)}}{|\mathbf{S}^0(\Gamma_1)|\cdot |\mathbf{S}^0(\Gamma_2)|}
	\,\prod_{i\geq 1} (i\cdot g)^{|\SSS^i_{\tiny\rightarrow}(\ribbon)|}
  \,i^{\,|\mathbf{P}_i(\ribbon)| }\\
  - \sum_{\ribbon=(\Gamma_1, \Gamma_2)\in\mathbf{L}(k, l)\atop |\mathbf{P}(\ribbon)| = 0}
	\frac{|\SSS_{-1}(\Gamma_1)|\cdot |\SSS_{-1}(\Gamma_2)|\cdot v_{\mu(\ribbon)}}{|\mathbf{S}^0(\Gamma_1)|\cdot |\mathbf{S}^0(\Gamma_2)|}
	\,\prod_{i\geq 1}(i\cdot g)^{|\SSS^i_{\tiny\rightarrow}(\ribbon)|}.
\end{multline*}
The calculation of $\E[\langle \Delta_\infty,\, x^{k-2}\rangle]=\lim_{d\to\infty}{\E^\a_{\chi_{d;g,\vv}}[\langle\Delta^\a_d, x^{k-2}\rangle]}$ follows the same ideas as in \cref{theo:CLTCDM}, and in this case one obtains exactly the RHS of~\cref{MeanGaussian_2}, but with $v_1=1$.
\end{proof}

\subsection{The algebra of $\alpha$-polynomial functions and universality of formulas}

\subsubsection{The algebra $\Poly$ of $\alpha$-polynomial functions}
We will use the following notations:
\begin{gather*}
\gamma := -\sqrt{\alpha}+\frac{1}{\sqrt{\alpha}},\\
X^\a_\ell(\la) := X_\ell(T_{\sqrt{\al},\frac{1}{\sqrt{\al}}}\la),\quad 
\forall\, X = M,B,R,S,\ \ \forall\, \la\in\Y,\ \ \forall\, \ell\in\Z_{\ge 1}.
\end{gather*}
We will treat each $X^\a_\ell$ as a real-valued function on $\Y$. 
Note that $X^\a_1 \equiv 0$ is identically zero, for all $X = M,B,R,S$.
The following polynomial algebra is called the \emph{algebra of $\al$-polynomial functions}:
\begin{equation*}
	\Poly := \QQ[\gamma,R^\a_2,R^\a_3,\dots].
\end{equation*}
Any element $x\in\Poly$ will be called an \emph{$\al$-polynomial function} (or just a 
polynomial function, for short), and it will be regarded as a function on $\Y$.
The product in $\Poly$ is given by (pointwise) multiplication of functions; 
if $x, y\in\Poly$, their product is denoted $x\times y$, or simply $xy$.

Using \eqref{eq:Moment}--\eqref{eq:Stransform}, it is not
difficult to see that
\[  \Poly = \QQ[\gamma,M^\a_2, M^\a_3,\dots]=
  \QQ[\gamma,B^\a_2,B^\a_3,\dots] =
  \QQ[\gamma,S^\a_2,S^\a_3,\dots].\]
In particular, if we let $X^\a_\mu := \prod_{i\ge 1}{X^\a_{\mu_i}}$, for all $X=M,R,B,S$ and $\mu\in\Y$, then we can write $\Poly = \bigoplus_{\ell\in\Z_{\ge 0}}{V_\ell} = \bigoplus_{\ell\in\Z_{\ge 0}}{\tilde{V}_\ell}$, where $V_\ell,\, \tilde{V}_\ell$ are the following subspaces:
\begin{equation}\label{def_Vs}
V_\ell := \Span_{\QQ}\{\gamma^k\cdot X^\a_\mu\colon |\mu|=\ell, k
  \in \Z_{\geq 0}\},\quad \tilde{V}_\ell := \Span_{\QQ}\{\gamma^k\cdot
  X^\a_\mu\colon k+|\mu|=\ell\}.
\end{equation}
Note that because of \cref{prop:relations}, the definitions~\eqref{def_Vs} of $V_\ell,\, \tilde{V}_\ell$ are the same regardless of whether we use $M,R,B$ or $S$ for $X$. We say that a nonzero $x\in V_\ell$ is homogenous of degree $\deg_1(x) = \ell$, and a nonzero $x\in\tilde{V}_\ell$ is homogenous of degree $\deg_2(x) = \ell$.

It will be important for us that $\Poly$ is linearly generated (over $\QQ[\gamma]$) 
by the following so-called \emph{normalized Jack characters}.

\begin{definition}\label{def:jack-character-classical}
	Let $\alpha>0$ and $\mu\in\Y$ be given. 
	For any $\lambda\in\Y$, the value of the \textbf{normalized Jack character} 
	$\Ch_\mu^\a(\lambda)$ is given by:
	\begin{equation}\label{eq:what-is-jack-character-zmiana}
	\Ch_{\mu}^\a(\lambda):=
	\begin{cases}
	|\la|_{\overline{|\mu|}}\cdot \chi^\a_\lambda(\mu), &\text{if }|\la| \ge |\mu| ,\\
	0, & \text{if }|\lambda| < |\mu|,
	\end{cases}
	\end{equation}
	where $\chi^\a_\lambda(\mu)$ is the irreducible Jack character
	\eqref{eq:character-Jack-unnormalized-zmiana} understood with
        the convention~\eqref{eq:extended-character}.
	In particular, for the empty partition $\mu=\emptyset$, the corresponding
	normalized Jack character is the constant function $\Ch_\emptyset^\a\equiv 1$, due to~\eqref{norm_jack}.
\end{definition}

\begin{theorem}[\cite{DolegaFeray2016}]\label{theo:DF16}
We have:
\begin{align*}
\Poly = \Span_{\QQ}&\{\gamma^k\Ch^\a_\mu \mid k\in\Z_{\ge 0},\, \mu\in\Y \},\\
\deg_1\left( \Ch_{\mu}^\a - \prod_{i=1}^{\ell(\mu)}R^\a_{\mu_i+1}\right) &< |\mu|+\ell(\mu),\\
\Ch_{\mu}^\a &\in \bigoplus^{\lfloor\frac{\mu+\ell(\mu)}{2}\rfloor}_{i = 0}\tilde{V}_{|\mu|+\ell(\mu)-2i}.
\end{align*}
\end{theorem}

Note that \cref{theo:DF16} implies that $\Poly = \QQ[\gamma,\Ch^\a_{(1)}, \Ch^\a_{(2)},\dots]$.
We define also a different algebra structure on the set $\Poly$, defined on its basis by
\[ \gamma^k\Ch^\a_\mu \bullet \gamma^l\Ch^\a_\nu :=
\gamma^{k+l}\Ch^\a_{\mu\cdot\nu}, \]
 and extended by linearity. Here, $\mu\cdot\nu$ is the product of
 two partitions $\mu,\nu\in\Y$, as defined in \cref{sec:afp}. We
 denote this algebra by $\Poly^\bullet$.

		Let $\chi_d:\Y_d\to\R$ be a Jack character with associated probability measure 
		$\PP^\a_{\chi_d}$ given by \cref{def:first}; we denote the expectation w.r.t.~this 
		measure by $\E^\a_{\chi_d}$.
      Each polynomial function $x\in\Poly$ can be restricted to the subset $\Y_d\subset\Y$ 
		of Young diagrams of size $d$, therefore $x$ can be treated as 
		a random variable on $\Y_d$, so it makes sense to compute the
      expectation $\E^\a_{\chi_d}(x)$. In particular, from the definition of
      the normalized Jack character, one has
      \begin{equation}\label{eq:IdentityCh}
        \E^\a_{\chi_d}(\Ch^\a_\mu) = d_{\overline{|\mu|}}\cdot\chi_d(\mu),\ \text{ for all }\mu\in\Y.
        \end{equation}
		More generally, for any polynomial functions $x_1,\dots,x_n
      \in \Poly$, we can define the cumulants
      $\kappa^{\chi, \times}_n(x_1,\dots,x_n)$ and
      $\kappa^{\chi, \bullet}_n(x_1,\dots,x_n)$ by the same equation
      \eqref{eq:log-laplace} used previously, but replacing $\chi$ by $\E^\a_\chi$, 
		and the product of partitions by the products $\times$ and $\bullet$, respectively:
\begin{align*}
\kappa^{\chi, \times}_n(x_1,\dots,x_n) &:= 
\left. \frac{\partial^n}{\partial t_1 \cdots \partial t_n}
\log \E^\a_\chi\left( \exp_\times(t_1 x_1+\cdots+t_n x_n) \right) \right|_{t_1=\cdots=t_n=0},\\
\kappa^{\chi, \bullet}_n(x_1,\dots,x_n) &:= 
\left. \frac{\partial^n}{\partial t_1 \cdots \partial t_n}
\log \E^\a_\chi\left( \exp_\bullet(t_1 x_1+\cdots+t_n x_n) \right) 
\right|_{t_1=\cdots=t_n=0},
\end{align*}
where
      \[ \exp_\times(x) := \sum_{n \geq 0}\frac{\overbrace{x\times x
            \cdots \times x}^{n \text{ times}}}{n!},
        \qquad \exp_\bullet(x) := \sum_{n \geq 0}\frac{\overbrace{x\bullet x \bullet
            \cdots \bullet x}^{n \text{ times}}}{n!}.\]

The following equivalent characterization of the AFP will be important.

\begin{theorem}[Theorem~2.3 from \cite{DolegaSniady2019}]\label{theo:AFPCharacterization}
Let $(\chi_d:\Y_d\to\R)_{d\ge 1}$ be any sequence of Jack characters.
The following conditions are equivalent:
		\begin{enumerate}[label=(\emph{\Alph*})]
			\setlength\itemsep{1em}
			\item for each integer $n\geq 1$, and all integers $k_1,\dots,k_n\geq 2$, the limit
			\begin{equation*}
			\lim_{d\to\infty} \kappa^{\chi_d}_n(k_1,\dots,k_n)\ d^{\frac{k_1+\cdots+k_n+n-2}{2}} 
			\end{equation*}
			exists and is finite;
			\item for each integer $n\geq 1$, and all $x_1,\dots,x_n\in\Poly$, the limit
			\begin{equation*}
			\lim_{d\to\infty} \kappa^{\chi_d, \times}_n(x_{1},\dots,x_{n})\ 
			d^{- \frac{\deg_2(x_1)+ \cdots + \deg_2(x_n) - 2(n-1)}{2}} 
			\end{equation*}
			exists and is finite;
			\item for each integer $n\geq 1$, and all $x_1,\dots,x_n\in\Poly^\bullet$, the limit
			\begin{equation*}
			\lim_{d\to\infty} \kappa^{\chi_d, \bullet}_n(x_{1},\dots,x_{n})\ 
			d^{- \frac{\deg_2(x_1)+ \cdots + \deg_2(x_n) - 2(n-1)}{2}} 
			\end{equation*}
			exists and is finite.
	\end{enumerate}
\end{theorem}

\begin{lemma}\label{lem:pomocnicz}
Let $\al=\al(d)$ be a sequence of positive reals such that
\eqref{eq:double-scaling-refined} is satisfied, for some
$g\in\R$. Also, let $\chi_d\colon\Y_d\to\R$ be Jack characters such
that the sequence $(\chi_d)_{d\ge 1}$ fulfills the AFP (enhanced AFP, respectively) and
let $v_2,v_3,\dots,$ (and $v_2',v_3',\dots$ ,respectively) be the associated parameters given by \cref{def:approx-factorization-charactersA}. 
Suppose that $x \in \tilde{V}_\ell \oplus\bigoplus_{i \leq \ell-2}\tilde{V}_i$. 
Then the following conditions hold true.
	\begin{itemize}
		\item There exists a homogenous polynomial 
		$P_x \in \QQ[x_1,\dots,x_{\ell}]$ of degree $\ell$, where
      $\deg(x_i) := i$, such that
\begin{equation*}
\E^\a_{\chi_d}(x) = P_x(g,v_1,\dots,v_{\ell-1})\bigg|_{v_1=1}\cdot d^{\frac{\ell}{2}} 
+ o(d^{\frac{\ell}{2}}), \text{ if }(\chi_d)_{d\ge 1}\,
\text{ fulfills the AFP, }
\end{equation*}
and
\begin{multline*}
\E^\a_{\chi_d}(x) = P_x(g,v_1,\dots,v_{\ell-1})\bigg|_{v_1=1}\!\cdot d^{\frac{\ell}{2}}\\
+ \left( g'\frac{\partial}{\partial g} + \sum_{i\ge 2}{v_i'\frac{\partial}{\partial v_i}} \right)\! P_x(g,v_1,\dots,v_{\ell-1})\bigg|_{v_1=1}\!\cdot d^{\frac{\ell-1}{2}}\\
+ o(d^{\frac{\ell-1}{2}}), \text{ if }(\chi_d)_{d\ge 1}\,\text{ fulfills the enhanced AFP. }
\end{multline*}
   
 \item Expand $P_x$ in monomials:
      \begin{equation}
        \label{eq:DefOfP_x}
        P_x = 
		\sum_{m_1,\dots,m_\ell\in\Z_{\ge 0},\atop m_1+2m_2+\cdots+\ell m_\ell =
                  \ell}a_{m_1,\dots,m_\ell}\,x_1^{m_1}\cdots x_\ell^{m_\ell}.
                \end{equation}
       Then \[ x - \!\!\sum_{m_1,\dots,m_\ell\in\Z_{\ge 0}\atop
		m_1+2m_2+\cdots+\ell m_\ell = \ell}
		a_{m_1,\dots,m_\ell}\cdot
                (-\gamma)^{m_1}\cdot\Ch^\a_{(1^{m_2},\dots,(\ell-1)^{m_\ell})}
                \in \bigoplus_{i \leq \ell-2}\tilde{V}_i.\]
	\end{itemize}
\end{lemma}

\begin{proof}
\emph{Step 1.} For any $\mu\in\Y$, we claim that 
\begin{equation}\label{claim_lim_chars}
\E^\a_{\chi_d}(\Ch^\a_\mu) - 
\prod_{i=1}^{\ell(\mu)}{\E^\a_{\chi_d}(\Ch^\a_{(\mu_i)})} = 
O\left(d^{\frac{|\mu|+\ell(\mu)-2}{2}}\right).
\end{equation}

Let $r := \ell(\mu)$. The moment-cumulant formula~\eqref{eq:cumu}
shows the ``$\bullet$-moment'' $\E^\a_{\chi_d}(\Ch^\a_\mu) = 
\E^\a_{\chi_d}(\Ch^\a_{(\mu_1)}\bullet \cdots \bullet \Ch^\a_{(\mu_r)})$
is given by the formula
\[  \E^\a_{\chi_d}(\Ch^\a_\mu) = \sum_{\pi\in\SP([1 \,..\, r])}
{ \prod_{B\in\pi} \kappa^{\chi_d, \bullet}_{|B|}\left( \left\{ \Ch^\a_{(\mu_i)}:i\in B \right\} \right) },  \]
where the sum is over set-partitions of $[1 \,..\, r]$. 
The term associated to $\pi = \{ \{1\}, \cdots, \{r\} \}$ is 
$\prod_{i=1}^r{\kappa_1^{\chi_d, \bullet}(\Ch^\a_{(\mu_i)})}
= \prod_{i=1}^r{\E^\a_{\chi_d}(\Ch^\a_{(\mu_i)})}$.
Then our claim \eqref{claim_lim_chars} will be proved if we verify that 
\begin{equation}\label{limit_chis}
\prod_{B\in\pi}
{ \kappa^{\chi_d, \bullet}_{|B|}\left( \left\{ \Ch^\a_{(\mu_i)}:i\in B \right\} \right)} = 
O\left(d^{\frac{|\mu|+\ell(\mu)-2}{2}}\right),
\end{equation}
for all $\pi\in\SP([r])$, $\pi\ne\{ \{1\}, \cdots, \{r\} \}$.
\cref{theo:DF16} shows that $\deg_2(\Ch^\a_{(\mu_i)}) = \mu_i + 1$, 
therefore \cref{theo:AFPCharacterization} implies that 
$\kappa^{\chi_d, \bullet}_{|B|}\left( \{ \Ch^\a_{(\mu_i)}:i\in B \} \right)
\cdot d^{-\frac{\sum_{i\in B}{(\mu_i+1)} - 2(|B|-1) }{2}} = O(1)$, as $d\to\infty$.
The reader can check that this estimate proves the desired \eqref{limit_chis}.

\smallskip
\emph{Step 2.} Set $v_1:=1$. For any $\mu=(1^{p_1}, 2^{p_2}, \cdots)\in\Y$, we claim 
\begin{align}
\E^\a_{\chi_d}(\Ch^\a_\mu) =& 
\,(v_1^{p_1}\cdot v_2^{p_2}\cdots)\cdot d^{\frac{|\mu|+\ell(\mu)}{2}} +
o(d^{\frac{|\mu|+\ell(\mu)}{2}}), \text{ if } (\chi_d)_{d\ge 1}\,
  \text{ fulfills the AFP, }\nonumber\\
\E^\a_{\chi_d}(\Ch^\a_\mu) =& 
\,(v_1^{p_1}\cdot v_2^{p_2}\cdots)\cdot d^{\frac{|\mu|+\ell(\mu)}{2}} +
\left(\sum_{i\geq 2}{v_i'\frac{\partial}{\partial v_i}}\right)
(v_1^{p_1}\cdot v_2^{p_2}\cdots)\cdot d^{\frac{|\mu|+\ell(\mu)-1}{2}}\label{claim_lim_chars_2}\\
&+ o(d^{\frac{|\mu|+\ell(\mu)-1}{2}}), \text{ if } (\chi_d)_{d\ge 1}\,
\text{ fulfills the enhanced AFP.}  \nonumber
\end{align}

We prove \eqref{claim_lim_chars_2} only when $\chi_d$ fulfills
the enhanced AFP, since the case of AFP is a simplified version of
this proof.

From \eqref{eq:IdentityCh} and the enhanced AFP, we have that for any integer $k\ge 2$:
\[  \E^\a_{\chi_d}(\Ch^\a_{(k)})= 
d_{\overline{k}}\cdot \chi_d(k)= 
\chi_d(k)\cdot d^{\frac{k-1}{2}}\cdot
(d^{\frac{k+1}{2}}+O(d^{\frac{k-1}{2}})) = v_k \cdot d^{\frac{k+1}{2}}
+ (v'_k +o(1))\cdot d^{\frac{k}{2}}.\]
For $k=1$, we have $\E^\a_{\chi_d}(\Ch^\a_{(1)})\cdot d^{-1} = \chi_d(1) = 1 = v_1$. 
Consequently,
\begin{gather}
\prod_{i=1}^{\ell(\mu)}
\!\left(  \E^\a_{\chi_d}(\Ch^\a_{(\mu_i)}) \right) = 
\prod_{i=1}^{\ell(\mu)}
\!\left(  v_{\mu_i} \cdot d^{\frac{\mu_i+1}{2}}
+ (v'_{\mu_i} +o(1))\cdot d^{\frac{\mu_i}{2}}\right)\nonumber\\
= (v_1^{p_1}\cdot v_2^{p_2}\cdots)\cdot d^{\frac{|\mu|+\ell(\mu)}{2}} 
+ \left(\sum_{i\geq 2}{v_i'\frac{\partial}{\partial v_i}}\right)\!
(v_1^{p_1}\cdot v_2^{p_2}\cdots)\cdot d^{\frac{|\mu|+\ell(\mu)-1}{2}}
  + o(d^{\frac{|\mu|+\ell(\mu)-1}{2}}). \label{claim_lim_chars_3}
\end{gather}
The desired \eqref{claim_lim_chars_2} then follows from 
\eqref{claim_lim_chars_3} and the conclusion \eqref{claim_lim_chars} from Step 1.

\smallskip
\emph{Step 3.} Set $v_1:=1$. For any $p\in\Z_{\ge 0}$ and
$\mu=(1^{p_1}, 2^{p_2}, \cdots)\in\Y$, we claim that
\begin{align}
\E^\a_{\chi_d}((-\ga)^p\cdot\Ch^\a_\mu) &= 
g^p (v_1^{p_1}\cdot v_2^{p_2}\cdots)\!\cdot\! d^{\frac{p+|\mu|+\ell(\mu)}{2}} \! +
o(d^{\frac{p+|\mu|+\ell(\mu)}{2}}), \text{ if } (\chi_d)_{d\ge 1}
  \text{ fulfills AFP },\nonumber\\
\E^\a_{\chi_d}((-\ga)^p\cdot\Ch^\a_\mu) &= 
g^p (v_1^{p_1}\cdot v_2^{p_2}\cdots)\!\cdot\! d^{\frac{p+|\mu|+\ell(\mu)}{2}} \! +
\left( g'\frac{\partial}{\partial g} + \sum_{i\ge 2}{v_i'\frac{\partial}{\partial v_i}} \right)
                             g^p (v_1^{p_1}\cdot v_2^{p_2}\cdots)\nonumber\\
  \cdot d^{\frac{p+|\mu|+\ell(\mu)-1}{2}}&+ o(d^{\frac{p+|\mu|+\ell(\mu)-1}{2}}), \text{ if } (\chi_d)_{d\ge 1}\text{ fulfills enhanced AFP. }\label{claim_lim_chars_4}
\end{align}
As before, we focus on the case of enhanced AFP. Then this follows from \eqref{claim_lim_chars_2} in Step 2, and the fact that 
$\frac{-\ga}{\sqrt{d}} = \frac{\sqrt{\al} - \frac{1}{\sqrt{\al}}}{\sqrt{d}} = g+g'\cdot d^{-\frac{1}{2}}+o(d^{-\frac{1}{2}})$, as $d\to\infty$.

\smallskip
\emph{Step 4.} Recall that $\{\gamma^p\Ch^\a_\mu \mid p\in\Z_{\ge 0},\, \mu\in\Y\}$ 
is a $\QQ$-basis of $\Poly$. Since $x \in \tilde{V}_\ell \oplus\bigoplus_{i \leq \ell-2}\tilde{V}_i$, then \cref{theo:DF16} implies
that there exist $a_{m_1,\dots,m_\ell}\in\QQ$ such that
\[ x - \sum_{m_1+2m_2+\cdots+\ell m_\ell =\ell}
a_{m_1,\dots,m_\ell}\cdot
(-\gamma)^{m_1}\cdot\Ch^\a_{(1^{m_2},\dots,(\ell-1)^{m_\ell})}
\in \bigoplus_{i \leq \ell-2}\tilde{V}_i.\]

Note that $\deg_2((-\ga)^p\cdot\Ch^\a_\mu) = p+|\mu|+\ell(\mu)$, 
therefore \eqref{claim_lim_chars_4} implies that for any $y\in \bigoplus_{i \leq \ell-2}\tilde{V}_i$ one 
has $\E^\a_{\chi_d}(y) = o(d^{\frac{\ell-1}{2}})$, and moreover:
\begin{align*}
        \E^\a_{\chi_d}(x)
      = P_x(g,v_1,\dots,v_{\ell-1})\cdot
        d^{\frac{\ell}{2}} &+o(d^{\frac{\ell}{2}}), \text{ if } (\chi_d)_{d\ge 1}\,
        \text{ fulfills the AFP,}\\ 
        \E^\a_{\chi_d}(x)
      = P_x(g,v_1,\dots,v_{\ell-1})\cdot
        d^{\frac{\ell}{2}}
&+ \left( g'\frac{\partial}{\partial g} + \sum_{i\geq 2}{v_i'\frac{\partial}{\partial v_i}} \right) 
P_x(g,v_1,\dots,v_{\ell-1})\cdot d^{\frac{\ell-1}{2}}\\        &+o(d^{\frac{\ell-1}{2}}), \text{ if } (\chi_d)_{d\ge 1}\,
        \text{ fulfills the enhanced AFP,}
\end{align*}
where $P_x$ is defined by \eqref{eq:DefOfP_x}. This finishes the proof.
\end{proof}

\begin{corollary}\label{lem:PolynomnialityOfMoments}
Let $\al=\al(d)$ be a sequence of positive reals such that \eqref{eq:double-scaling-refined} is satisfied, for some $g\in\R$. Also, let $\chi_d\colon\Y_d\to\R$ be Jack characters such that the sequence $(\chi_d)_{d\ge 1}$ fulfills the AFP and let $v_2,v_3,\dots$ be the associated parameters given by \cref{def:approx-factorization-charactersA}. 
Let $\omega_{\Lambda_\infty}$ be the corresponding limit shape given by \cref{theo:LLNDS}, and let $\mu_{\Lambda_\infty}$ be its transition measure.

Then the moment $\int_\R{x^\ell \mu_{\Lambda_\infty}(dx)}$ is a 
polynomial in $g,v_2,\dots,v_{\ell-1}$.
\end{corollary}

\begin{proof}
Let $\mu_{\Lambda^\a_d}$ be the transition measure of $\Lambda^\a_d=T_{\sqrt{\frac{\al}{d}},\,\sqrt{\frac{1}{\al d}}}\la$. If $\la$ is $\mathbb{P}_{\chi_d}$-distributed, then $\mu_{\Lambda^\a_d}$ is a random probability measure, but the expectation $\mathbb{E}_{\chi_d}[\mu_{\Lambda^\a_d}]$ is a deterministic probability measure with moments $\mathbb{E}_{\chi_d}[M_\ell(\Lambda^\a_d)]$.
We shall need two facts:

$\bullet$ The convergence $\mathbb{E}_{\chi_d}[\mu_{\Lambda^\a_d}]\to\mu_{\Lambda_\infty}$ holds in the sense of moments, as $d\to\infty$.

\noindent In the proof of \cref{theo:Depoissonization}, we argued that $\mathbb{E}_{\chi_d}[\mu_{\Lambda^\a_d}]\to\mu_{\Lambda_\infty}$ weakly. In general, weak convergence does not imply convergence of moments. However, this is true if the sequences of moments $\{\mathbb{E}_{\chi_d}[M_\ell(\Lambda^\a_d)]\}_d$ are uniformly bounded (since this would imply uniform integrability).
In \cite[(5.5)]{DolegaSniady2019}, it is shown that the all limits of the form $\lim_{d\to\infty}{\mathbb{E}_{\chi_d}[S_{k_1}(\Lambda^\a_d)\cdots S_{k_n}(\Lambda^\a_d)]}$ exist and are finite.
Then by \cref{prop:relations}, all limits $\lim_{d\to\infty}{\mathbb{E}_{\chi_d}[M_\ell(\Lambda^\a_d)]}$ also exist and are finite; this implies the desired boundedness.

$\bullet$ $M_\ell(\Lambda^\a_d) 
= M_\ell(T_{\sqrt{\frac{\al}{d}},\,\sqrt{\frac{1}{\al d}}}\la) 
= d^{-\frac{\ell}{2}}\cdot M_\ell( T_{\sqrt{\al},\,\frac{1}{\sqrt{\al}}}\la) 
= d^{-\frac{\ell}{2}}\cdot M^\a_\ell(\la)$.

\noindent This is a result of the scaling property \eqref{eq:scaling0}.

\smallskip

Putting these facts together gives:
\begin{equation}\label{eq:PomocCor5.10}
  \int_\R{x^\ell \mu_{\Lambda_\infty}(dx)} = 
\lim_{d\to\infty}{\E^\a_{\chi_d} \left[ M_\ell(\Lambda^\a_d) \right]} =
\lim_{d\to\infty}{d^{-\frac{\ell}{2}}\cdot
  \E^\a_{\chi_d}\left[M^\a_\ell\right]},
\end{equation}
for any $\ell\in\Z_{\ge 1}$. 
The proof is then finished by \cref{lem:pomocnicz} applied to $x=M^\a_\ell$.
\end{proof}

\begin{lemma}\label{equal_polys}
Let $\ell\in\Z_{\geq 1}$, and $P, Q\in\R[x_1, \dots, x_\ell]$ be polynomials such that
\begin{equation}\label{P_equal_Q}
P(u_1, \dots, u_\ell) = Q(u_1, \dots, u_\ell),
\end{equation}
for all $(u_1, \dots, u_\ell)$ belonging to an open subset $U\subseteq\R^{\ell}$. Then $P = Q$.

In particular, if \eqref{P_equal_Q} holds whenever $u_1 = g,\, u_k = \left(a_1^k+\dots +a_{\ell-1}^k\right),\ k=2, \dots, \ell$, 
for all $a_1 > \dots > a_{\ell-1} > 0$, $\sum_{i=1}^{\ell-1}a_i
< 1$, and $g \in W$, for some open set $W\subset\R$, then $P=Q$.
\end{lemma}
\begin{proof}
The first statement claims that if the $\ell$-variate polynomial $P - Q$ vanishes on an 
open subset of $\R^\ell$, then it must be the zero polynomial: this is a well-known fact. 
For the second statement, note that the map $\R^\ell\to\R^\ell$ given by 
\begin{equation}\label{powers_map}
(g, a_1, a_2, \dots, a_{\ell-1}) \mapsto \left(g,\sum_{i=1}^{\ell-1}{a_i^2},\ \sum_{i=1}^{\ell-1}{a_i^3}, \cdots,\ \sum_{i=1}^{\ell-1}{a_i^\ell} \right),
\end{equation}
has nonzero Jacobian whenever all $a_i$'s are distinct and nonzero, in particular,
for all points in the open set $V\subseteq\R^{\ell}$ of points defined by the inequalities $\ell^{-1}>a_1>\dots>a_{\ell-1}>0$.
Note that this set $V$ is contained in the set described in the second statement of the lemma (the condition $\ell^{-1}>a_1$ ensures that $\sum_{i=1}^{\ell-1}a_i<1$),
and it is mapped into an open subset $U\subseteq\R^\ell$
by the map \eqref{powers_map}. As a result, the first statement can be applied and implies that $P=Q$.
\end{proof}

\begin{proof}[Proof of \cref{theo:ShapeDS}]
For $\ell=1$, both sides of \eqref{mu_moments:DS} are equal to $0$. 
Assume in the remaining that $\ell \geq 2$. Notice that the RHS 
of~\eqref{mu_moments:DS} is a polynomial $Q(g,v_2,\dots,v_{\ell-1})$, while 
\cref{lem:PolynomnialityOfMoments} says that the LHS 
of~\eqref{mu_moments:DS} is also a polynomial $P(g,v_2,\dots,v_{\ell-1})$. The goal is to prove that $P=Q$.

\cref{theo:Depoissonization}  implies that for any admissible pair $(g, \vv)$ with $v_1=1$, there exists a 
sequence of characters $(\chi_{d;g,\vv})_{d\ge 1}$ such that the corresponding 
limit shape $\omega_{\Lambda_\infty}$ (guaranteed by \cref{theo:LLNDS})
is equal to the limit shape $\omega_{\Lambda_{g; \vv}}$ arising from 
limits of Jack--Thoma measures. Then for such pairs $(g, \vv)$, 
the transition measures $\mu_{\Lambda_\infty},\,\mu_{g;\vv}$ 
are also the same, hence their moments are equal, which proves that 
$P(g, v_2, \cdots, v_{\ell-1}) = Q(g, v_2, \cdots, v_{\ell-1})$ in this case 
(the fact that the moments of $\mu_{g;\vv}$ are given by the RHS 
of \eqref{mu_moments:DS} is shown in \cref{theo:LLNCDM}).
Further, \cref{def:admissible_pair} and \cref{thm_jack_spec} show that if 
$g>0,\, v_1=1,\, v_k=\sum_{i=1}^{\ell-1}{a_i^k},\ \forall\,k\ge 2$, for some 
real numbers $a_1>\dots>a_{\ell-1}>0$, $\sum_{i=1}^{\ell-1}a_i < 1$, 
then $(g, \vv)$ is an admissible pair with $v_1=1$. But then \cref{equal_polys} 
implies the desired equality of polynomials $P=Q$, finishing the proof.
\end{proof}

The following corollary will be important for proving \cref{theo:CovDS}.

\begin{corollary}\label{cor:MSTop}
The following formulas hold true:
	\begin{align*}
	M^\a_{\ell} &= \sum_{\Gamma\in\mathbf{L}_0(\ell)} \prod_{i\ge 1} (-i\gamma)^{|\SSS_{\tiny\rightarrow}^i(\Gamma)|}\cdot \Ch^\a_{\mu(\Gamma)} \,+\, \text{ terms of $\deg_2$ at most $(\ell-2)$},\\
	S^\a_{\ell} &= \sum_{\Gamma\in\mathbf{L}_0(\ell)} \frac{1}{|\SSS^{0}(\Gamma)|}\prod_{i\ge 1} (-i\gamma)^{|\SSS_{\tiny\rightarrow}^i(\Gamma)|}\cdot \Ch^\a_{\mu(\Gamma)} \,+\,  \text{ terms of $\deg_2$ at most $(\ell-2)$}.
	\end{align*}
\end{corollary}

\begin{proof}
By applying \cref{lem:pomocnicz} to $x = M^\a_{\ell}$, we can conclude
that the first identity would follow from the following equality
\[ \lim_{d\to\infty}{d^{-\frac{\ell}{2}}\cdot
  \E^\a_{\chi_d}\left[M^\a_\ell\right]} = \sum_{\Gamma\in\mathbf{L}_0(\ell)}
\prod_{i\ge 1}{(i\cdot g)^{|\SSS_{\tiny\rightarrow}^i(\Gamma)|}\,
  v_i^{\,|\SSS_i(\Gamma)|}},\]
which in turn follows from~\eqref{eq:PomocCor5.10}
and~\cref{theo:ShapeDS}. The second identity can be proven in a similar way as the first one, but additionally we need to use the relations between
moments, Boolean cumulants, and the fundamental functionals of shape
given by \cref{prop:relations}. Indeed, that proposition has the following combinatorial
interpretation: if moments are expressed as sums of weighted Łukasiewicz paths, then 
Boolean cumulants are instead sums of weighted Łukasiewicz paths not touching $0$ except at their starting and
ending points; similarly, the fundamental functionals of shape are sums of weighted Łukasiewicz paths 
with the additional multiplicative weight of $1/n$, where $n+1$ is the number of points touching $0$.
In particular, we have
\[ \lim_{d\to\infty}{d^{-\frac{\ell}{2}}\cdot
  \E^\a_{\chi_d}\left[S^\a_\ell\right]} = \sum_{\Gamma\in\mathbf{L}_0(\ell)}
\frac{1}{|\SSS^{0}(\Gamma)|}\prod_{i\ge 1}{(i\cdot g)^{|\SSS_{\tiny\rightarrow}^i(\Gamma)|}\,
  v_i^{\,|\SSS_i(\Gamma)|}},\]
which finishes the proof.
\end{proof}

\begin{remark}
  Shortly after finishing this paper, the second author with Ben Dali have
  found an exact (not just the top-degree term) formula expressing the
  basis $M^\a_\mu$ as a linear
  combination of the normalized Jack characters $\Ch^\a_{\nu}$ in terms of ribbon Łukasiewicz paths --
  see~\cite[Theorem 3.6]{BenDaliDolega2023}.
  \end{remark}

\subsubsection{Central Limit Theorem}

Before proving \cref{theo:CovDS}, we need the following technical lemma. 
For any partition $\mu$, and $1\leq i\leq\ell(\mu)$, we denote the partition $\mu\setminus\{\mu_i\}$ by $\mu^-_i$.

\begin{lemma}
For any integers $k, l\ge 0$, and partitions $\mu,\nu$, we have
\begin{equation}\label{eq:KumuOfCh}
    \lim_{d \to \infty}d^{-\frac{|\mu|+|\nu|+\ell(\mu)+\ell(\nu)+k+l-2}{2}}\kappa^{\chi_d, \bullet}_2 (\gamma^k\Ch^{\a}_\mu,\gamma^l \Ch^\a_\nu) =
   (-g)^{k+l}\sum_{i=1}^{\ell(\mu)}\sum_{j=1}^{\ell(\nu)}{v_{(\mu_i|\nu_j)} v_{\mu^-_i} v_{\nu^-_j}}.
\end{equation}
\end{lemma}

\begin{proof}
From~\eqref{eq:IdentityCh}, one has 
   \begin{multline*}
        d^{-\frac{|\mu|+|\nu|+\ell(\mu)+\ell(\nu)+k+l-2}{2}}
		\kappa^{\chi_d, \bullet}_2(\gamma^k\Ch_\mu,\gamma^l \Ch_\nu) \\
		= (-g)^{k+l} d^{\frac{|\mu|+|\nu|-\ell(\mu)-\ell(\nu)+2}{2}}
		\kappa^{\chi_d}_2(\mu,\nu)\cdot \left(1 + O(d^{-1})\right),
	\end{multline*}
so that \eqref{eq:KumuOfCh} is equivalent to the following limit
\[ \lim_{d \to \infty}d^{\frac{|\mu|+|\nu|-\ell(\mu)-\ell(\nu)+2}{2}}{\kappa^{\chi_d}_2(\mu,\nu)} = 
\sum_{i=1}^{\ell(\mu)}\sum_{j=1}^{\ell(\nu)}v_{(\mu_i|\nu_j)} v_{\mu^-_i} v_{\nu^-_j}.\]

The moment-cumulant formula~\eqref{eq:cumu} implies that
\begin{equation}\label{chi_part_1}
\chi_d(\mu\cdot\nu) = \!\!\sum_{\pi \in \SP(\mu\cdot\nu)}\prod_{B \in\pi}\ka_{|B|}^{\chi_d}(B) 
= \!\!\!\sum_{\substack{\pi \in \SP(\mu\cdot\nu)\\ |\pi| \geq \ell(\mu)+\ell(\nu)-1}}
\!\prod_{B\in\pi}{\ka_{|B|}^{\chi_d}(B)} + o\left(d^{-\frac{|\mu|+|\nu|-\ell(\mu)-\ell(\nu)+2}{2}}\right)
\end{equation}
	by the assumption~\eqref{eq:aprox-fact-property}. 
	We slightly abused the notation: here, $\SP(\mu\cdot\nu)$ is the set of set-partitions of 
	$\{\mu_1, \dots, \mu_{\ell(\mu)}, \nu_1, \dots, \nu_{\ell(\nu)}\}$, 
	therefore the quantities $\ka_{|B|}^{\chi_d}(B)$ are cumulants of the form 
	$\ka_n^{\chi_d}(k_1, \cdots, k_n)$, 
	where each $k_i$ is equal to some $\mu_j$ or some $\nu_j$. Similarly,
\begin{multline}\label{chi_part_2}
\chi_d(\mu)\cdot\chi_d(\nu) = \sum_{\substack{\pi' \in \SP(\mu)\\ \pi'' \in \SP(\nu)}}
\prod_{\substack{B'\in\pi \\ \,B''\in\pi'}}\ka_{|B'|}^{\chi_d}(B') \ka_{|B''|}^{\chi_d}(B'')\\
= \!\!\!\sum_{\substack{\pi'\in\SP(\mu),\ \pi''\in \SP(\nu)\\ |\pi'| + |\pi''|\geq \ell(\mu)+\ell(\nu)-1 \ \ }}
\!\!\prod_{\substack{ B'\in\pi' \\ \,B''\in\pi''}}\!\!\ka_{|B'|}^{\chi_d}(B') \ka_{|B''|}^{\chi_d}(B'') +
o\left(d^{-\frac{|\mu|+|\nu|-\ell(\mu)-\ell(\nu)+2}{2}}\right).
\end{multline}
Note that there is a unique set-partition $\pi\in\SP(\mu\cdot\nu)$ 
with $|\pi| = \ell(\mu)+\ell(\nu)$, namely $\pi = \{\{\mu_1\},\cdots, \{\nu_{\ell(\nu)}\}\}$, 
and there are unique set-partitions $\pi' \in \SP(\mu)$ and $\pi''\in\SP(\nu)$ 
with $|\pi'|+|\pi''| = \ell(\mu)+\ell(\nu)$, namely $\pi' = \{ \{\mu_1\},\cdots, \{\mu_{\ell(\mu)}\} \}$
and $\pi'' = \{ \{\nu_1\},\cdots, \{\nu_{\ell(\nu)}\} \}$. 
The corresponding terms in \eqref{chi_part_1} and \eqref{chi_part_2} are equal. 
Pairs of set-partitions $(\pi', \pi'')\in\SP(\mu)\times\SP(\nu)$ with 
$|\pi'|+|\pi''| = \ell(\mu)+\ell(\nu)-1$ are bijectively mapped (by 
$(\pi',\pi'') \mapsto \pi'\cup\pi''$) to set-partitions $\pi\in\SP(\mu\cdot\nu)$ 
with $|\pi| = \ell(\mu)+\ell(\nu)-1$ such that the unique block of $\pi$ of size $2$ 
contains only parts of $\mu$ or only parts of $\nu$. 
Consequently, taking the difference of the RHSs of the above equations gives 
	\begin{align*}
	&\lim_{d \to \infty}d^{\frac{|\mu|+|\nu|-\ell(\mu)-\ell(\nu)+2}{2}}\kappa^{\chi_d}_{2}(\mu,\nu) =\\
	&\lim_{d \to \infty}d^{\frac{|\mu|+|\nu|-\ell(\mu)-\ell(\nu)+2}{2}}
	\sum_{i=1}^{\ell(\mu)}\sum_{j=1}^{\ell(\nu)}\kappa_2^{\chi_d}(\mu_i,\nu_j)
    \prod_{\substack{m=1\\m\neq i}}^{\ell(\mu)}{\!\chi_d(\mu_m)}
	\prod_{\substack{n=1\\n\neq j}}^{\ell(\nu)}{\!\chi_d(\nu_n)}
   	= \sum_{i=1}^{\ell(\mu)}\sum_{j=1}^{\ell(\nu)}v_{(\mu_i|\nu_j)}v_{\mu^-_i} v_{\nu^-_j}
	\end{align*}
by the assumptions~\eqref{eq:refined-asymptotics-characters}
and~\eqref{eq:SecondCumu}. This finishes the proof.
\end{proof}

\begin{proof}[Proof of \cref{theo:CovDS}]

Let $x,y \in \Poly$ and denote $\kappa^\times_\bullet (x,y) := x\times y - x\bullet y$.
First observe that
\begin{multline*}
(\ell-1)\cdot\langle\Delta^\a_d, x^{\ell-2}\rangle\!=\!d^{\frac{1}{2}}\cdot\left(S_\ell(\Lambda^\a_d) - S_\ell(\Lambda_\infty)\right)\\
= d^{-\frac{\ell-1}{2}}\left(S_\ell^\a(\la)\!-\!\lim_{d' \to \infty}\left( \frac{d}{d'}\right)^{\frac{\ell}{2}}\E^\a_{\chi_{d'}}\!\left[S_\ell^\a\right] \right).\end{multline*}
As a result,
\begin{multline}\label{helper_CLT1}
(\ell-1)\cdot\E^\a_{\chi_d}[\langle \Delta^\a_d,\, x^{\ell-2} \rangle] = d^{-\frac{\ell-1}{2}}\left( \E^\a_{\chi_{d}}\!\left[S_\ell^\a\right] - \lim_{d'\to\infty}\left(\frac{d}{d'}\right)^{\frac{\ell}{2}}\E^\a_{\chi_{d'}}\!\left[S_\ell^\a\right] \right),\\
(k-1)(l-1)\cdot\Cov^\a_{\chi_d}\left( \langle\Delta^\a_d, x^{k-2}\rangle, 
\langle\Delta^\a_d, x^{l-2}\rangle \right)
= d^{-\frac{k+l-2}{2}}\kappa^{\chi_d, \times}_2(S_k^\a,S_l^\a)\\
= d^{-\frac{k+l-2}{2}} \bigg(\E^\a_{\chi_d}\kappa^{\times}_\bullet(S_k^\a,S_l^\a) 
+ \kappa^{\chi_d, \bullet}_2(S_k^\a, S_l^\a)\bigg),
\end{multline}
and we want to take the limit of these expressions as $d\to\infty$.
\cref{cor:MSTop} together with \cref{lem:pomocnicz} imply that
\begin{multline*}
\E^\a_{\chi_{d}}\!\left[S_\ell^\a\right] - \lim_{d' \to\infty}\left(\frac{d}{d'}\right)^{\frac{\ell}{2}}\E^\a_{\chi_{d'}}\!\left[S_\ell^\a\right]\\
= d^{\frac{\ell-1}{2}}\cdot\!\left( g'\frac{\partial}{\partial g} + \sum_{i\ge 2}{v_i'\frac{\partial}{\partial v_i}} \right) \!\sum_{\Gamma\in\mathbf{L}_0(\ell)}{\frac{1}{|\SSS^0(\Gamma)|}\prod_{i\ge 1}(i\cdot g)^{|\SSS_{\tiny\rightarrow}^i(\ribbon)|}\,
  v_i^{\,|\SSS_i(\ribbon)|}} + o(d^{\frac{\ell-1}{2}}),
\end{multline*}
and by taking the limit $d \to \infty$ in \eqref{helper_CLT1} we get the desired formula for $\E[\langle \Delta_\infty,\, x^{\ell-2} \rangle]$.

In order to compute
$\lim_{d\to\infty}d^{-\frac{k+l-2}{2}}\kappa_2^{\chi_d, \bullet}{(S_k^\a,S_l^\a)}$
we express $S_k^\a,\, S_l^\a$ as linear combinations of the normalized Jack characters, 
by~\cref{cor:MSTop}, and then apply \eqref{eq:KumuOfCh}. The result is:
\begin{multline}\label{helper_CLT2}
   \lim_{d \to \infty}d^{-\frac{k+l-2}{2}}\kappa^{\chi_d,\bullet}_2(S_k^\a, S_l^\a) = \\
   = \sum_{\ribbon=(\Gamma_1, \Gamma_2)\in\mathbf{L}(k, l)\atop |\mathbf{P}(\ribbon)| = 0}
		\sum_{s_1 \in \SSS_{\tiny\nearrow}(\Gamma_1)\atop s_2\in\SSS_{\tiny\nearrow}(\Gamma_2)} v_{(\deg(s_1)|\deg(s_2))}\cdot\frac{v_{\mu(\ribbon)\setminus\{\deg(s_1),\,\deg(s_2)\}}}{|\mathbf{S}^0(\Gamma_1)|\cdot |\mathbf{S}^0(\Gamma_2)|}
\prod_{i\geq 1}(i\cdot g)^{|\SSS^i_{\tiny\rightarrow}(\ribbon)|}.
\end{multline}

Since item (A) from \cref{theo:AFPCharacterization} holds by assumption then (B) and (C) also hold. In particular, both $\lim_{d\to\infty}{d^{-\frac{k+l-2}{2}}\kappa_2^{\chi_d, \times}(S_k^\a, S_l^\a)},\,\lim_{d\to\infty}{d^{-\frac{k+l-2}{2}}\kappa_2^{\chi_d, \bullet}(S_k^\a, S_l^\a)}$ are finite, and so is their difference $\lim_{d\to\infty}{d^{-\frac{k+l-2}{2}}\cdot\E^\a_{\chi_d}\kappa_\bullet^{\times}(S_k^\a, S_l^\a)}$.
Then \cref{lem:pomocnicz} implies that $\deg_2\kappa^\times_\bullet(S_k^\a, S_l^\a)=k+l-2$,
and there exists a polynomial $P(x_1,\dots,x_{k+l-2})$ such that
\begin{equation}\label{helper_CLT3}
	\lim_{d\to\infty}d^{-\frac{k+l-2}{2}}\cdot\E^\a_{\chi_d}\kappa^\times_\bullet(S_k^\a,S_l^\a)
	= P(g,v_2,\dots,v_{k+l-2}).
\end{equation}
The equations \eqref{helper_CLT1}, \eqref{helper_CLT2}, \eqref{helper_CLT3} imply 
that the LHS of \eqref{eq:CovDS} is equal to the sum of the RHS's of \eqref{helper_CLT2} and \eqref{helper_CLT3}, in particular it is a polynomial in the variables 
$g,\, (v_i)_i,\, (v_{(i|j)})_{i, j}$ that will be completely determined if we find 
$P(g,v_2,\dots,v_{k+l-2})$. In order to find this polynomial, 
we employ \cref{theo:Depoissonization}, which treats the case of the sequence of Jack characters $(\chi_{d;g,\vv})_{d\ge 1}$ in \eqref{eq:JackCDMHighLow}, where $(g,\vv)$ is an admissible pair with $v_1=1$. For these characters, the cumulants of order $n\ge 2$ vanish, so $v_{(i|j)}=0$, for all $i, j$, and therefore the RHS of \eqref{helper_CLT2} vanishes in this case.
As a result, whenever $(g, \vv)$ is an admissible pair with $v_1=1$, it follows that 
\begin{multline}\label{equality_polys_covariance}
P(g, v_2, \dots, v_{k+l-2})=\sum_{\ribbon=(\Gamma_1, \Gamma_2)\in\mathbf{L}^{\conn}(k,l)\atop |\mathbf{P}(\ribbon)| = 1} \frac{v_{\mu(\ribbon)}}{|\mathbf{S}^0(\Gamma_1)|\cdot |\mathbf{S}^0(\Gamma_2)|}\prod_{i\geq 1} (i\cdot g)^{|\SSS^i_{\tiny\rightarrow}(\ribbon)|}
  \,i^{\,|\mathbf{P}_i(\ribbon)| } \\
  - \sum_{\ribbon=(\Gamma_1, \Gamma_2)\in\mathbf{L}(k, l)\atop |\mathbf{P}(\ribbon)| = 0}\frac{|\SSS_{-1}(\Gamma_1)|\cdot |\SSS_{-1}(\Gamma_2)|\cdot v_{\mu(\ribbon)}}{|\mathbf{S}^0(\Gamma_1)|\cdot |\mathbf{S}^0(\Gamma_2)|}\prod_{i\geq 1}(i\cdot g)^{|\SSS^i_{\tiny\rightarrow}(\ribbon)|}.
\end{multline}

Let $g>0$ and $a_1>\dots>a_{k+l-3}>0$ be any real numbers with $\sum_{i=1}^{k+l-3}a_i < 1$. By \cref{thm_jack_spec}, the pair $(g,\vv)$ with $v_1:=1$ and $v_m := \sum_{i=1}^{k+l-3}a_i^m$, $\,m \geq 2$, is an admissible pair.
Then for such parameters $g,\,(v_i)_i$, the equality \eqref{equality_polys_covariance} holds.
Finally, \cref{equal_polys} implies that this equality actually holds always.
This means that the LHS of \eqref{eq:CovDS} is equal to the sum of the RHS's of \eqref{helper_CLT2} and \eqref{equality_polys_covariance}, and this sum can be expressed as the RHS of \eqref{eq:CovDS}, see~\cref{funny_sum}. The proof is now finished.
\end{proof}

\section{Properties of the limit shape in high and low temperature regimes}\label{sec:LimitShape}

\cref{theo:ShapeDS} implies that in the fixed temperature regime (i.e.~when $g=0$) 
the limit shape $\omega_{\Lambda_{g;\vv}}$ does not depend on the choice of the value of $\alpha>0$, in
particular, it is the same as in the classical setting when $\alpha=1$. In this section, we study the limit shape $\omega_{\Lambda_{g;\vv}}$ (equivalently, the associated transition measure $\mu_{g;\vv}$) 
which arises in the high and low temperature regimes, i.e.~when $g>0$ and $g<0$, respectively.

\subsection{Uniqueness and duality between measures $\mu_{g; \vv}$ with
  $g>0$ and $g<0$}

Let $g,g',v_2,v_3,\cdots\in\R$. Recall from \cref{prop:DualityAFP}
that $\mathcal{F}^{\AFP}_{g,g',\vv}$ is
the set of the sequences of Jack characters with the AFP with the parameters $g,g',\vv$ described
in~\eqref{eq:double-scaling-refined} and
\cref{def:approx-factorization-charactersA}.

\begin{proposition}[The measure $\mu_{g;\vv}$ is determined by its moments]\label{lemma_moments}
Suppose that $\mathcal{F}^{\AFP}_{g,g',\vv}\neq \emptyset$.
Then, if we set $v_1:=1$, there exists a unique probability measure $\mu_{g; \vv}$ with moments
\begin{equation}\label{mu_moments'}
\int_\R{x^\ell \mu_{g; \vv}(dx)} = \sum_{\Gamma\in\mathbf{L}_0(\ell)} 
\prod_{i\ge 1} (i\cdot g)^{|\SSS_{\tiny\rightarrow}^i(\Gamma)|}\,v_i^{\,|\SSS_i(\Gamma)|},
\quad\forall\, \ell\in\Z_{\ge 1}.
\end{equation}
\end{proposition}

The probability measure in this proposition is the transition measure of the limit shape $\omega_{\Lambda_{g;\vv}}$ of 
certain ensembles of large random partitions, see~\cref{theo:ShapeDS}. Note that, 
because of \cref{theo:Depoissonization}, if $(g,\vv)$ is an admissible pair (in the sense of 
\cref{def:admissible_pair}) with $v_1=1$, then the conditions of the proposition are satisfied. 

\begin{lemma}\label{lemma:PathsProp}
  Let $\Gamma\in\mathbf{L}_0(\ell)$. Then the following properties
  hold true.
  \begin{enumerate}
  \item If $\SSS^i(\Gamma) \ne\emptyset$, then $i\leq\ell$,
    \item $\sum_{j\ge 1}{j|\SSS_j(\Gamma)|} =
      |\SSS_{\searrow}(\Gamma)|\le \ell,$
      \item $|\mathbf{L}_0(\ell)|\le 4^\ell$, for all large enough
        $\ell$.
    \end{enumerate}
  \end{lemma}
  \begin{proof}
    Since $\mathbf{L}_0(\ell) \subset \mathbf{L}(\ell)$, 
    property (1) is a special case of \cref{claimA}.

    Note that for any $j\in\Z$, the quantity $j|\SSS_j(\Gamma)|$ can be interpreted as
    $\sum_{w_i := (i,y_i) \in \SSS_j(\Gamma)}
    (y_{i}-y_{i-1})$. Since any Łukasiewicz path
    has only down steps of degree $1$ and $y_\ell = y_0 = 0$, one has
    \[ \sum_{j\ge 1}{j|\SSS_j(\Gamma)|} - |\SSS_{\searrow}(\Gamma)|
      = \sum_{i=1}^{\ell}(y_{i} - y_{i-1}) = y_\ell - y_0 = 0,\]
which proves property (2).

Finally, Łukasiewicz paths of length $\ell$ (with the classical
convention that we allow horizontal steps at height $0$) are in
bijection with non-crossing partitions of $[1\,..\,\ell]$, whose number is famously given
by the Catalan number $C_\ell = \frac{1}{\ell+1}{2\ell \choose \ell}\sim
\frac{4^\ell}{\ell^{3/2}\sqrt{\pi}}$, as $\ell\to\infty$. This proves (3).
\end{proof}

\begin{proof}[Proof of \cref{lemma_moments}]      
The existence follows from~\cref{theo:LLNDS} and~\cref{theo:ShapeDS}.
For the uniqueness, we will check Carleman's condition.
In other words, let $M_\ell$ be the $\ell$-th moment
of $\mu_{g; \vv}$ given by the RHS of \eqref{mu_moments'}; 
we verify that the sum $\sum_{k\ge 1}{(M_{2k})^{-1/(2k)}}$ diverges.

Let us find an upper bound for $M_{2k}$.
Take any $\Gamma\in\mathbf{L}_0(2k)$ and upper bound the absolute values of the 
summand in the RHS of \eqref{mu_moments'} that corresponds to $\Gamma$. From
\cref{lemma:PathsProp} (1), 
\begin{equation*}
\prod_{i\ge 1}{|i\cdot g|^{|\SSS_{\tiny\rightarrow}^i(\Gamma)|}} = 
\prod_{i=1}^{2k}{|i\cdot g|^{|\SSS_{\tiny\rightarrow}^i(\Gamma)|}} \le 
\prod_{i=1}^{2k}{(2k|g|)^{|\SSS_{\tiny\rightarrow}^i(\Gamma)|}} = 
(2k|g|)^{|\SSS_{\tiny\rightarrow}(\Gamma)|}.
\end{equation*}

On the other hand, our assumption \eqref{eq:what-is-m} implies that
there exists a constant $C>0$ such that $|v_i|\le (i \cdot C)^i$, for all $i\ge 1$. 
Then from \cref{lemma:PathsProp} (2),
\begin{equation*}
\prod_{i\ge 1}{|v_i|^{\,|\SSS_i(\Gamma)|}}
\le \prod_{i\ge 1}{(i\cdot C)^{i |\SSS_i(\Gamma)|}}  \le (2kC)^{\,\sum_{i\ge 1}{i|\SSS_i(\Gamma)|}}
= (2kC)^{\,|\SSS_{\searrow}(\Gamma)|}.
\end{equation*}

Since $|\SSS_{\tiny\rightarrow}(\Gamma)| + |\SSS_{\searrow}(\Gamma)| \le |\SSS(\Gamma)| = 2k$, 
these inequalities imply that 
\begin{equation*}
0\le\left|\prod_{i\ge 1}{ (i\cdot g)^{|\SSS_{\tiny\rightarrow}^i(\Gamma)|}\,v_i^{\,|\SSS_i(\Gamma)|}}\right|
\le (Dk)^{2k},
\end{equation*}
for some constant $D>0$ that is independent of $k$ and of the Łukasiewicz path $\Gamma\in\mathbf{L}_0(2k)$. Then by \cref{lemma:PathsProp} (3), the $(2k)$-th moment of $\mu_{g; \vv}$ is bounded as follows (for large $k$):
\begin{equation*}
0\le M_{2k}\le (Dk)^{2k}\cdot |\mathbf{L}_0(2k)| \le (Dk)^{2k}\cdot 4^{2k}.
\end{equation*}
The previous inequality shows that for large $k$,
we have $(M_{2k})^{-1/(2k)}\ge \frac{1}{4Dk}$.
Since $\sum_{k=1}^\infty{\frac{1}{k}}$ diverges, so does the sum 
$\sum_{k=1}^\infty{(M_{2k})^{-1/(2k)}}$. 
This proves Carleman's condition and completes the proof.
\end{proof}

\smallskip

Let us complete here one of the steps in the proof of \cref{theo:LLNCDM} that was left out in \cref{subsubsec:Cumulants}. There we proved the existence of a probability measure $P_\infty$ with moments 
\begin{equation}\label{moments_P_infty}
\int_\R{x^\ell P_\infty(dx)} = \frac{v_1^{-\frac{\ell+2}{2}}}{\ell+1}\sum_{\Gamma\in\mathbf{L}_0(\ell+2)} {\frac{1}{|\SSS^0(\Gamma)|}\prod_{i\geq 1} (i\cdot g)^{|\SSS^i_{\tiny\rightarrow}(\Gamma)|} \, v_i^{\,|\SSS_i(\Gamma)|}},
\quad\forall\, \ell\in\Z_{\ge 0}.
\end{equation}

As a reminder, $P_\infty$ was relevant to us because we proved that it has density $\frac{\omega_{\Lambda_{g;\vv}}(x)-|x|}{2}$, where $\omega_{\Lambda_{g;\vv}}$ is the limit shape of random Young diagrams obtained in \cref{theo:LLNCDM}.

\begin{corollary}\label{cor:UniqueP}
Let $g,v_1,v_2,\cdots\in\R$ be parameters for which there exist sequences of positive reals 
$\al(d),\, u(d)>0$ such that the weak version of \cref{main_assumption} is satisfied.
We showed in~\cref{sec:proofs_CDM} that there exists a probability measure $P_\infty$ with moments given by \eqref{moments_P_infty}.
Then the probability measure $P_\infty$ is uniquely determined by its moments.
\end{corollary}

\begin{proof}
  \cref{thm:KOO} implies that for each $d$, there
  exist sequences $a_1(d)\geq a_2(d) \geq \cdots \geq 0$, $b_1(d)\geq
  b_2(d) \geq \cdots \geq 0$ such that $u(d)\cdot v_1\ge \sum_{i=1}^\infty (a_i(d) +
  b_i(d))$ and $u(d)\cdot v_k = \sum_{i=1}^\infty a_i(d)^k +
(-\alpha(d))^{1-k}\sum_{i=1}^\infty b_i(d)^k$, for all $k \geq 2$. 
One can easily check that this implies that the sequence $(v_k)_{k \geq 1}$ satisfies the assumption \eqref{eq:what-is-m} from \cref{def:approx-factorization-charactersA}. This allows us to use the same estimates from the proof of \cref{lemma_moments} to prove Carleman's condition.

  For all $\ell\in\Z_{\ge 1}$, let $|m_\ell| := \sum_{\Gamma\in\mathbf{L}_0(\ell)} \left|\prod_{i\geq 1} (i\cdot g)^{|\SSS^i_{\tiny\rightarrow}(\Gamma)|} \, v_i^{\,|\SSS_i(\Gamma)|}\right|$, and let $M_\ell^\infty$ be the $\ell$-th moment of $P_\infty$. Since $|\SSS^0(\Gamma)|\ge 1$, then $|M_\ell^\infty| \le |v_1|^{-\frac{\ell+2}{2}}\cdot |m_{\ell+2}|$, and in particular,
\[ (M_{2k}^\infty)^{-1/(2k)} \ge C_1\cdot \left(|m_{2k+2}|^{-1/(2k+2)}\right)^{(k+1)/k}, \]
for some constant $C_1>0$ and all $k\in\Z_{\ge 1}$.
In the proof of \cref{lemma_moments}, it was shown that $(m_{2k+2})^{-1/(2k+2)}\ge\frac{C_2}{k}$, for some constant $C_2>0$ and all $k\in\Z_{\ge 1}$.
Putting these two inequalities together gives the existence of some constant $C>0$ such that $(M_{2k}^\infty)^{-1/(2k)} \ge \frac{C}{k}$, for all $k\in\Z_{\ge 1}$. Since $\sum_{k=1}^\infty{\frac{1}{k}}$ diverges, so does $\sum_{k=1}^\infty{(M_{2k}^\infty)^{-1/(2k)}}$. This proves Carleman's condition and finishes the proof.
\end{proof}

For the next proposition, recall that we use the notation $\pm\vv:=(v_1, -v_2, v_3, -v_4, \dots)$. 

\begin{proposition}[Duality between $\mu_{g;\vv}$ and $\mu_{-g;\,\pm\vv}$]\label{prop_duality}
 Let $\left(\chi_d^\a\right)_{d \geq 1} \in \mathcal{F}^{\AFP}_{g,g',\vv}$, and
 let $\left(\tilde{\chi}_d^{(\tilde{\alpha})}\right)_{d \geq 1} \in
 \mathcal{F}^{\AFP}_{-g,-g',\pm\vv}$ be the associated sequences from
 \cref{prop:DualityAFP}. Then
 \[\omega_{\Lambda^{(\alpha)}_d}(-x) =
   \omega_{\Lambda^{(\tilde{\alpha})}_d}(x)\]
 as random functions. In particular
 \[ \omega_{\Lambda_{g;\vv}}(x) = \omega_{\Lambda_{-g;\pm\vv}}(-x),
   \text{ and }
 \mu_{g;\vv} = \mu_{-g;\,\pm\vv}.\]
\end{proposition}
\begin{proof}
For any partition $\la$ of $d$, we have
\begin{equation}\label{eq:duality_proof1}
\omega_{T_{\sqrt{\frac{\al}{d}},\sqrt{\frac{1}{\al d}}}\la}(-x) =
   \omega_{T_{\sqrt{\frac{1}{\al d}},\sqrt{\frac{\al}{d}}}\la'}(x) =
   \omega_{T_{\sqrt{\frac{\tilde{\al}}{d}},\sqrt{\frac{1}{\tilde{\al}d}}}\la'} (x).
\end{equation}
\cref{prop:DualityAFP} states that $\mathbb{P}^{(\tilde{\alpha})}_{\tilde{\chi}_d}(\lambda) = \mathbb{P}^{\a}_{\chi_d}(\lambda')$, which together with~\eqref{eq:duality_proof1} implies
that $\omega_{\Lambda^{(\alpha)}_d}(-x)$ and $\omega_{\Lambda^{(\tilde{\alpha})}_d}(x)$ have the same distribution.
Passing to the limit $d \to \infty$,
 \cref{theo:ShapeDS} implies that $\omega_{\Lambda_{g;\vv}}(x) =
 \omega_{\Lambda_{-g;\pm\vv}}(-x)$.
Finally, $\mu_{g;\vv} = \mu_{-g;\,\pm\vv}$ follows from the Markov--Krein correspondence.
\end{proof}

\subsection{The limit shape $\omega_{\Lambda_{g;\vv}}$ is an infinite staircase}

For the remainder of this section, assume that $(g,\vv)$ is an
  admissible pair with $v_1=1$, and $\mu_{g;\vv}$ is the probability
  measure determined by \cref{lemma_moments}.

\hspace{1pt}

Note that if $(g,\vv)$ is an admissible pair, then so is
$(-g,\pm\vv)$. In particular, this means that both probability measures
$\mu_{g; \vv}$ and $\mu_{-g;\,\pm\vv}$ exist and are uniquely determined
by their moments.
\cref{prop:DualityAFP} and \cref{prop_duality} imply that the cases $g>0$ and $g<0$ are dual to
each other. Thus, in the proofs below we will stick to only one of these two cases, either $g>0$ or $g<0$, if the other one follows by duality.
We would like to completely describe the limit shape
$\omega_{\Lambda_{g;\vv}}$ for $g\ne 0$.
Qualitatively, this shape is described by the following result.

\begin{proposition}\label{prop:LongestRows}
Let $(\chi_d)_{d\ge 1}$ be the Jack characters described in \cref{theo:Depoissonization}, let $\PP^\a_{\chi_d}$ be the corresponding probability measures, and let $\la^{(d)}\in\Y_d\,$ be $\PP^\a_{\chi_d}$-distributed random partitions.

If $g>0$, then the following limits exist in probability
\begin{equation}\label{eq:limit_cols}
\tilde{\la}'_i := g^{-1}\cdot\lim\frac{(\la^{(d)})'_i}{d},
\end{equation}
for all $i=1,2,\dots$, and
\begin{equation}\label{eq:LimProfileShape}
\omega_{\Lambda_{g,\vv}} (x) = 
\begin{cases}
-x, &\text{ if } x < -\tilde{\la}'_1,\\
2\tilde{\lambda}'_i+x, &\text{ if } x\in\big[-\tilde{\la}'_i+(i-1)g,\ -\tilde{\la}'_i+ig\big],\text{ for some } i\ge 1,\\
-x+2ig, &\text{ if } x\in\big[-\tilde{\lambda}'_i+ig,\ -\tilde{\lambda}'_{i+1}+ig\big],\text{ for some } i\ge 1.
\end{cases}
\end{equation}

Likewise, if $g<0$, then the following limits exist in probability
\begin{equation}\label{eq:limit_rows}
\tilde{\la}_i := -g^{-1}\cdot\lim\frac{\la^{(d)}_i}{d},
\end{equation}
for all $i=1,2,\dots$, and
\begin{equation}\label{eq:LimProfileShape'}
\omega_{\Lambda_{g,\vv}} (x) = 
\begin{cases}
x, &\text{ if } x >\tilde{\la}_1,\\
2\tilde{\lambda}_i-x, &\text{ if } x\in\big[\tilde{\la}_i+ig,\ \tilde{\la}_i+(i-1)g\big],\textrm{ for some } i \geq 1, \\
x-2ig, &\text{ if } x\in\big[\tilde{\la}_{i+1}+ig,\ \tilde{\la}_i+ig\big], \textrm{ for some } i\ge 1.
\end{cases}
\end{equation}
\end{proposition}
\begin{proof}
Assume that $g >0$. Recall that, in the notation~\eqref{eq:Lambda}, we have the rescaled partition $\Lambda^\a_d = T_{\sqrt{\frac{\al}{d}},\, \sqrt{\frac{1}{\al d}}}\,\la^{(d)}$.
Since the parameters $\al,\, d,\, g$ are related by~\eqref{alpha_d}, we have $\Lambda^\a_d = T_{g,\, \frac{1}{gd}}\left(\la^{(d)}\right)$.
This means that $\Lambda^\a_d$ is a random diagram, which is obtained (in the French notation) by stacking rectangles of sizes $g\times\frac{(\la^{(d)})'_1}{gd}, g\times\frac{(\la^{(d)})'_2}{gd},\cdots$ next to each other from
left to right.
Equivalently, the profile of $\Lambda^\a_d$ is equal to
\begin{equation*}
\omega_{\Lambda^\a_d}(x) = 
\begin{cases}
-x, &\text{ if } x < -\frac{(\la^{(d)})'_1}{gd},\\
2\frac{(\la^{(d)})'_i}{gd}+x, &\text{ if } x\in\big[-\frac{(\la^{(d)})'_i}{gd}+(i-1)g,\ 
-\frac{(\la^{(d)})'_i}{gd}+ig\big],\text{ for some }i\ge 1,\\
-x+2ig, &\text{ if } x\in\big[-\frac{(\la^{(d)})'_i}{gd}+ig,\ -\frac{(\la^{(d)})'_{i+1}}{gd}+ig\big],\text{ for some }i\ge 1.
\end{cases}
\end{equation*}
Due to the existence of a limit shape $\omega_{\Lambda_\infty}$ of the
profiles $\omega_{\Lambda^\a_d}$ in the supremum norm in probability
(\cref{theo:LLNDS}) and the fact that $g>0$ is deterministic, it
necessarily follows that the limits in probability
\[
\tilde{\la}'_i := g^{-1}\cdot\lim\frac{(\la^{(d)})'_i}{d}
\]
exist, for all $i=1,2,\cdots$.
Moreover, the limit $\Lambda_\infty$ of the rescaled Young diagrams is obtained (in the French notation) by putting rectangles of sizes $g\times \tilde{\lambda}'_1, g\times \tilde{\lambda}'_2, \dots$ next to each other from left to right.
Equivalently, the limit profile $\omega_{\Lambda_\infty}$ is given by the RHS of~\eqref{eq:LimProfileShape}\footnote{See~\cref{fig:LowTempLim} for an illustration of the rescaled Young diagram $\Lambda^\a_d$ and the limit shape $\Lambda_\infty$ when $g<0$.}.
By \cref{theo:Depoissonization}, the limit profile $\omega_{\Lambda_\infty}$ is equal to $\omega_{\Lambda_{g,\vv}}$, therefore \cref{eq:LimProfileShape} follows.
\end{proof}

As a result of \cref{prop:LongestRows}, the limit profile $\omega_{\Lambda_{g,\vv}}$ is a piecewise linear curve that only has
slopes $\pm 1$ (it is a \emph{staircase shape}). Moreover, assuming
that $\tilde{\la}_1'>\tilde{\la}_2'>\dotsm$ when $g>0$ (resp. $\tilde{\la}_1>\tilde{\la}_2>\dotsm$ when $g<0$) are all distinct, then $\omega_{\Lambda_{g,\vv}}$ 
is entirely determined by the locations of its local minima.
The interplay between the local minima of $\omega_{\Lambda_{g,\vv}}$ and the support of the transition measure $\mu_{g;\vv}$ in the next result will play a critical role.

\begin{proposition}\label{measures_discrete}
The support of $\mu_{g;\vv}$ is discrete, and it coincides with the set of local minima of $\omega_{\Lambda_{g;\vv}} (x)$.
Moreover, if $g>0$, then $\supp(\mu_{g; \vv}) \subseteq [-g^{-1},\, +\infty)$, whereas if $g<0$, then $\supp(\mu_{g; \vv}) \subseteq (-\infty,\, -g^{-1}]$.
Finally, the distance between two consecutive points of $\supp\mu_{g;\vv}$ is at least $|g|$.
\end{proposition}
\begin{proof}
Assume that $g>0$.
By \cref{prop:LongestRows}, the limit shape $\omega_{\Lambda_{g;\vv}}$ only has slopes $\pm 1$, thus it is uniquely determined by its local minima $x_i$ and local maxima $y_i$ that form an interlacing sequence of the form $x_1 < y_1 < x_2 < y_2 < \cdots$ (which could be finite or countably infinite).
As a result, \cref{eq:TransitionSupport} implies that $\supp(\mu_{g; \vv}) = \{x_1,x_2,\dots\}$ is discrete.

Moreover, by~\eqref{eq:LimProfileShape}, the smallest local minimum is $x_1=-\tilde{\la}'_1$.
Since $(\la^{(d)})'_1\le \big|\la_1^{(d)}\big|=d$, it follows that
\[
x_1=-\tilde{\la}'_1 = -g^{-1}\cdot\lim\frac{(\la^{(d)})'_i}{d} \ge -g^{-1}\cdot 1 = -g^{-1}.
\]
Finally, by~\eqref{eq:LimProfileShape}, two consecutive local minima must be of the form $-\tilde{\la}'_i + (i-1)g$ and $-\tilde{\la}'_j + (j-1)g$, for some $j>i$.
Since $(\la^{(d)})'_i\ge(\la^{(d)})'_j$, whenever $j>i$, the distance between these consecutive local minima is
\[
(-\tilde{\la}'_j + (j-1)g) - (-\tilde{\la}'_i + (i-1)g)
= (j-i)g + g^{-1}\cdot\lim\frac{(\la^{(d)})'_i - (\la^{(d)})'_j}{d}\ge g,
\]
which ends the proof.
\end{proof}

It is not clear from \cref{prop:LongestRows} whether the staircase shape $\omega_{\Lambda_{g,\vv}}$ has infinitely many local minima or only finitely many.
In fact, $\omega_{\Lambda_{g,\vv}}$ will have only finitely many local minima if some $\tilde{\la}_i'$ (resp.~$\tilde{\la}_i$), as defined by~\eqref{eq:limit_cols} when $g>0$ (resp.~by~\eqref{eq:limit_rows} when $g<0$), is equal to zero;
the next theorem shows that this never happens.

\begin{theorem}\label{main_thm_2}
The support of $\mu_{g;\vv}$ is countably infinite. Moreover, if $\pm g>0$, then the unique accumulation point of the support of $\mu_{g; \vv}$ is $\pm\infty$.
Consequently, the limit shape $\omega_{\Lambda_{g,\vv}}$ is a semi-infinite staircase shape with only slopes $\pm 1$ that extends to $+\infty$, if $g>0$, and extends to $-\infty$, if $g<0$.
\end{theorem}
\begin{proof}
The statement about $\omega_{\Lambda_{g,\vv}}$ follows from ones about $\mu_{g;\vv}$, together with \cref{prop:LongestRows,measures_discrete}.
Thus it remains to prove the statements about $\mu_{g;\vv}$.
Assume that $g >0$. By \cref{measures_discrete}, the measure $\mu_{g;\vv}$ is discrete and has 
at most one accumulation point, namely $+\infty$.
If we proved that $\mu_{g; \vv}$ does not have compact support, the theorem would follow.
Thus, for the sake of contradiction, assume that the support of $\mu_{g; \vv}$ is compact,
and therefore finite. In particular, there exist $C, \epsilon > 0$ such
that the moments $M_{\ell} = \int_{\R}{x^{\ell}\mu_{g; \vv}(dx)}$ of
$\mu_{g; \vv}$ are bounded as follows:
\begin{equation}\label{bound_moments}
|M_{\ell}| \le C\eps^{-\ell},\ \text{ for all }\ell\in\Z_{\ge 1}.
\end{equation}
On the other hand, recall that the moments $M_\ell$ are given by
\eqref{mu_moments'}. For any $k\in\Z_{\ge 1}$, consider the following Łukasiewicz path $\Gamma_0\in\mathbf{L}_0(3k)$:
\begin{itemize}
	\item $w_0 = (0, 0),\, w_1=(1, 1),\cdots, w_k=(k, k)$ ($k$ up steps of degree $1$).
	\item $w_{k+1} = (k+1, k),\, w_{k+2}=(k+2, k),\cdots, w_{2k}=(2k, k)$ ($k$ horizontal steps).
	\item $w_{2k+1} = (2k+1, k-1),\, w_{2k+2}=(2k+2, k-2),\cdots, w_{3k}=(3k, 0)$ ($k$ down steps).
\end{itemize}
The term corresponding to $\Gamma_0$ in the sum \eqref{mu_moments'} is equal to 
$(kg)^k\cdot v_1^k = (kg)^k$.
Next, since $\vv$ determines a totally Jack-positive specialization, \cref{thm_jack_spec} implies that $v_i\ge 0$, for all $i$, therefore all terms in the sum \eqref{mu_moments'} are nonnegative and so
$M_{3k}\ge (kg)^k$. But according to \eqref{bound_moments}, $M_{3k}\le C \eps^{-3k}$.
As a result, $C \ge (kg\cdot\eps^3)^k$, for all $k\in\Z_{\ge 1}$,
which is impossible. Hence, $\mu_{g; \vv}$ does not have compact support.
\end{proof}

We can now derive precise information about the \emph{edge asymptotics} (longest rows or columns) of random partitions, starting from the knowledge we have about $\supp(\mu_{g;\vv})$.
Denote
\[
\max\supp(\mu) := \sup\{x \mid x\in\supp(\mu) \},\quad\min\supp(\mu) := \inf\{x \mid x\in\supp(\mu) \},
\]
for any measure $\mu$ on $\R$.

\begin{proposition}\label{prop:LongestRow}
Let $\left(\chi_d\right)_{d\ge 1}\in\mathcal{F}^{\AFP}_{g,g',\vv}$, let $\PP^\a_{\chi_d}$ be the corresponding probability measures, and let $\la^{(d)}\in\Y_d\,$ be $\mathbb{P}^\a_{\chi_d}$-distributed random partitions.
Further, set $v_1:=1$, and let $\mu_{g;\vv}$ be
the transition measure of the corresponding limit shape
$\omega_{\Lambda_\infty}$ (whose existence is given in \cref{theo:LLNDS}).
Equivalently, since $\omega_{\Lambda_\infty}=\omega_{\Lambda_{g;\vv}}$ (due to \cref{theo:ShapeDS}), $\mu_{g;\vv}$ is 
the probability measure on $\R$ determined by \cref{lemma_moments}.

If $g>0$, then we have
\[
\frac{(\la^{(d)})'_1}{-g\cdot d} \to \min\supp(\mu_{g;\vv}),
\]
in probability.
If $g<0$, then we have
\[
\frac{\la^{(d)}_1}{-g\cdot d} \to \max\supp(\mu_{g;\vv}),
\]
in probability.
\end{proposition}
\begin{proof}
Assume that $g<0$ and set $M:=\max\supp(\mu_{g;\vv})$.
According to \cref{measures_discrete}, $M$ is the largest local minimum of $\omega_{\Lambda_{g;\vv}}$; in the notation of \cref{prop:LongestRows} and \cref{eq:LimProfileShape'}, this means $M = \tilde{\la}_1$, and therefore $\omega_{\Lambda_{g;\vv}}(x) = x$, for $x>M$, and $\omega_{\Lambda_{g;\vv}}(x) = 2M-x$, for $x\in[M+g,M]$.
A simple geometric argument thus implies that, if $\la^{(d)}_1\cdot\frac{\sqrt{\al}}{\sqrt{d}}\ge M$, then
\begin{equation}\label{eq:tech_1}
\left|\omega_{\Lambda^\a_d}(M)-\omega_{\Lambda_\infty}(M)\right|
\ge 2\cdot\min\left\{ \frac{1}{\sqrt{\al d}},\ \left|\la^{(d)}_1\cdot\frac{\sqrt{\al}}{\sqrt{d}} - M\right| \right\}.
\end{equation}
Similarly, if $\la^{(d)}_1\cdot\frac{\sqrt{\al}}{\sqrt{d}}\le M$, then
\begin{equation}\label{eq:tech_2}
\left|\omega_{\Lambda^\a_d}\left(\la^{(d)}_1\cdot\frac{\sqrt{\al}}{\sqrt{d}}\right)-\omega_{\Lambda_\infty}\left(\la^{(d)}_1\cdot\frac{\sqrt{\al}}{\sqrt{d}}\right)\right|\ge 2\cdot\min\left\{ -g,\  \left|\la^{(d)}_1\cdot\frac{\sqrt{\al}}{\sqrt{d}} - M\right| \right\}.
\end{equation}
Recall that in our regime of interest we have $\frac{1}{\sqrt{\al d}}\sim -g$.
Take any small $-g>\eps>0$, and assume that $\eps\ge\|\omega_{\Lambda^\a_d} - \omega_{\Lambda_\infty}\|_\infty$ and $\eps/2 > |\sqrt{\al d} + g^{-1}|$.
Then the previous inequalities~\eqref{eq:tech_1}-\eqref{eq:tech_2} imply
\[
\frac{\eps}{2}\ge\left|\la^{(d)}_1\cdot\frac{\sqrt{\al}}{\sqrt{d}} - M\right|.
\]
By our choice of $\eps>0$, it follows that
\[
\eps\ge\left|\frac{\la^{(d)}_1}{-g\cdot d} - M\right|.
\]
Hence,
\[
\PP^\a_{\chi_d}( \|\omega_{\Lambda^\a_d} - \omega_{\Lambda_\infty}\|_\infty \le \eps)
\le \PP^\a_{\chi_d}\left(\left|\frac{\la^{(d)}_1}{-g\cdot d} - M\right| \le \eps\right),\]
for all large $d$.
\cref{theo:LLNDS} shows $\lim_{d\to\infty}\PP^\a_{\chi_d}( \|\omega_{\Lambda^\a_d} - \omega_{\Lambda_\infty}\|_\infty \le \eps)=1$, therefore
\[
\lim_{d\to\infty}{\PP^\a_{\chi_d}\left( \left|\frac{\la^{(d)}_1}{-g\cdot d} - M\right| \le \eps\right)} = 1,
\]
which is the desired limit in probability.
\end{proof}

\begin{remark}
For concreteness, let us consider only the case $g>0$ in this remark.
We cannot tie the quantity $\frac{(\la^{(d)})'_i}{-g\cdot d}$ to the $i$-th smallest point in the support of $\mu_{g;\vv}$ (equivalently, the $i$-th local minima of $\omega_{\Lambda_{g,\vv}}$) , unless $i=1$, because it is not clear from~\cref{eq:LimProfileShape} what is the $i$-th smallest local minimum of $\omega_{\Lambda_{g,\vv}}$.
The problem occurs when, in the notation of \cref{prop:LongestRows}, $\tilde{\la}'_j=\tilde{\la}'_{j+1}$, for some $j$'s.
This is why the general \cref{prop:LongestRow} only deals with the longest columns $(\la^{(d)})'_1$ of the random partitions.
In the next section, we will study one case where this analysis can be extended to several of the longest columns $(\la^{(d)})'_1, (\la^{(d)})'_2, \cdots$ (or several of the longest rows, if $g<0$).
\end{remark}

In the next section, we consider the special case when $\PP^\a_{\chi_d}$ are the Jack--Plancherel measures and we find not only the limits in probability of the longest row/column of the random partition in terms of $\supp(\mu_{g;\vv})$, but we will compute exactly the limits of any finite number of longest rows/columns, and therefore we will entirely determine the limit shape $\omega_{\Lambda_\infty}$ in this case.

\subsection{Edge asymptotics of Jack--Plancherel-distributed random partitions}

Consider the $\Z_{\ge 0}\times\Z_{\ge 0}$ matrix $J_{g;\vv}$ defined by $(J_{g;\vv})_{i, j} := ig\cdot\delta_{i,j} + v_{j-i}$, for all $i, j\in\Z_{\ge 0}$, where $v_{-1}:=1$ and $v_0=v_{-l}:=0$, for all $l\ge 2$. Also consider the special case when $v_1=1,\, v_i=0$, for all $i\ge 2$, to be denoted by $J^{\text{Planc}}_g := J_{g;(1,0,0,\cdots)}$, i.e.
\[  J_{g;\vv} := \begin{bmatrix} 
0 & v_1 & v_2 & v_3 & \cdots\\
1 & g & v_1 & v_2 & \ddots\\
0 & 1 & 2g & v_1 & \ddots\\
0 & 0 & 1 & 3g & \ddots\\
\vdots & \ddots & \ddots & \ddots & \ddots
\end{bmatrix}, \qquad
J^{\text{Planc}}_g := \begin{bmatrix} 
0 & 1 & 0 & 0 & \cdots\\
1 & g & 1 & 0 & \ddots\\
0 & 1 & 2g & 1 & \ddots\\
0 & 0 & 1 & 3g & \ddots\\
\vdots & \ddots & \ddots & \ddots & \ddots
\end{bmatrix}. \]

\begin{proposition}\label{prop:matrix_moms}
The $(0, 0)$-entries of the powers of $J_{g;\vv}$ are equal to the moments of the measure $\mu_{g;\vv}$:
\begin{equation}\label{matrix_moms}
(J_{g;\vv})^\ell_{0,0} = \int_{\R}{x^\ell d\mu_{g;\vv}(x)},\ \text{ for all } \ell\in\Z_{\ge 0}.
\end{equation}
\end{proposition}
\begin{proof}
Denote $J := J_{g;\vv}$ and take any $\ell\in\Z_{\ge 0}$. Then
\[(J^\ell)_{0,0} =\sum_{i_1, \dots, i_{\ell-1}\in\Z_{\ge 0}}J_{0, i_1}J_{i_1, i_2}\cdots J_{i_{\ell-1}, 0}.\]
The sum is over sequences $(i_1, \dots, i_{\ell-1})\in(\Z_{\ge 0})^{\ell-1}$ with $i_{k+1}-i_k\ge -1$, for all $k$, since otherwise the term vanishes.
To any such sequence $(i_1, \dots, i_{\ell-1})$, associate the excursion $\Gamma = (w_0, w_1, \cdots, w_\ell)$, defined by $w_k=(k, y_k),\, y_0=y_\ell=0$, and $y_k=i_k$, for $1\le k\le \ell-1$. The condition $i_{k+1}-i_k\ge -1$ implies that $\Gamma$ is a Łukasiewicz path of length $\ell$. Moreover, the term $J_{0, i_1}J_{i_1, i_2}\cdots J_{i_{\ell-1}, 0}$ contains a factor $v_i$ for every up step of degree $i$, and a factor $i\cdot g$ for every horizontal step at height $i$, so the term equals $\prod_{i\ge 1}{(i\cdot g)^{|\SSS_{\tiny\rightarrow}^i(\Gamma)|}\,v_i^{\,|\SSS_i(\Gamma)|}}$. Hence, $(J^\ell)_{0, 0}$ is the sum in the RHS of~\eqref{mu_moments'}, which also equals the $\ell$-th moment of $\mu_{g;\vv}$.
\end{proof}

In the special case $\vv=(1,0,0,\cdots)$, the resulting $J^{\text{Planc}}_g$ is a \emph{Jacobi matrix}, i.e.~a real symmetric, tridiagonal matrix whose off-diagonal entries are positive. The (unique) probability measure $\mu_g^{\text{Planch}} := \mu_{g;(1,0,0,\cdots)}$ satisfying \eqref{matrix_moms} is said to be the \emph{spectral measure of $J^{\text{Planc}}_g$}.
Note also that the moment formula \eqref{mu_moments'} simplifies because the terms coming from Łukasiewicz paths with up steps of degree $\ge 2$ vanish. Let $\mathbf{M}_0(\ell)\subseteq\mathbf{L}_0(\ell)$ be the subset of Łukasiewicz paths whose up steps (and down steps) all have degree $1$: this is the set of \emph{Motzkin paths of length $\ell$}, see~\cref{fig:motzkin}.\footnote{Our Motzkin paths do not have horizontal steps at height $0$, though this is usually allowed in the literature.}

\begin{figure}
\begin{center}
\includegraphics[width=0.7\linewidth]{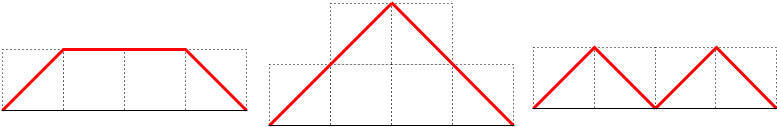}
\end{center}
\caption{All three Motzkin paths in $\mathbf{M}_0(4)$.}
\label{fig:motzkin}
\end{figure}

If $g=0$, then $J^{\text{Planc}}_{g=0}$ is a Toeplitz matrix with spectral measure $\mu_{g=0}^{\text{Planch}}$ being the semicircle distribution. When $g\ne 0$, our Propositions \ref{lemma_moments}, \ref{prop:matrix_moms} and Theorem \ref{main_thm_2} imply the following.

\begin{corollary}
  \label{cor:PlancherelMoments}
If $\pm g>0$, the spectral measure $\mu_g^{\text{Planch}}$ of $J^{\text{Planc}}_g$ has discrete countable support $\supp(\mu^{\text{Planc}}_g)$ with unique accumulation point $\pm\infty$. Moreover, $\mu_g^{\text{Planch}}$ is the unique probability measure with moments:
\[  \int_\R{x^\ell \mu^{\text{Planc}}_g(dx)} = \sum_{\Gamma\in\mathbf{M}_0(\ell)} 
\prod_{i\ge 1} (i\cdot g)^{|\SSS_{\tiny\rightarrow}^i(\Gamma)|},\ \text{ for all }\ell\in\Z_{\ge 1}.  \]
\end{corollary}

Let
\[ J_{\nu}(x) =  \sum_{m=0}^\infty\frac{(-1)^m}{m!\Gamma(m+\nu+1)}\left(\frac{x}{2}\right)^{2m+\nu}\]
be the Bessel function of the first kind, where $\Gamma(z)$ is the
Gamma function. It turned out that the operator $J^{\text{Planc}}_g$ was studied in the past, and the following relation between the zeroes of $J_{\nu}(x)$, as a function of the label parameter $\nu$, and the spectrum of $\mu_{g}^{\text{Planch}}$ is well-known (see~\cite{GradZakrajsek1973,IkebeKikuchiFujishiro1991,StampachStovicek2015}).

\begin{theorem}\label{theo:BesselZeroes}
Let $g>0$, and let
\[
\big\{l_1^{(g)}<l_2^{(g)}<\cdots \big\}  = \{z \in \R\colon J_{-z\cdot g^{-1}}(2 g^{-1}) = 0\}
\]
be the set of zeroes of the Bessel function of the first kind. Then the support of the spectral measure $\mu_g^{\text{Planch}}$ of $J^{\text{Planc}}_g$ is
\[
\supp(\mu_g^{\text{Planch}}) = \big\{l_1^{(g)}-g < l_2^{(g)}-g < \cdots\big\}.
\]
\end{theorem}

Now we prove the main theorem of this section, which describes the limits in probability of the longest rows or columns of random Jack--Plancherel-distributed random partitions in the low or high temperature regimes.
In particular it gives an explicit description of the limit shape $\omega_{\Lambda_\infty}$ in these regimes in terms of zeroes of Bessel functions.

\begin{theorem}\label{thm_asymptotic_positions}
Let  $(\la^{(d)}\in\Y_d)_{d\ge 1}$ be a sequence of random Young diagrams distributed by the Jack--Plancherel measures, as defined by \cref{eq:Jack-Planch} or in \cref{ex:Plancherel}.

\vspace{1pt}

$\bullet$ Assume that $g>0$. In the regime $\al,d\to\infty,\,\sqrt{\al}/\sqrt{d}\to g$, for any $i=1,2,\cdots$, we have
$\,\frac{(\la^{(d)})'_i}{-g\cdot d}\to l_{i}^{(g)}-ig$, in probability, and the limit shape is
\[
\omega_{\Lambda_{\infty}} (x) = 
\begin{cases}
-x, &\text{ if } x < -l_1^{(g)}-g,\\
2(l_i^{(g)}-ig)+x, &\text{ if } x\in\big[l_i^{(g)}-g,\ l_i^{(g)}\big], \text{ for some } i\ge 1,\\
-x+2ig, &\text{ if } x\in\big[l_i^{(g)},\ l_{i+1}^{(g)}-g\big], \text{ for some } i\ge 1.
\end{cases}
\]

\vspace{1pt}

$\bullet$ Assume that $g<0$. In the regime $\al\to 0,\, d\to\infty,\,\sqrt{\al d}\to -g^{-1}$, for any $i=1,2,\cdots$, we have
$\,\frac{\la^{(d)}_i}{-g\cdot d}\to -l_i^{(-g)}-ig$, in probability, and
\[
\omega_{\Lambda_\infty} (x) = 
\begin{cases}
x, &\text{ if } x > -l_1^{(-g)}-g,\\
-2(l_i^{(-g)}+ig)-x, &\text{ if } x\in\big[-l_i^{(-g)},\ -l_i^{(-g)}-g\big], \text{ for some } i \ge 1,\\
x-2ig, &\text{ if } x\in\big[-l_{i+1}^{(-g)}-g,\ -l_i^{(-g)}\big], \text{ for some } i\ge 1.
\end{cases}
\]
\end{theorem}
\begin{proof}
Consider only $g<0$, since the case $g>0$ is analogous.

We first claim that whenever $\omega_{\Lambda_{\infty}}$ has a local minimum at $z_0$, then it has a local maximum at $z_0+g$.
By \cref{theo:Depoissonization}, we have $\omega_{\Lambda_{\infty}}=\omega_{g;(1,0,0,\cdots)}$ and Equation~\eqref{eq:LimProfileShape'} describes this limit shape: qualitatively, it is a staircase shape with slopes $\pm 1$.
As a result of \cref{main_thm_2}, this is an infinite staircase, so in view of \cref{eq:LimProfileShape'}, the limits $\tilde{\la}_i$ in \cref{eq:limit_rows} are never zero.
Note that our claim might appear obvious from a look at \cref{eq:LimProfileShape'}, however this is not the case when $\tilde{\la}_i=\tilde{\la}_{i+1}$, for some $i\ge 1$ (if $\tilde{\la}_i+(i-1)g$ is the position of a local minimum and $\tilde{\la}_i=\tilde{\la}_{i+1}$, then $\tilde{\la}_i+ig$ is not a local maximum).

The infinite staircase $\omega_{\Lambda_{\infty}}$ of slopes $\pm 1$ fits into the framework of our discussion in \cref{sec:transition_measure}, which implies that the interlacing sequences of local minima and local maxima of $\omega_{\Lambda_{\infty}}$ correspond to the poles and zeroes, respectively, of the associated (via Markov--Krein correspondence) Cauchy transform $G_{\mu_g^{\text{Planch}}}(z)$.
Therefore our desired claim translates to: if $z_0$ is a pole of $G_{\mu_g^{\text{Planch}}}(z)$, then $z_0+g$ is one of its zeroes.
We will next prove the following functional equation, which immediately implies this claim:
\begin{equation}\label{eq:CauchyTransEq}
z\cdot G_{\mu_g^{\text{Planch}}}(z)-1 = G_{\mu_g^{\text{Planch}}}(z) G_{\mu_g^{\text{Planch}}}(z-g).
\end{equation}
By virtue of~\eqref{eq:Moment}, both sides of
\eqref{eq:CauchyTransEq} can be treated as formal power series in
$z^{-1}$, and it is enough to prove that their coefficients are the
same. \cref{cor:PlancherelMoments} implies that both sides of
\eqref{eq:CauchyTransEq} are $O(z^{-2})$, and for $\ell \geq 2$ the coefficient of
$z^{-\ell}$ of the LHS is equal to
\[ M^{\text{Planch}}_\ell(g) := \sum_{\Gamma\in\mathbf{M}_0(\ell)} 
\prod_{i\ge 1} (i\cdot g)^{|\SSS_{\tiny\rightarrow}^i(\Gamma)|}.\]
Every Motzkin path $\Gamma $ can be uniquely
decomposed as the concatenation of two Motzkin paths $\Gamma =
((0,0),\Gamma_1,(k,0),\Gamma_2)$,with $0<k\le \ell$ and 
$|\SSS^0\big(((0,0),\Gamma_1,(k,0))\big)|=1$. This gives
\[ M^{\text{Planch}}_0(g) = 1,\ \  M^{\text{Planch}}_1(g) = 0,\ \
M^{\text{Planch}}_\ell(g) = \sum_{\ell_1+\ell_2 =
\ell-2}\widetilde{M}^{\text{Planch}}_{\ell_1}(g)\cdot
M^{\text{Planch}}_{\ell_2}(g),\ \forall\,\ell\ge 2,\]
where $\widetilde{M}^{\text{Planch}}_{\ell}(g) := \sum_{\Gamma\in\mathbf{M}_0(\ell)} 
\prod_{i\ge 1} ((i+1)\cdot
g)^{|\SSS_{\tiny\rightarrow}^i(\Gamma)|}$. We can interpret the factor
\[\prod_{i\ge 1} ((i+1)\cdot g)^{|\SSS_{\tiny\rightarrow}^i(\Gamma)|} = 
\sum_{j_1, j_2, \dots \in\Z_{\geq 0}}
\,\prod_{i\ge 1}\binom{|\SSS^i_{\tiny\rightarrow}(\Gamma)|}{j_i}(i\cdot g)^{|\SSS_{\tiny\rightarrow}^i(\Gamma)| - j_i}\cdot g^{j_i}\]
as a sum over all subsets $S\subseteq \SSS_{\tiny\rightarrow}(\Gamma)$ of horizontal steps
of $\Gamma$ and associating a standard weight $\prod_{i\ge 1} (i\cdot g)^{|\SSS_{\tiny\rightarrow}^i(\widetilde{\Gamma})|}$ multiplied by $g^{|S|}$ to a Motzkin path
$\widetilde{\Gamma}$ obtained from $\Gamma$ by removing the horizontal
steps from the chosen $S$. This interpretation allows to rewrite
$\widetilde{M}^{\text{Planch}}_{\ell}(g)$:

\[ \widetilde{M}^{\text{Planch}}_{\ell}(g) = 
\sum_{k =0}^\ell\binom{\ell}{k}g^k\cdot M^{\text{Planch}}_{\ell-k}(g) = 
\sum_{k=0}^{\infty}\binom{-\ell+k-1}{k}(-g)^k\cdot M^{\text{Planch}}_{\ell-k}(g),\]
which gives
\begin{align*}
M^{\text{Planch}}_\ell(g) &= \sum_{\ell_1,\ell_2,k\in\Z_{\ge 0}\colon\atop \ell_1+\ell_2+k = \ell-2}
{\!\binom{-\ell_1-1}{k}(-g)^k\cdot M^{\text{Planch}}_{\ell_1}(g) M^{\text{Planch}}_{\ell_2}(g)}\\
&= [z^{-\ell}]G_{\mu_g^{\text{Planch}}}(z) G_{\mu_g^{\text{Planch}}}(z-g).
\end{align*}
This proves the desired~\eqref{eq:CauchyTransEq}, and therefore we
conclude the desired statement: if $\omega_{\Lambda_{\infty}}$ has a
local minimum at $z_0$, then it has a local maximum at $z_0+g$.
This means, in the notation of \cref{eq:limit_rows} and \cref{eq:LimProfileShape'}, that $\tilde{\la}_i\ne\tilde{\la}_{i+1}$, for all $i$, and therefore the local minima of $\omega_{\Lambda_{\infty}}$ are $\tilde{\la}_1>\tilde{\la}_2+g>\tilde{\la}_3+2g>\cdots$.
On the other hand, \cref{theo:BesselZeroes} together with \cref{prop_duality} imply that
\[
\supp(\mu_g^{\text{Planch}}) = \{-l^{(-g)}_1-g > -l^{(-g)}_2-g > \cdots \},
\]
and by the framework of \cref{sec:transition_measure}, it follows that this is also the set of local minima of $\omega_{\Lambda_{\infty}}$.
Hence, $\tilde{\la}_i+(i-1)g=-l^{(-g)}_i-g$, for all $i\ge 1$, so plugging back into \eqref{eq:LimProfileShape'} gives
\[
\omega_{\Lambda_\infty} (x) = 
\begin{cases}
x, &\text{ if } x > -l_1^{(-g)}-g,\\
-2(l_i^{(-g)}+ig)-x, &\text{ if } x\in\big[-l_i^{(-g)},\ -l_i^{(-g)}-g\big],\text{ for some }i\ge 1,\\
x-2ig, &\text{ if } x\in\big[-l_{i+1}^{(-g)}-g,\ -l_{i}^{(-g)}\big],\text{ for some }i\ge 1.
\end{cases}
\]

Finally, let $i\ge 1$ be arbitrary.
Take any $-g>\eps>0$, and assume that $\eps\ge\|\omega_{\Lambda^\a_d} - \omega_{\Lambda_\infty}\|_\infty$.
Using the same arguments as in the proof of \cref{prop:LongestRow}, one can show that there exists a constant $c>0$ (independent of $d\ge 1$) such that
\[
c\eps\ge \left|\left(\la^{(d)}_i\cdot\frac{\sqrt{\al}}{\sqrt{d}} + \frac{(i-1)}{\sqrt{\al d}}\right) -  \left(-l_i^{(-g)}-(2i-1)g\right)\right|\sim \left|\frac{\la^{(d)}_i}{-g\cdot d} -  \left(-l_i^{(-g)}-ig\right)\right|.
\]
Hence,
\[
\PP^\a_{\chi_d}( \|\omega_{\Lambda^\a_d} - \omega_{\Lambda_\infty}\|_\infty \le \eps)
\le\PP^\a_{\chi_d}\left(\left|\frac{\la^{(d)}_i}{-g\cdot d} -  \left(-l_i^{(-g)}-ig\right)\right|\le c\eps\right),
\]
for all large $d$.
\cref{theo:LLNDS} shows $\lim_{d\to\infty}\PP^\a_{\chi_d}(\|\omega_{\Lambda^\a_d} - \omega_{\Lambda_\infty}\|_\infty \le
\eps)=1$, therefore
\[
\lim_{d\to\infty}{\PP^\a_{\chi_d}\left(\left|\frac{\la^{(d)}_i}{-g\cdot d} -  \left(-l_i^{(-g)}-ig\right)\right|\le c\eps\right)} = 1,
\]
as desired.
\end{proof}

\section*{Acknowledgement}

CC was partially supported by the NSF grant DMS-2348139 and by the Simons Foundation's Travel Support for Mathematicians grant MP-TSM-00006777.
CC would like to thank the hospitality of the Galileo Galilei Institute for Theoretical Physics, where part of this research was done, and to the participants of the workshop Randomness, Integrability and Universality, for stimulating discussions. CC is also grateful to the Institute of Mathematics of the Polish Academy of Sciences for hosting him during his research visit in June of 2022.
MD would like to thank to Valentin F\'eray and Piotr Śniady for many
years of collaboration on strictly related problems and countless
interesting discussions. He would also like to thank to Philippe Di
Francesco, Pierre-L\"oic M\'eliot, and Ryszard Szwarc for their valuable comments, and
to Mateusz Kwaśnicki for letting us know about the
reference~\cite{StampachStovicek2015}. AM would like to thank the
faculty and staff at UMass Boston for their support during the early
stages of this project, especially Eric Grinberg, Eduardo González,
Alfred No\"el, Maureen Title, and Mirjana Vuleti\'{c}. We also thank the referees for their valuable comments.

\bibliographystyle{amsalpha}

\bibliography{biblio2015}

\def\cprime{$'$}
\providecommand{\bysame}{\leavevmode\hbox to3em{\hrulefill}\thinspace}
\providecommand{\MR}{\relax\ifhmode\unskip\space\fi MR }
\providecommand{\MRhref}[2]{%
  \href{http://www.ams.org/mathscinet-getitem?mr=#1}{#2}
}
\providecommand{\href}[2]{#2}
\begin{thebibliography}{BGCG22}

\bibitem[AB19]{AkemannByun2019}
Gernot Akemann and Sung-Soo Byun, \emph{The high temperature crossover for general 2{D} {C}oulomb gases}, J. Stat. Phys. \textbf{175} (2019), no.~6, 1043--1065. \MR{3962973}

\bibitem[ABG12]{AllezBouchaudGuionnet2012}
Romain Allez, Jean-Philippe Bouchaud, and Alice Guionnet, \emph{Invariant beta ensembles and the gauss-wigner crossover}, Phys. Rev. Lett. \textbf{109} (2012), 094102.

\bibitem[BAG97]{BenGuionnet1997}
G.~Ben~Arous and A.~Guionnet, \emph{Large deviations for {W}igner's law and {V}oiculescu's non-commutative entropy}, Probab. Theory Related Fields \textbf{108} (1997), no.~4, 517--542. \MR{1465640}

\bibitem[BC14]{BorodinCorwin2014}
A.~Borodin and I.~Corwin, \emph{{M}acdonald processes}, Probab. Theory Related Fields \textbf{158} (2014), no.~1-2, 225--400.

\bibitem[BD22]{BenDali2022a}
Houcine Ben~Dali, \emph{Generating series of non-oriented constellations and marginal sums in the {M}atching-{J}ack conjecture}, Algebr. Comb. \textbf{5} (2022), no.~6, 1299--1336.

\bibitem[BD23]{BenDali2023}
\bysame, \emph{Integrality in the {M}atching-{J}ack conjecture and the {F}arahat-{H}igman algebra}, Trans. Amer. Math. Soc. \textbf{376} (2023), no.~5, 3641--3662. \MR{4577343}

\bibitem[BDD23]{BenDaliDolega2023}
Houcine Ben~Dali and Maciej Do{\l}\k{e}ga, \emph{{Positive formula for Jack polynomials, Jack characters and proof of Lassalle's conjecture}}, Preprint arXiv:22305.07966, 2023.

\bibitem[BDJ99]{BaikDeiftJohansson1999}
J.~Baik, P.~Deift, and K.~Johansson, \emph{{On the distribution of the length of the longest increasing subsequence of random permutations}}, J. Amer. Math. Soc. \textbf{12} (1999), no.~4, 1119--1178. \MR{2000e:05006}

\bibitem[BDS16]{BaikDeiftSuidan2016}
Jinho Baik, Percy Deift, and Toufic Suidan, \emph{Combinatorics and random matrix theory}, Graduate Studies in Mathematics, vol. 172, American Mathematical Society, Providence, RI, 2016. \MR{3468920}

\bibitem[BGCG22]{Benaych-GeorgesCuencaGorin2022}
Florent Benaych-Georges, Cesar Cuenca, and Vadim Gorin, \emph{Matrix addition and the {D}unkl transform at high temperature}, Comm. Math. Phys. \textbf{394} (2022), no.~2, 735--795. \MR{4469406}

\bibitem[BGG17]{BorodinGorinGuionnet2017}
A.~Borodin, V.~Gorin, and A.~Guionnet, \emph{{Gaussian asymptotics of discrete $\beta$-ensembles}}, Publ. Math. Inst. Hautes \'Etudes Sci. \textbf{125} (2017), 1--78. \MR{3668648}

\bibitem[BGP15]{Benaych-GeorgesPeche2015}
Florent Benaych-Georges and Sandrine P\'{e}ch\'{e}, \emph{Poisson statistics for matrix ensembles at large temperature}, J. Stat. Phys. \textbf{161} (2015), no.~3, 633--656. \MR{3406702}

\bibitem[Bia98]{Biane1998}
Philippe Biane, \emph{{Representations of symmetric groups and free probability}}, Adv. Math. \textbf{138} (1998), no.~1, 126--181. \MR{MR1644993 (2001b:05225)}

\bibitem[Bia01]{Biane2001}
\bysame, \emph{{Approximate factorization and concentration for characters of symmetric groups}}, Internat. Math. Res. Notices (2001), no.~4, 179--192. \MR{MR1813797 (2002a:20017)}

\bibitem[Bil95]{Billingsley1995}
Patrick Billingsley, \emph{Probability and measure}, third ed., Wiley Series in Probability and Mathematical Statistics, John Wiley \& Sons, Inc., New York, 1995, A Wiley-Interscience Publication. \MR{1324786}

\bibitem[BO05]{BorodinOlshanski2005}
Alexei Borodin and Grigori Olshanski, \emph{{$Z$}-measures on partitions and their scaling limits}, European J. Combin. \textbf{26} (2005), no.~6, 795--834. \MR{2143199}

\bibitem[BO17]{BorodinOlshanski2017}
\bysame, \emph{Representations of the infinite symmetric group}, Cambridge Studies in Advanced Mathematics, vol. 160, Cambridge University Press, Cambridge, 2017. \MR{3618143}

\bibitem[BOO00]{BorodinOkounkovOlshanski2000}
Alexei Borodin, Andrei Okounkov, and Grigori Olshanski, \emph{{Asymptotics of Plancherel measures for symmetric groups}}, Journal of the American Mathematical Society \textbf{13} (2000), no.~3, 481--515.

\bibitem[CD22]{ChapuyDolega2022}
Guillaume Chapuy and Maciej Do{\l}{\k{e}}ga, \emph{Non-orientable branched coverings, {$b$}-{H}urwitz numbers, and positivity for multiparametric {J}ack expansions}, Adv. Math. \textbf{409} (2022), no.~part A, Paper No. 108645, 72. \MR{4477016}

\bibitem[DE02]{DumitriuEdelman2002}
Ioana Dumitriu and Alan Edelman, \emph{Matrix models for beta ensembles}, J. Math. Phys. \textbf{43} (2002), no.~11, 5830--5847. \MR{1936554}

\bibitem[DF16]{DolegaFeray2016}
Maciej Do{\l}{\k{e}}ga and Valentin F\'eray, \emph{Gaussian fluctuations of {Y}oung diagrams and structure constants of {J}ack characters}, Duke Math. J. \textbf{165} (2016), no.~7, 1193--1282. \MR{3498866}

\bibitem[DF{\'S}10]{DolegaFeraySniady2010}
Maciej Do{\l}{\k{e}}ga, Valentin F{\'e}ray, and Piotr {\'S}niady, \emph{{Explicit combinatorial interpretation of {K}erov character polynomials as numbers of permutation factorizations}}, Adv. Math. \textbf{225} (2010), no.~1, 81--120. \MR{2669350 (2011i:05267)}

\bibitem[DFS14]{DolegaFeraySniady2014}
Maciej Do{\l}{\k{e}}ga, Valentin F\'eray, and Piotr \'Sniady, \emph{Jack polynomials and orientability generating series of maps}, S\'em. Lothar. Combin. \textbf{70} (2014), Art. B70j, 50. \MR{3378809}

\bibitem[DK19]{DimitrovKnizel2019}
Evgeni Dimitrov and Alisa Knizel, \emph{Log-gases on quadratic lattices via discrete loop equations and {$q$}-boxed plane partitions}, J. Funct. Anal. \textbf{276} (2019), no.~10, 3067--3169. \MR{3944289}

\bibitem[Do{\l}17]{Dolega2017a}
Maciej Do{\l}{\k e}ga, \emph{Top degree part in {$b$}-conjecture for unicellular bipartite maps}, Electron. J. Combin. \textbf{24} (2017), no.~3, Paper No. 3.24, 39. \MR{3691541}

\bibitem[DS15]{DuyShirai2015}
Trinh~Khanh Duy and Tomoyuki Shirai, \emph{The mean spectral measures of random {J}acobi matrices related to {G}aussian beta ensembles}, Electron. Commun. Probab. \textbf{20} (2015), no. 68, 13. \MR{3407212}

\bibitem[DS19]{DolegaSniady2019}
Maciej Do{\l}{\k e}ga and Piotr \'{S}niady, \emph{Gaussian fluctuations of {J}ack-deformed random {Y}oung diagrams}, Probab. Theory Related Fields \textbf{174} (2019), no.~1-2, 133--176. \MR{3947322}

\bibitem[F{\'e}r09]{Feray2009}
Valentin F{\'e}ray, \emph{{Combinatorial interpretation and positivity of {K}erov's character polynomials}}, J. Algebraic Combin. \textbf{29} (2009), no.~4, 473--507. \MR{MR2506718}

\bibitem[For10]{Forrester2010}
P.~J. Forrester, \emph{Log-gases and random matrices}, London Mathematical Society Monographs Series, vol.~34, Princeton University Press, Princeton, NJ, 2010. \MR{2641363}

\bibitem[For22]{Forrester2022}
Peter~J. Forrester, \emph{High-low temperature dualities for the classical {$\beta$}-ensembles}, Random Matrices Theory Appl. \textbf{11} (2022), no.~4, Paper No. 2250035, 25. \MR{4480830}

\bibitem[FS09]{FlajoletSedgewick2009}
Philippe Flajolet and Robert Sedgewick, \emph{Analytic combinatorics}, Cambridge University Press, Cambridge, 2009. \MR{2483235}

\bibitem[GH19]{GuionnetHuang2019}
A.~Guionnet and J.~Huang, \emph{{Rigidity and Edge Universality of Discrete $\beta$-Ensembles}}, Communications on Pure and Applied Mathematics \textbf{72} (2019), no.~9, 1875--1982.

\bibitem[GZ73]{GradZakrajsek1973}
J.~Grad and E.~Zakraj{\v s}ek, \emph{Method for evaluation of zeros of {B}essel functions}, J. Inst. Math. Appl. \textbf{11} (1973), 57--72. \MR{331725}

\bibitem[HL21]{HardyLambert2021}
Adrien Hardy and Gaultier Lambert, \emph{C{LT} for circular beta-ensembles at high temperature}, J. Funct. Anal. \textbf{280} (2021), no.~7, Paper No. 108869, 40. \MR{4211032}

\bibitem[Hua21]{Huang2021}
Jiaoyang Huang, \emph{Law of large numbers and central limit theorems through {J}ack generating functions}, Adv. Math. \textbf{380} (2021), Paper No. 107545, 91. \MR{4200465}

\bibitem[IKF91]{IkebeKikuchiFujishiro1991}
Yasuhiko Ikebe, Yasushi Kikuchi, and Issei Fujishiro, \emph{Computing zeros and orders of {B}essel functions}, Proceedings of the {I}nternational {S}ymposium on {C}omputational {M}athematics ({M}atsuyama, 1990), vol.~38, 1991, pp.~169--184. \MR{1146980}

\bibitem[IO02]{IvanovOlshanski2002}
Vladimir Ivanov and Grigori Olshanski, \emph{{Kerov's central limit theorem for the {P}lancherel measure on {Y}oung diagrams}}, {Symmetric functions 2001: surveys of developments and perspectives}, {NATO Sci. Ser. II Math. Phys. Chem.}, vol.~74, Kluwer Acad. Publ., Dordrecht, 2002, pp.~93--151. \MR{MR2059361 (2005d:05148)}

\bibitem[Joh98]{Johansson1998}
Kurt Johansson, \emph{{On fluctuations of eigenvalues of random {H}ermitian matrices}}, Duke Math. J. \textbf{91} (1998), no.~1, 151--204. \MR{MR1487983 (2000m:82026)}

\bibitem[Ker93a]{Kerov1993}
Sergei Kerov, \emph{{The asymptotics of interlacing sequences and the growth of continual {Y}oung diagrams}}, Zap. Nauchn. Sem. S.-Peterburg. Otdel. Mat. Inst. Steklov. (POMI) \textbf{205} (1993), no.~Differentsialnaya Geom. Gruppy Li i Mekh. 13, 21--29, 179. \MR{MR1255301}

\bibitem[Ker93b]{Kerov1993transition}
\bysame, \emph{{Transition probabilities of continual {Y}oung diagrams and the {M}arkov moment problem}}, Funct. Anal. Appl. \textbf{27} (1993), no.~3, 104--117.

\bibitem[Ker98]{Kerov1998}
\bysame, \emph{{Interlacing measures}}, {Kirillov's seminar on representation theory}, {Amer. Math. Soc. Transl. Ser. 2}, vol. 181, Amer. Math. Soc., Providence, RI, 1998, pp.~35--83. \MR{MR1618739 (99h:30034)}

\bibitem[Ker00]{Kerov2000}
\bysame, \emph{{Anisotropic Young diagrams and Jack symmetric functions}}, Funct. Anal. Appl. \textbf{34} (2000), 41--51.

\bibitem[KO94]{KerovOlshanski1994}
Sergei Kerov and Grigori Olshanski, \emph{{Polynomial functions on the set of {Y}oung diagrams}}, C. R. Acad. Sci. Paris S{\'e}r. I Math. \textbf{319} (1994), no.~2, 121--126. \MR{MR1288389 (95f:05116)}

\bibitem[KOO98]{KerovOkounkovOlshanski1998}
Sergei Kerov, Andrei Okounkov, and Grigori Olshanski, \emph{The boundary of the {Y}oung graph with {J}ack edge multiplicities}, Internat. Math. Res. Notices (1998), no.~4, 173--199. \MR{1609628}

\bibitem[Las08]{Lassalle2008b}
Michel Lassalle, \emph{{A positivity conjecture for {J}ack polynomials}}, Math. Res. Lett. \textbf{15} (2008), no.~4, 661--681. \MR{2424904}

\bibitem[Las09]{Lassalle2009}
\bysame, \emph{{Jack polynomials and free cumulants}}, Adv. Math. \textbf{222} (2009), no.~6, 2227--2269. \MR{2562783}

\bibitem[LS77]{LoganShepp1977}
B.~F. Logan and L.~A. Shepp, \emph{{A variational problem for random {Y}oung tableaux}}, Advances in Math. \textbf{26} (1977), no.~2, 206--222. \MR{MR1417317 (98e:05108)}

\bibitem[M\'17]{Meliot2017}
Pierre-L\"oic M\'{e}liot, \emph{Representation theory of symmetric groups}, Discrete Mathematics and its Applications (Boca Raton), CRC Press, Boca Raton, FL, 2017. \MR{3616172}

\bibitem[Mac95]{Macdonald1995}
I.~G. Macdonald, \emph{{Symmetric functions and {H}all polynomials}}, second ed., {Oxford Mathematical Monographs}, The Clarendon Press Oxford University Press, New York, 1995, With contributions by A. Zelevinsky, Oxford Science Publications. \MR{1354144}

\bibitem[Mat19]{Matveev2019}
K.~Matveev, \emph{Macdonald-positive specializations of the algebra of symmetric functions: proof of the {K}erov conjecture}, Ann. of Math. (2) \textbf{189} (2019), no.~1, 277--316. \MR{3898175}

\bibitem[Mol23]{Moll2023}
Alexander Moll, \emph{{Gaussian Asymptotics of Jack Measures on Partitions from Weighted Enumeration of Ribbon Paths}}, Int. Math. Res. Not. IMRN (2023), no.~3, 1801--1881.

\bibitem[MP22]{MergnyPotters2022}
Pierre Mergny and Marc Potters, \emph{Rank one hciz at high temperature: interpolating between classical and free convolutions}, SciPost Physics \textbf{12} (2022), no.~1, 022.

\bibitem[Nic95]{Nica1995}
Alexandru Nica, \emph{A one-parameter family of transforms, linearizing convolution laws for probability distributions}, Communications in Mathematical Physics \textbf{168} (1995), 187--207.

\bibitem[NS06]{NicaSpeicherBook}
Alexandru Nica and Roland Speicher, \emph{{Lectures on the combinatorics of free probability}}, {London Mathematical Society Lecture Note Series}, vol. 335, Cambridge University Press, Cambridge, 2006. \MR{2266879 (2008k:46198)}

\bibitem[NS13]{NazarovSklyanin2013}
Maxim Nazarov and Evgeny Sklyanin, \emph{Integrable hierarchy of the quantum {B}enjamin-{O}no equation}, SIGMA Symmetry Integrability Geom. Methods Appl. \textbf{9} (2013), Paper 078, 14. \MR{3141546}

\bibitem[NT18]{NakanoTrink2018}
Fumihiko Nakano and Khanh~Duy Trinh, \emph{Gaussian beta ensembles at high temperature: eigenvalue fluctuations and bulk statistics}, J. Stat. Phys. \textbf{173} (2018), no.~2, 295--321. \MR{3860215}

\bibitem[Oko00]{Okounkov2000}
Andrei Okounkov, \emph{Random matrices and random permutations}, Internat. Math. Res. Notices (2000), no.~20, 1043--1095. \MR{1802530}

\bibitem[Oko01]{Okounkov2001}
\bysame, \emph{{Infinite wedge and random partitions}}, Selecta Math. (N.S.) \textbf{7} (2001), no.~1, 57--81. \MR{MR1856553 (2002f:60019)}

\bibitem[Oko03]{Okounkov2003}
\bysame, \emph{{The uses of random partitions}}, {Fourteenth International Congress on Mathematical Physics}, Word Scientists, 2003, pp.~379--403.

\bibitem[Ora05]{Oravecz2005}
Ferenc Oravecz, \emph{Nica's $q$-convolution is not positivity preserving}, Communications in mathematical physics \textbf{258} (2005), no.~2, 475--478.

\bibitem[Pak20]{Pakzad2020}
Cambyse Pakzad, \emph{Large deviations principle for the largest eigenvalue of the {G}aussian {$\beta$}-ensemble at high temperature}, J. Theoret. Probab. \textbf{33} (2020), no.~1, 428--443. \MR{4064307}

\bibitem[RRV11]{RamirezRiderVirag2011}
Jos\'{e}~A. Ram\'{\i}rez, Brian Rider, and B\'{a}lint Vir\'{a}g, \emph{Beta ensembles, stochastic {A}iry spectrum, and a diffusion}, J. Amer. Math. Soc. \textbf{24} (2011), no.~4, 919--944. \MR{2813333}

\bibitem[\'{S}19]{Sniady2019}
Piotr \'{S}niady, \emph{Asymptotics of {J}ack characters}, J. Combin. Theory Ser. A \textbf{166} (2019), 91--143. \MR{3921039}

\bibitem[{\'S}ni06]{Sniady2006c}
Piotr {\'S}niady, \emph{{Gaussian fluctuations of characters of symmetric groups and of {Y}oung diagrams}}, Probab. Theory Related Fields \textbf{136} (2006), no.~2, 263--297. \MR{MR2240789}

\bibitem[{\v S}{\v S}15]{StampachStovicek2015}
F.~{{\v S}}tampach and P.~{{\v S}}{t}ov{\'\i}{{\v c}}ek, \emph{Special functions and spectrum of {J}acobi matrices}, Linear Algebra Appl. \textbf{464} (2015), 38--61. \MR{3271307}

\bibitem[Sta89]{Stanley1989}
Richard~P. Stanley, \emph{{Some combinatorial properties of {J}ack symmetric functions}}, Adv. Math. \textbf{77} (1989), no.~1, 76--115. \MR{1014073 (90g:05020)}

\bibitem[Tho64]{Thoma1964}
Elmar Thoma, \emph{{Die unzerlegbaren, positiv-definiten {K}lassenfunktionen der abz{\"a}hlbar unendlichen, symmetrischen {G}ruppe}}, Math. Z. \textbf{85} (1964), 40--61. \MR{MR0173169 (30 \#3382)}

\bibitem[Var58]{Varadarajan1958}
V.~S. Varadarajan, \emph{Weak convergence of measures on separable metric spaces}, Sankhy\={a} \textbf{19} (1958), 15--22. \MR{94838}

\bibitem[Ver74]{Vershik1974}
A.~M. Vershik, \emph{A description of invariant measures for actions of certain infinite-dimensional groups}, Dokl. Akad. Nauk SSSR \textbf{218} (1974), 749--752. \MR{372161}

\bibitem[VK77]{VershikKerov1977}
A.~M. Vershik and S.~V. Kerov, \emph{{Asymptotic behavior of the {P}lancherel measure of the symmetric group and the limit form of {Y}oung tableaux}}, Dokl. Akad. Nauk SSSR \textbf{233} (1977), no.~6, 1024--1027. \MR{0480398 (58 \#562)}

\bibitem[VV09]{ValkoVirag2009}
Benedek Valk\'{o} and B\'{a}lint Vir\'{a}g, \emph{Continuum limits of random matrices and the {B}rownian carousel}, Invent. Math. \textbf{177} (2009), no.~3, 463--508. \MR{2534097}

\end{thebibliography}

\end{document}